\documentclass[12pt,leqno,twoside]{amsproc}
\usepackage{amsfonts}
\usepackage{amsmath}
\usepackage{amssymb}
\usepackage{epsf}
\usepackage{graphics}
\usepackage{graphicx}
\usepackage{latexsym}
\usepackage{psfrag}
\usepackage{tikz}
\usepackage{verbatim}
\usepackage{hyperref}
\usepackage{pgfplots}
\usetikzlibrary{positioning}
\usetikzlibrary{arrows}
\usetikzlibrary{decorations.pathreplacing}

\newcommand{\loza}
{--++ (1/2,{sqrt(3)/2})
--++ (-1/2,{sqrt(3)/2})
--++ (-1/2,{-sqrt(3)/2})
--++ (1/2,{-sqrt(3)/2})}

\newcommand{\lozb}
{--++ (1,0)
--++ (-1/2,{sqrt(3)/2})
--++ (-1,0)
--++ (1/2,{-sqrt(3)/2})
[fill=gray!50!white]}

\newcommand{\lozc}
{--++ (1,0)
--++ (1/2,{sqrt(3)/2})
--++ (-1,0)
--++ (-1/2,{-sqrt(3)/2})
[fill=gray!100!white]}

\newtheorem{thm}{Theorem}[section]
\newtheorem{lem}{Lemma}[section]
\newtheorem{cor}[thm]{Corollary}
\newtheorem{definition}{Definition}[section]

\newtheorem{rem}{Remark}[section]
\newtheorem{hyp}{Hypothesis}[section]

\def\g{\gamma}
\def\G{\Gamma}
\def\d{\delta}
\def\e{\epsilon}
\def\l{\lambda}

\def\N{{\mathbb N}}
\def\Z{{\mathbb Z}}
\def\R{{\mathbb R}}
\def\C{{\mathbb C}}
\def\E{{\mathbb E}}
\def\EE{{\mathcal E}}
\def\LL{{\mathcal L}}

\def\P{{\mathbb P}}
\def\supp{{\text{Supp}}}

\title{Asymptotic geometry of discrete interlaced patterns: Part I}
\author{Erik Duse and Anthony Metcalfe}
\date{\today}

\begin{document}

\begin{abstract}
A discrete Gelfand-Tsetlin pattern is a configuration of particles in
$\Z^2$. The particles are arranged in a finite number of consecutive
rows, numbered from the bottom. There is one particle on the first row,
two particles on the second row, three particles on the third row, etc,
and particles on adjacent rows satisfy an interlacing constraint.

We consider the uniform probability measure on the set of all discrete
Gelfand-Tsetlin patterns of a fixed size where the particles on the top
row are in deterministic positions. This measure arises naturally as an
equivalent description of the uniform probability measure on the set of
all tilings of certain polygons with lozenges. We prove a determinantal
structure, and calculate the correlation kernel.

We consider the asymptotic behaviour of the system as the size increases
under the assumption that the empirical distribution of the deterministic
particles on the top row converges weakly. We consider the asymptotic
`shape' of such systems. We provide parameterisations of the
asymptotic boundaries and investigate the local geometric properties of
the resulting curves. We show that the boundary can be partitioned into
natural sections which are determined by the behaviour of the roots of
a function related to the correlation kernel. This paper should be
regarded as a companion piece to the paper, \cite{Duse15b},
in which we resolve some of the remaining issues. Both of these papers
serve as background material for the papers, \cite{Duse15c} and
\cite{Duse15d}, in which we examine the edge asymptotic behaviour.
\end{abstract}

\maketitle

\section{Introduction}

\subsection{Random lozenge tilings of the regular hexagon}
\label{secTilings}

In this paper we consider the asymptotic shape of random tilings of
`half-hexagons'. We define the systems in the next section, and note
that they have a determinantal structure in section \ref{sectdsodgtp}.
Our motivation for studying such systems is to consider the local
asymptotic boundary behaviour. Of course, in order to do this, one
must first define and characterise a natural boundary. In section
\ref{secmasomr}, under some natural asymptotic assumptions, we note
that the asymptotic behaviour of the correlation kernel of the
systems is amenable to steepest descent techniques. Essentially,
the correlation kernel is expressed as a double contour integral
(see equation (\ref{eqKnrnunsnvn1})), and steepest descent techniques
suggest that the asymptotic behaviour of the expression is determined
by the behaviour of the roots of a certain analytic function (see
equation (\ref{eqf'0})). In this paper, we define the edge, solely,
by considering the behaviour of the roots of this function. The edge
is a natural boundary on which universal asymptotic behaviour is
expected. This expectation is confirmed, using steepest
descent techniques, in the paper \cite{Duse15c}. In this paper, we also
characterise other natural parts of the boundary. We continue this analysis
in the paper \cite{Duse15b}, where we find previously unknown parts of the
boundary. Finally, we find novel edge asymptotic behaviour in the paper
\cite{Duse15d}.

In this section we introduce the systems of `half-hexagons'
by considering the simpler case of random tilings of a regular hexagon.
The left hand side of figure \ref{figRegHexLoz} depicts a regular hexagon with
sides of length $m \ge 1$. The middle depicts three different types of {\em lozenges}
(polygons with angles $\frac{\pi}3$ and $\frac{2\pi}3$) with sides of length
$1$. We label these as types A, B, and C as shown. A complete covering of the
interior of the hexagon with these lozenges is called a {\em tiling}. An example
tiling, when $m=4$, is given on the right of figure \ref{figRegHexLoz}.

Note, in the example, lozenges of type A are adjacent to the left-most and
right-most corners, lozenges of type B are adjacent to the bottom-left and
top-right corners, and lozenges of type C are adjacent to the top-left and
bottom-right corners. Given such a corner behaviour, consider a particular
lozenge adjacent to a corner. The {\em connected component of this
lozenge} is defined as the area covered by the set of all adjacent lozenges
of the same type. Also, the {\em frozen region} is defined as the union
of these six connected components. The remaining tiles form the so-called
{\em disordered region}, and the boundary between the frozen and disordered
regions is called the {\em frozen boundary}. The frozen boundary of the
example tiling is shown on the left of figure \ref{figArctic}.

Impose the uniform probability measure on the set of all possible tilings
of a regular hexagon with sides of length $m$. Then, it is natural to consider
the asymptotic behaviour of the system as $m \to \infty$. Some interesting
results were obtained in \cite{Coh98}: The probability of observing the above
corner behaviour converges to $1$ as $m \to \infty$. For this reason we define:
\begin{definition}
\label{defTypTil}
A tiling of the regular hexagon is referred to as typical if and only if
lozenges of type A are adjacent to the left-most and right-most corners,
lozenges of type B are adjacent to the bottom-left and top-right corners,
and lozenges of type C are adjacent to the top-left and bottom-right corners.
\end{definition}
Also, rescaling so that the sides of the hexagon are of length $1$ for all
$m$, the frozen boundary of typical random tilings converges to the inscribed
circle of the rescaled hexagon as $m \to \infty$. This asymptotic shape is
called the {\em Arctic circle}, and is shown on the right of figure
\ref{figArctic}. See \cite{Coh98} for more precise statements.

\begin{figure}[t]
\centering
\begin{tikzpicture}[xscale=1/2,yscale=1/2]

\draw (0,0)
--++ (4,0)
--++ (2,{2*sqrt(3)})
--++ (-2,{2*sqrt(3)})
--++ (-4,0)
--++ (-2,{-2*sqrt(3)})
--++ (2,{-2*sqrt(3)});

\draw (2,-.3) node {$m$};
\draw (5.4,{1*sqrt(3)-.2}) node {$m$};
\draw (5.4,{3*sqrt(3)+.2}) node {$m$};
\draw (2,{4*sqrt(3)+.3}) node {$m$};
\draw (-1.4,{3*sqrt(3)+.2}) node {$m$};
\draw (-1.4,{1*sqrt(3)-.2}) node {$m$};

\draw (8.5,{(1.5)*sqrt(3)}) \loza;
\draw (8.5,4.8) node {A};
\draw (10.5,{(1.75)*sqrt(3)}) \lozb;
\draw (10.5,4.8) node {B};
\draw (12.5,{(1.75)*sqrt(3)}) \lozc;
\draw (13.5,4.8) node {C};

\draw (18,0) \lozb;
\draw (19,0) \lozb;
\draw (20,0) \loza;
\draw (20,0) \lozc;
\draw (21,0) \lozc;
\draw (17.5,{sqrt(3)/2}) \lozb;
\draw (18.5,{sqrt(3)/2}) \lozb;
\draw (19.5,{sqrt(3)/2}) \loza;
\draw (20.5,{sqrt(3)/2}) \lozb;
\draw (21.5,{sqrt(3)/2}) \loza;
\draw (21.5,{sqrt(3)/2}) \lozc;
\draw (17,{sqrt(3)}) \lozb;
\draw (18,{sqrt(3)}) \loza;
\draw (18,{sqrt(3)}) \lozc;
\draw (20,{sqrt(3)}) \lozb;
\draw (21,{sqrt(3)}) \loza;
\draw (22,{sqrt(3)}) \lozb;
\draw (23,{sqrt(3)}) \loza;
\draw (16.5,{(1.5)*sqrt(3)}) \loza;
\draw (16.5,{(1.5)*sqrt(3)}) \lozc;
\draw (18.5,{(1.5)*sqrt(3)}) \loza;
\draw (18.5,{(1.5)*sqrt(3)}) \lozc;
\draw (19.5,{(1.5)*sqrt(3)}) \lozc;
\draw (21.5,{(1.5)*sqrt(3)}) \loza;
\draw (21.5,{(1.5)*sqrt(3)}) \lozc;
\draw (23.5,{(1.5)*sqrt(3)}) \loza;
\draw (17,{2*sqrt(3)}) \lozb;
\draw (18,{2*sqrt(3)}) \loza;
\draw (19,{2*sqrt(3)}) \lozb;
\draw (20,{2*sqrt(3)}) \loza;
\draw (20,{2*sqrt(3)}) \lozc;
\draw (22,{2*sqrt(3)}) \loza;
\draw (22,{2*sqrt(3)}) \lozc;
\draw (16.5,{(2.5)*sqrt(3)}) \lozc;
\draw (18.5,{(2.5)*sqrt(3)}) \loza;
\draw (18.5,{(2.5)*sqrt(3)}) \lozc;
\draw (20.5,{(2.5)*sqrt(3)}) \loza;
\draw (20.5,{(2.5)*sqrt(3)}) \lozc;
\draw (22.5,{(2.5)*sqrt(3)}) \lozb;
\draw (17,{3*sqrt(3)}) \lozc;
\draw (19,{3*sqrt(3)}) \lozb;
\draw (20,{3*sqrt(3)}) \loza;
\draw (21,{3*sqrt(3)}) \lozb;
\draw (22,{3*sqrt(3)}) \lozb;
\draw (17.5,{(3.5)*sqrt(3)}) \lozc;
\draw (18.5,{(3.5)*sqrt(3)}) \lozc;
\draw (20.5,{(3.5)*sqrt(3)}) \lozb;
\draw (21.5,{(3.5)*sqrt(3)}) \lozb;

\draw (20,-.4) node {$4$};
\draw (23.3,{1*sqrt(3)-.2}) node {$4$};
\draw (23.3,{3*sqrt(3)+.2}) node {$4$};
\draw (20,{4*sqrt(3)+.4}) node {$4$};
\draw (16.7,{3*sqrt(3)+.2}) node {$4$};
\draw (16.7,{1*sqrt(3)-.2}) node {$4$};

\end{tikzpicture}
\caption{Left: A regular hexagon with sides of length $m \ge 1$.
\newline
Middle: Three types of lozenges with sides of length $1$.
\newline
Right: An example tiling when $m=4$.}
\label{figRegHexLoz}
\end{figure}
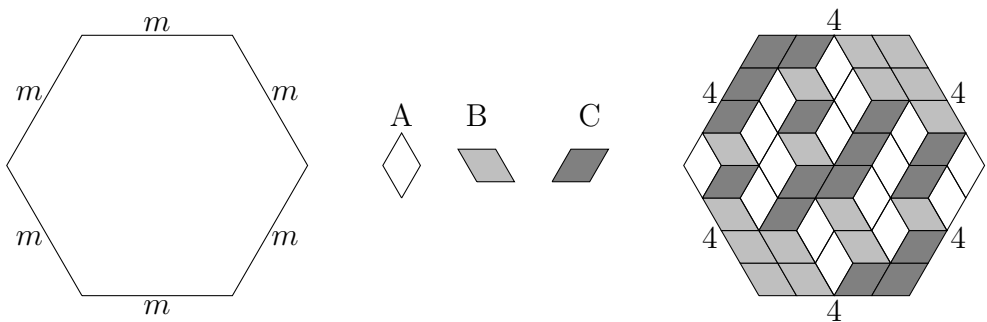

\begin{figure}[t]
\centering
\begin{tikzpicture}[xscale=1/2,yscale=1/2]

\draw (0,0) \lozb;
\draw (1,0) \lozb;
\draw (2,0) \loza;
\draw (2,0) \lozc;
\draw (3,0) \lozc;
\draw (-.5,{sqrt(3)/2}) \lozb;
\draw (.5,{sqrt(3)/2}) \lozb;
\draw (1.5,{sqrt(3)/2}) \loza;
\draw (2.5,{sqrt(3)/2}) \lozb;
\draw (3.5,{sqrt(3)/2}) \loza;
\draw (3.5,{sqrt(3)/2}) \lozc;
\draw (-1,{sqrt(3)}) \lozb;
\draw (0,{sqrt(3)}) \loza;
\draw (0,{sqrt(3)}) \lozc;
\draw (2,{sqrt(3)}) \lozb;
\draw (3,{sqrt(3)}) \loza;
\draw (4,{sqrt(3)}) \lozb;
\draw (5,{sqrt(3)}) \loza;
\draw (-1.5,{(1.5)*sqrt(3)}) \loza;
\draw (-1.5,{(1.5)*sqrt(3)}) \lozc;
\draw (.5,{(1.5)*sqrt(3)}) \loza;
\draw (.5,{(1.5)*sqrt(3)}) \lozc;
\draw (1.5,{(1.5)*sqrt(3)}) \lozc;
\draw (3.5,{(1.5)*sqrt(3)}) \loza;
\draw (3.5,{(1.5)*sqrt(3)}) \lozc;
\draw (5.5,{(1.5)*sqrt(3)}) \loza;
\draw (-1,{2*sqrt(3)}) \lozb;
\draw (0,{2*sqrt(3)}) \loza;
\draw (1,{2*sqrt(3)}) \lozb;
\draw (2,{2*sqrt(3)}) \loza;
\draw (2,{2*sqrt(3)}) \lozc;
\draw (4,{2*sqrt(3)}) \loza;
\draw (4,{2*sqrt(3)}) \lozc;
\draw (-1.5,{(2.5)*sqrt(3)}) \lozc;
\draw (.5,{(2.5)*sqrt(3)}) \loza;
\draw (.5,{(2.5)*sqrt(3)}) \lozc;
\draw (2.5,{(2.5)*sqrt(3)}) \loza;
\draw (2.5,{(2.5)*sqrt(3)}) \lozc;
\draw (4.5,{(2.5)*sqrt(3)}) \lozb;
\draw (-1,{3*sqrt(3)}) \lozc;
\draw (1,{3*sqrt(3)}) \lozb;
\draw (2,{3*sqrt(3)}) \loza;
\draw (3,{3*sqrt(3)}) \lozb;
\draw (4,{3*sqrt(3)}) \lozb;
\draw (-.5,{(3.5)*sqrt(3)}) \lozc;
\draw (.5,{(3.5)*sqrt(3)}) \lozc;
\draw (2.5,{(3.5)*sqrt(3)}) \lozb;
\draw (3.5,{(3.5)*sqrt(3)}) \lozb;

\draw [ultra thick] (2,0)
--++ (1/2,{sqrt(3)/2})
--++ (1,0)
--++ (1/2,{sqrt(3)/2})
--++ (1,0)
--++ (-1/2,{sqrt(3)/2})
--++ (1,{sqrt(3)})
--++ (-1,0)
--++ (-1/2,{sqrt(3)/2})
--++ (-1,0)
--++ (-1,{sqrt(3)})
--++ (-1/2,{-sqrt(3)/2})
--++ (-1,0)
--++ (-1,{-sqrt(3)})
--++ (-1,0)
--++ (1/2,{-sqrt(3)/2})
--++ (-1/2,{-sqrt(3)/2})
--++ (1,0)
--++ (1/2,{-sqrt(3)/2})
--++ (1,0)
--++ (1,{-sqrt(3)});

\draw (.25,.4) node {B};
\draw (3.75,.4) node {C};
\draw (5.55,{2*sqrt(3)+.1}) node {A};
\draw (3.75,{4*sqrt(3)-.4}) node {B};
\draw (.25,{4*sqrt(3)-.4}) node {C};
\draw (-1.55,{2*sqrt(3)+.1}) node {A};

\draw (18,0)
--++ (4,0)
--++ (2,{2*sqrt(3)})
--++ (-2,{2*sqrt(3)})
--++ (-4,0)
--++ (-2,{-2*sqrt(3)})
--++ (2,{-2*sqrt(3)});
\draw (20,{2*sqrt(3)}) circle ({2*sqrt(3)});

\draw (20,-.4) node {$1$};
\draw (23.3,{1*sqrt(3)-.2}) node {$1$};
\draw (23.3,{3*sqrt(3)+.2}) node {$1$};
\draw (20,{4*sqrt(3)+.4}) node {$1$};
\draw (16.7,{3*sqrt(3)+.2}) node {$1$};
\draw (16.7,{1*sqrt(3)-.2}) node {$1$};

\end{tikzpicture}
\caption{Left: The frozen boundary of the example tiling of figure \ref{figRegHexLoz}.
\newline
Right: The Arctic circle, i.e., the asymptotic shape of the frozen boundary
of a `typical random tiling' as $m \to \infty$.}
\label{figArctic}
\end{figure}
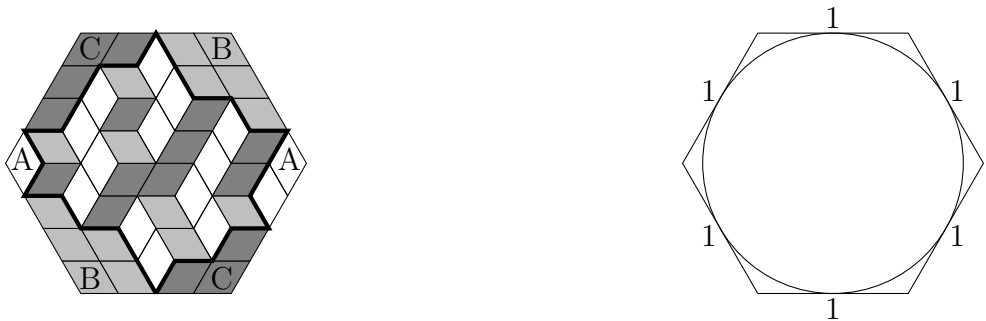

The goal of our work was to study the fluctuations of the frozen boundary
around the asymptotic limit. We wanted to show that universal edge asymptotic
behaviour holds. More exactly, we wanted to show that, when appropriately
rescaled, the fluctuations of the frozen boundary converge to the {\em Airy
process}. More generally, we wished to study the analogous question when
tiling certain `half-hexagons', which we define in the next section.
Convergence to the $1$-dimensional Airy process, in the case of the
regular hexagon, was established by Baik {\em et al}, \cite{Baik07}.
The analogous question for the frozen boundary of the Aztec diamond was
settled by Johansson, \cite{Joh05a}. Convergence to the $2$-dimensional
Airy process, in the case of the regular hexagon, was recently and
independently established by Petrov, \cite{Pet14}. The Airy process has
also been observed asymptotically at the edge of the spectra of various
ensembles of random matrices (see, for example, Mehta, \cite{Mehta04}).
It was first observed in a study by Pr{\"a}hofer and Spohn, \cite{Pra02},
of certain random growth models. See also Johansson, \cite{Joh03}, which
considers a different random growth model. It has also appeared
asymptotically in other, seeming unrelated, systems.

The key to the asymptotic analysis in this and the accompanying papers is
the investigation of a saddle-point function involving a known probability
measure. This type of saddle point problem occurs in several contexts. A
recent work which is closely related to ours is that of Hachem {\em et al.},
\cite{Hachem15}, which studies the asymptotic behaviour of the edge of the
spectrum of large complex correlated Wishart matrices.

\subsection{Random lozenge tilings of the `half-hexagon' and equivalent interlaced
particle configurations}
\label{secTilingsHalf}

We again begin by considering tilings of the regular hexagon. For
simplicity we now restrict to the set of all typical tilings,
defined in definition \ref{defTypTil}. It is not hard to see that
any tiling is uniquely determined by the locations of the
lozenges of type A, i.e., the {\em vertical lozenges}. Then, placing
particles in the center of each vertical lozenge, we see that the uniform
measure on the set of all typical tilings is equivalent to the uniform
measure on the set of all configurations of particles which satisfy:
\begin{itemize}
\item
The particles are in the interior of the hexagon.
\item
The particles lie on $2m-1$ rows, which we label from the bottom to the top.
\item
Adjacent rows are a distance of $\frac{\sqrt{3}}2$ apart.
\item
There are $r$ distinct particles on row $r$ when $r \in \{1,\ldots,m\}$.
\item
There are $2m-r$ distinct particles on row $r$ when $r \in \{m,\ldots,2m-1\}$.
\item
Particles are in integer positions on the even rows.
\item
Particles are in half-integer positions (i.e., $\Z + \frac12$) on the odd rows.
\item
Particles on adjacent rows {\em interlace}: When $r \in \{1,\ldots,m-1\}$, there
is exactly one particle on row $r$ `between' each neighbouring pair of particles
on row $r+1$. Also, when $r \in \{m+1,\ldots,2m-1\}$, there is exactly one
particle on row $r$ `between' each neighbouring pair of particles on row $r-1$.
\end{itemize}
The particle configuration which is equivalent to the given example tiling is
shown on the left of figure \ref{figEquivIntPart}.

\begin{figure}[t]
\centering
\begin{tikzpicture}[xscale=1/2,yscale=1/2]

\draw (0,0) \lozb;
\draw (1,0) \lozb;
\draw (2,0) \loza;
\draw (2,0) \lozc;
\draw (3,0) \lozc;
\draw (-.5,{sqrt(3)/2}) \lozb;
\draw (.5,{sqrt(3)/2}) \lozb;
\draw (1.5,{sqrt(3)/2}) \loza;
\draw (2.5,{sqrt(3)/2}) \lozb;
\draw (3.5,{sqrt(3)/2}) \loza;
\draw (3.5,{sqrt(3)/2}) \lozc;
\draw (-1,{sqrt(3)}) \lozb;
\draw (0,{sqrt(3)}) \loza;
\draw (0,{sqrt(3)}) \lozc;
\draw (2,{sqrt(3)}) \lozb;
\draw (3,{sqrt(3)}) \loza;
\draw (4,{sqrt(3)}) \lozb;
\draw (5,{sqrt(3)}) \loza;
\draw (-1.5,{(1.5)*sqrt(3)}) \loza;
\draw (-1.5,{(1.5)*sqrt(3)}) \lozc;
\draw (.5,{(1.5)*sqrt(3)}) \loza;
\draw (.5,{(1.5)*sqrt(3)}) \lozc;
\draw (1.5,{(1.5)*sqrt(3)}) \lozc;
\draw (3.5,{(1.5)*sqrt(3)}) \loza;
\draw (3.5,{(1.5)*sqrt(3)}) \lozc;
\draw (5.5,{(1.5)*sqrt(3)}) \loza;
\draw (-1,{2*sqrt(3)}) \lozb;
\draw (0,{2*sqrt(3)}) \loza;
\draw (1,{2*sqrt(3)}) \lozb;
\draw (2,{2*sqrt(3)}) \loza;
\draw (2,{2*sqrt(3)}) \lozc;
\draw (4,{2*sqrt(3)}) \loza;
\draw (4,{2*sqrt(3)}) \lozc;
\draw (-1.5,{(2.5)*sqrt(3)}) \lozc;
\draw (.5,{(2.5)*sqrt(3)}) \loza;
\draw (.5,{(2.5)*sqrt(3)}) \lozc;
\draw (2.5,{(2.5)*sqrt(3)}) \loza;
\draw (2.5,{(2.5)*sqrt(3)}) \lozc;
\draw (4.5,{(2.5)*sqrt(3)}) \lozb;
\draw (-1,{3*sqrt(3)}) \lozc;
\draw (1,{3*sqrt(3)}) \lozb;
\draw (2,{3*sqrt(3)}) \loza;
\draw (3,{3*sqrt(3)}) \lozb;
\draw (4,{3*sqrt(3)}) \lozb;
\draw (-.5,{(3.5)*sqrt(3)}) \lozc;
\draw (.5,{(3.5)*sqrt(3)}) \lozc;
\draw (2.5,{(3.5)*sqrt(3)}) \lozb;
\draw (3.5,{(3.5)*sqrt(3)}) \lozb;

\filldraw (2,{(.5)*sqrt(3)}) circle (3pt);
\filldraw (1.5,{sqrt(3)}) circle (3pt);
\filldraw (3.5,{sqrt(3)}) circle (3pt);
\filldraw (0,{(1.5)*sqrt(3)}) circle (3pt);
\filldraw (3,{(1.5)*sqrt(3)}) circle (3pt);
\filldraw (5,{(1.5)*sqrt(3)}) circle (3pt);
\filldraw (-1.5,{2*sqrt(3)}) circle (3pt);
\filldraw (.5,{2*sqrt(3)}) circle (3pt);
\filldraw (3.5,{2*sqrt(3)}) circle (3pt);
\filldraw (5.5,{2*sqrt(3)}) circle (3pt);
\filldraw (0,{(2.5)*sqrt(3)}) circle (3pt);
\filldraw (2,{(2.5)*sqrt(3)}) circle (3pt);
\filldraw (4,{(2.5)*sqrt(3)}) circle (3pt);
\filldraw (0.5,{3*sqrt(3)}) circle (3pt);
\filldraw (2.5,{3*sqrt(3)}) circle (3pt);
\filldraw (2,{(3.5)*sqrt(3)}) circle (3pt);

\draw [dotted] (5,{(0.5)*sqrt(3)}) --++ (3.8,0);
\draw (10,{(0.55)*sqrt(3)}) node {row $1$};
\draw [dotted] (11.2,{(0.5)*sqrt(3)}) --++ (3.8,0);
\draw [dotted] (5.5,{sqrt(3)}) --++ (3.3,0);
\draw (10,{(1.05)*sqrt(3)}) node {row $2$};
\draw [dotted] (11.2,{sqrt(3)}) --++ (3.3,0);
\draw [dotted] (6,{(1.5)*sqrt(3)}) --++ (2.8,0);
\draw (10,{(1.55)*sqrt(3)}) node {row $3$};
\draw [dotted] (11.2,{(1.5)*sqrt(3)}) --++ (2.8,0);
\draw [dotted] (6.5,{2*sqrt(3)}) --++ (2.3,0);
\draw (10,{(2.05)*sqrt(3)}) node {row $4$};
\draw [dotted] (11.2,{2*sqrt(3)}) --++ (2.3,0);
\draw [dotted] (6,{(2.5)*sqrt(3)}) --++ (2.3,0);
\draw (9.5,{(2.55)*sqrt(3)}) node {row $5$};
\draw [dotted] (10.7,{(2.5)*sqrt(3)}) --++ (2.3,0);
\draw [dotted] (5.5,{3*sqrt(3)}) --++ (2.3,0);
\draw (9,{(3.05)*sqrt(3)}) node {row $6$};
\draw [dotted] (10.2,{3*sqrt(3)}) --++ (2.3,0);
\draw [dotted] (5,{(3.5)*sqrt(3)}) --++ (2.3,0);
\draw (8.5,{(3.55)*sqrt(3)}) node {row $7$};
\draw [dotted] (9.7,{(3.5)*sqrt(3)}) --++ (2.3,0);
\draw [dotted] (4.5,{4*sqrt(3)}) --++ (2.3,0);
\draw (8,{(4.05)*sqrt(3)}) node {row $8$};
\draw [dotted] (9.2,{4*sqrt(3)}) --++ (2.3,0);

\draw (16,0) \lozb;
\draw (17,0) \lozb;
\draw (18,0) \loza;
\draw (18,0) \lozc;
\draw (19,0) \lozc;
\draw (15.5,{sqrt(3)/2}) \lozb;
\draw (16.5,{sqrt(3)/2}) \lozb;
\draw (17.5,{sqrt(3)/2}) \loza;
\draw (18.5,{sqrt(3)/2}) \lozb;
\draw (19.5,{sqrt(3)/2}) \loza;
\draw (19.5,{sqrt(3)/2}) \lozc;
\draw (15,{sqrt(3)}) \lozb;
\draw (16,{sqrt(3)}) \loza;
\draw (16,{sqrt(3)}) \lozc;
\draw (18,{sqrt(3)}) \lozb;
\draw (19,{sqrt(3)}) \loza;
\draw (20,{sqrt(3)}) \lozb;
\draw (21,{sqrt(3)}) \loza;
\draw (14.5,{(1.5)*sqrt(3)}) \loza;
\draw (14.5,{(1.5)*sqrt(3)}) \lozc;
\draw (16.5,{(1.5)*sqrt(3)}) \loza;
\draw (16.5,{(1.5)*sqrt(3)}) \lozc;
\draw (17.5,{(1.5)*sqrt(3)}) \lozc;
\draw (17.5,{(1.5)*sqrt(3)}) \lozc;
\draw (19.5,{(1.5)*sqrt(3)}) \loza;
\draw (19.5,{(1.5)*sqrt(3)}) \lozc;
\draw (21.5,{(1.5)*sqrt(3)}) \loza;
\draw (14,{2*sqrt(3)}) \loza;
\draw (15,{2*sqrt(3)}) \lozb;
\draw (16,{2*sqrt(3)}) \loza;
\draw (17,{2*sqrt(3)}) \lozb;
\draw (18,{2*sqrt(3)}) \loza;
\draw (18,{2*sqrt(3)}) \lozc;
\draw (20,{2*sqrt(3)}) \loza;
\draw (20,{2*sqrt(3)}) \lozc;
\draw (22,{2*sqrt(3)}) \loza;
\draw (13.5,{(2.5)*sqrt(3)}) \loza;
\draw (14.5,{(2.5)*sqrt(3)}) \loza;
\draw (14.5,{(2.5)*sqrt(3)}) \lozc;
\draw (16.5,{(2.5)*sqrt(3)}) \loza;
\draw (16.5,{(2.5)*sqrt(3)}) \lozc;
\draw (18.5,{(2.5)*sqrt(3)}) \loza;
\draw (18.5,{(2.5)*sqrt(3)}) \lozc;
\draw (20.5,{(2.5)*sqrt(3)}) \lozb;
\draw (21.5,{(2.5)*sqrt(3)}) \loza;
\draw (22.5,{(2.5)*sqrt(3)}) \loza;
\draw (13,{3*sqrt(3)}) \loza;
\draw (14,{3*sqrt(3)}) \loza;
\draw (15,{3*sqrt(3)}) \loza;
\draw (15,{3*sqrt(3)}) \lozc;
\draw (17,{3*sqrt(3)}) \lozb;
\draw (18,{3*sqrt(3)}) \loza;
\draw (19,{3*sqrt(3)}) \lozb;
\draw (20,{3*sqrt(3)}) \lozb;
\draw (21,{3*sqrt(3)}) \loza;
\draw (22,{3*sqrt(3)}) \loza;
\draw (23,{3*sqrt(3)}) \loza;
\draw (12.5,{(3.5)*sqrt(3)}) \loza;
\draw (13.5,{(3.5)*sqrt(3)}) \loza;
\draw (14.5,{(3.5)*sqrt(3)}) \loza;
\draw (15.5,{(3.5)*sqrt(3)}) \loza;
\draw (15.5,{(3.5)*sqrt(3)}) \lozc;
\draw (16.5,{(3.5)*sqrt(3)}) \lozc;
\draw (18.5,{(3.5)*sqrt(3)}) \lozb;
\draw (19.5,{(3.5)*sqrt(3)}) \lozb;
\draw (20.5,{(3.5)*sqrt(3)}) \loza;
\draw (21.5,{(3.5)*sqrt(3)}) \loza;
\draw (22.5,{(3.5)*sqrt(3)}) \loza;
\draw (23.5,{(3.5)*sqrt(3)}) \loza;

\filldraw (18,{(.5)*sqrt(3)}) circle (3pt);
\filldraw (17.5,{sqrt(3)}) circle (3pt);
\filldraw (19.5,{sqrt(3)}) circle (3pt);
\filldraw (16,{(1.5)*sqrt(3)}) circle (3pt);
\filldraw (19,{(1.5)*sqrt(3)}) circle (3pt);
\filldraw (21,{(1.5)*sqrt(3)}) circle (3pt);
\filldraw (14.5,{2*sqrt(3)}) circle (3pt);
\filldraw (16.5,{2*sqrt(3)}) circle (3pt);
\filldraw (19.5,{2*sqrt(3)}) circle (3pt);
\filldraw (21.5,{2*sqrt(3)}) circle (3pt);
\draw (14,{(2.5)*sqrt(3)}) circle (3pt);
\filldraw (16,{(2.5)*sqrt(3)}) circle (3pt);
\filldraw (18,{(2.5)*sqrt(3)}) circle (3pt);
\filldraw (20,{(2.5)*sqrt(3)}) circle (3pt);
\draw (22,{(2.5)*sqrt(3)}) circle (3pt);
\draw (13.5,{3*sqrt(3)}) circle (3pt);
\draw (14.5,{3*sqrt(3)}) circle (3pt);
\filldraw (16.5,{3*sqrt(3)}) circle (3pt);
\filldraw (18.5,{3*sqrt(3)}) circle (3pt);
\draw (21.5,{3*sqrt(3)}) circle (3pt);
\draw (22.5,{3*sqrt(3)}) circle (3pt);
\draw (13,{(3.5)*sqrt(3)}) circle (3pt);
\draw (14,{(3.5)*sqrt(3)}) circle (3pt);
\draw (15,{(3.5)*sqrt(3)}) circle (3pt);
\filldraw (18,{(3.5)*sqrt(3)}) circle (3pt);
\draw (21,{(3.5)*sqrt(3)}) circle (3pt);
\draw (22,{(3.5)*sqrt(3)}) circle (3pt);
\draw (23,{(3.5)*sqrt(3)}) circle (3pt);
\draw (12.5,{4*sqrt(3)}) circle (3pt);
\draw (13.5,{4*sqrt(3)}) circle (3pt);
\draw (14.5,{4*sqrt(3)}) circle (3pt);
\draw (15.5,{4*sqrt(3)}) circle (3pt);
\draw (20.5,{4*sqrt(3)}) circle (3pt);
\draw (21.5,{4*sqrt(3)}) circle (3pt);
\draw (22.5,{4*sqrt(3)}) circle (3pt);
\draw (23.5,{4*sqrt(3)}) circle (3pt);

\end{tikzpicture}
\caption{Left: Equivalent interlaced particle configuration of the example tiling
of figure \ref{figRegHexLoz}.
\newline
Right: Equivalent interlaced particle configuration with added deterministic
lozenges/particles. The unfilled circles represent the deterministic
particles.}
\label{figEquivIntPart}
\end{figure}
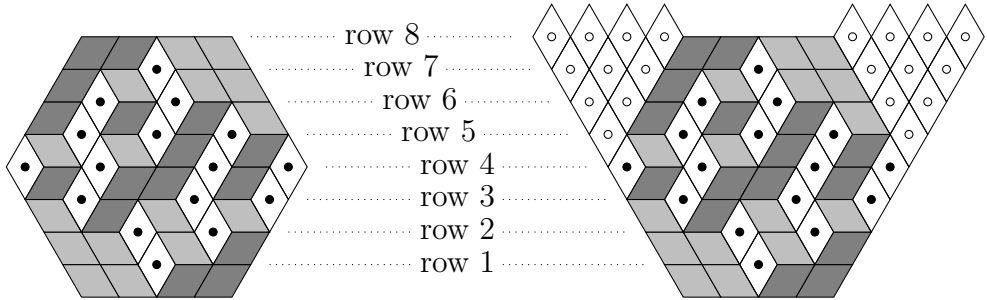

A further equivalent measure is obtained by adding deterministic
lozenges/particles to each configuration in a particular way,
as was done in Nordenstam, \cite{Nord09}, and Petrov, \cite{Pet14}.
This is demonstrated on the right of figure \ref{figEquivIntPart}: We
trivially add two densely packed blocks of lozenges/particles
to the upper left and upper right sides of the hexagon. These
deterministic lozenges/particles are independent of the tiling. Also,
specifying the positions of the deterministic lozenges/particles on the
top row, the interlacing constraint induces the positions of the deterministic
lozenges/particles on the lower rows. Therefore the uniform
measure on the set of all typical tilings is equivalent to the uniform
measure on the set of all configurations of particles which satisfy:
\begin{itemize}
\item
The particles lie on $2m$ rows, which we label from the bottom to the top.
\item
Adjacent rows are a distance of $\frac{\sqrt{3}}2$ apart.
\item
There are $r$ distinct particles on row $r$ for all $r \in \{1,\ldots,2m\}$.
\item
Particles are in integer positions on the even rows.
\item
Particles are in half-integer positions (i.e., $\Z + \frac12$) on the odd rows.
\item
The particles on row $2m$ are in the positions
$\{1,\ldots,m\} \cup \{2m+1,\ldots,3m\}$.
\item
Particles on adjacent rows {\em interlace}: For all $r \in \{1,\ldots,2m-1\}$,
there is exactly one particle on row $r$ `between' each neighbouring pair of
particles on row $r+1$.
\end{itemize}

A natural generalisation is to allow the lozenges/particles on the
top row to be in arbitrary positions. More specifically, we consider the set
of all lozenge tilings of the `half-hexagon' with sides of length $n \ge 1$
and $m \ge 1$, as shown on the left of figure \ref{figTilHalfHex}. We fix
$n$ vertical lozenges/particles in arbitrary deterministic integer positions
on the upper boundary, and consider the uniform measure on the
set of all possible tilings with this top row. An example of such a tiling,
and its equivalent interlaced particle configuration, is shown on the
right of figure \ref{figTilHalfHex}.

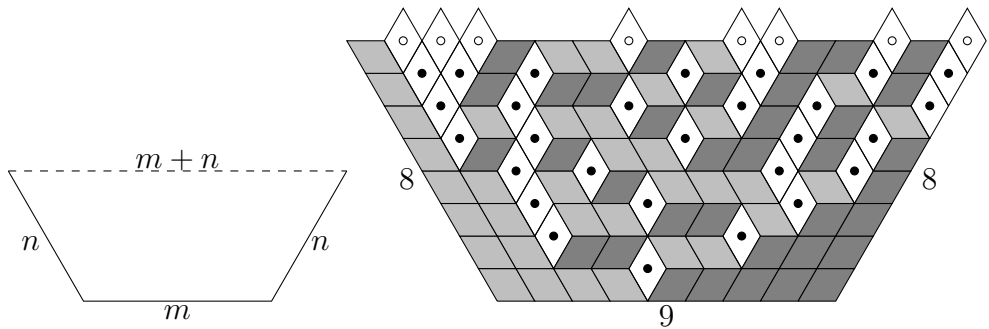
\begin{figure}[t]
\centering
\begin{tikzpicture}[xscale=1/2,yscale=1/2]

\draw (-13,{2*sqrt(3)})
--++ (2,{-2*sqrt(3)})
--++ (5,0)
--++ (2,{2*sqrt(3)});
\draw [dashed] (-13,{2*sqrt(3)}) --++ (9,0); 
\draw (-8.5,-.3) node {$m$};
\draw (-4.7,{sqrt(3)-.2}) node {$n$};
\draw (-8.5,{2*sqrt(3)+.3}) node {$m+n$};
\draw (-12.4,{sqrt(3)-.2}) node {$n$};

\draw (0,0) \lozb;
\draw (1,0) \lozb;
\draw (2,0) \lozb;
\draw (3,0) \lozb;
\draw (4,0) \loza;
\draw (4,0) \lozc;
\draw (5,0) \lozc;
\draw (6,0) \lozc;
\draw (7,0) \lozc;
\draw (8,0) \lozc;
\draw (-.5,{sqrt(3)/2}) \lozb;
\draw (.5,{sqrt(3)/2}) \lozb;
\draw (1.5,{sqrt(3)/2}) \loza;
\draw (1.5,{sqrt(3)/2}) \lozc;
\draw (2.5,{sqrt(3)/2}) \lozc;
\draw (4.5,{sqrt(3)/2}) \lozb;
\draw (5.5,{sqrt(3)/2}) \lozb;
\draw (6.5,{sqrt(3)/2}) \loza;
\draw (6.5,{sqrt(3)/2}) \lozc;
\draw (7.5,{sqrt(3)/2}) \lozc;
\draw (8.5,{sqrt(3)/2}) \lozc;
\draw (-1,{sqrt(3)}) \lozb;
\draw (0,{sqrt(3)}) \lozb;
\draw (1,{sqrt(3)}) \loza;
\draw (2,{sqrt(3)}) \lozb;
\draw (3,{sqrt(3)}) \lozb;
\draw (4,{sqrt(3)}) \loza;
\draw (4,{sqrt(3)}) \lozc;
\draw (5,{sqrt(3)}) \lozc;
\draw (7,{sqrt(3)}) \lozb;
\draw (8,{sqrt(3)}) \loza;
\draw (8,{sqrt(3)}) \lozc;
\draw (9,{sqrt(3)}) \lozc;
\draw (-1.5,{(1.5)*sqrt(3)}) \lozb;
\draw (-.5,{(1.5)*sqrt(3)}) \lozb;
\draw (.5,{(1.5)*sqrt(3)}) \loza;
\draw (1.5,{(1.5)*sqrt(3)}) \lozb;
\draw (2.5,{(1.5)*sqrt(3)}) \loza;
\draw (2.5,{(1.5)*sqrt(3)}) \lozc;
\draw (4.5,{(1.5)*sqrt(3)}) \lozb;
\draw (5.5,{(1.5)*sqrt(3)}) \lozb;
\draw (6.5,{(1.5)*sqrt(3)}) \lozb;
\draw (7.5,{(1.5)*sqrt(3)}) \loza;
\draw (8.5,{(1.5)*sqrt(3)}) \lozb;
\draw (9.5,{(1.5)*sqrt(3)}) \loza;
\draw (9.5,{(1.5)*sqrt(3)}) \lozc;
\draw (-2,{2*sqrt(3)}) \lozb;
\draw (-1,{2*sqrt(3)}) \loza;
\draw (-1,{2*sqrt(3)}) \lozc;
\draw (1,{2*sqrt(3)}) \loza;
\draw (1,{2*sqrt(3)}) \lozc;
\draw (3,{2*sqrt(3)}) \lozb;
\draw (4,{2*sqrt(3)}) \lozb;
\draw (5,{2*sqrt(3)}) \loza;
\draw (5,{2*sqrt(3)}) \lozc;
\draw (6,{2*sqrt(3)}) \lozc;
\draw (8,{2*sqrt(3)}) \loza;
\draw (8,{2*sqrt(3)}) \lozc;
\draw (10,{2*sqrt(3)}) \loza;
\draw (10,{2*sqrt(3)}) \lozc;
\draw (-2.5,{(2.5)*sqrt(3)}) \lozb;
\draw (-1.5,{(2.5)*sqrt(3)}) \loza;
\draw (-.5,{(2.5)*sqrt(3)}) \lozb;
\draw (.5,{(2.5)*sqrt(3)}) \loza;
\draw (1.5,{(2.5)*sqrt(3)}) \lozb;
\draw (2.5,{(2.5)*sqrt(3)}) \lozb;
\draw (3.5,{(2.5)*sqrt(3)}) \loza;
\draw (3.5,{(2.5)*sqrt(3)}) \lozc;
\draw (5.5,{(2.5)*sqrt(3)}) \lozb;
\draw (6.5,{(2.5)*sqrt(3)}) \loza;
\draw (6.5,{(2.5)*sqrt(3)}) \lozc;
\draw (8.5,{(2.5)*sqrt(3)}) \loza;
\draw (8.5,{(2.5)*sqrt(3)}) \lozc;
\draw (10.5,{(2.5)*sqrt(3)}) \lozb;
\draw (11.5,{(2.5)*sqrt(3)}) \loza;
\draw (-3,{3*sqrt(3)}) \lozb;
\draw (-2,{3*sqrt(3)}) \loza;
\draw (-1,{3*sqrt(3)}) \loza;
\draw (-1,{3*sqrt(3)}) \lozc;
\draw (1,{3*sqrt(3)}) \loza;
\draw (1,{3*sqrt(3)}) \lozc;
\draw (2,{3*sqrt(3)}) \lozc;
\draw (4,{3*sqrt(3)}) \lozb;
\draw (5,{3*sqrt(3)}) \loza;
\draw (5,{3*sqrt(3)}) \lozc;
\draw (7,{3*sqrt(3)}) \loza;
\draw (7,{3*sqrt(3)}) \lozc;
\draw (9,{3*sqrt(3)}) \lozb;
\draw (10,{3*sqrt(3)}) \loza;
\draw (10,{3*sqrt(3)}) \lozc;
\draw (12,{3*sqrt(3)}) \loza;
\draw (-3.5,{(3.5)*sqrt(3)}) \lozb;
\draw (-2.5,{(3.5)*sqrt(3)}) \loza;
\draw (-1.5,{(3.5)*sqrt(3)}) \loza;
\draw (-.5,{(3.5)*sqrt(3)}) \loza;
\draw (-.5,{(3.5)*sqrt(3)}) \lozc;
\draw (1.5,{(3.5)*sqrt(3)}) \lozb;
\draw (2.5,{(3.5)*sqrt(3)}) \lozb;
\draw (3.5,{(3.5)*sqrt(3)}) \loza;
\draw (3.5,{(3.5)*sqrt(3)}) \lozc;
\draw (5.5,{(3.5)*sqrt(3)}) \lozb;
\draw (6.5,{(3.5)*sqrt(3)}) \loza;
\draw (7.5,{(3.5)*sqrt(3)}) \loza;
\draw (7.5,{(3.5)*sqrt(3)}) \lozc;
\draw (8.5,{(3.5)*sqrt(3)}) \lozc;
\draw (10.5,{(3.5)*sqrt(3)}) \loza;
\draw (10.5,{(3.5)*sqrt(3)}) \lozc;
\draw (12.5,{(3.5)*sqrt(3)}) \loza;

\filldraw (4,{sqrt(3)/2}) circle (3pt);
\filldraw (1.5,{sqrt(3)}) circle (3pt);
\filldraw (6.5,{sqrt(3)}) circle (3pt);
\filldraw (1,{(1.5)*sqrt(3)}) circle (3pt);
\filldraw (4,{(1.5)*sqrt(3)}) circle (3pt);
\filldraw (8,{(1.5)*sqrt(3)}) circle (3pt);
\filldraw (.5,{2*sqrt(3)}) circle (3pt);
\filldraw (2.5,{2*sqrt(3)}) circle (3pt);
\filldraw (7.5,{2*sqrt(3)}) circle (3pt);
\filldraw (9.5,{2*sqrt(3)}) circle (3pt);
\filldraw (-1,{(2.5)*sqrt(3)}) circle (3pt);
\filldraw (1,{(2.5)*sqrt(3)}) circle (3pt);
\filldraw (5,{(2.5)*sqrt(3)}) circle (3pt);
\filldraw (8,{(2.5)*sqrt(3)}) circle (3pt);
\filldraw (10,{(2.5)*sqrt(3)}) circle (3pt);
\filldraw (-1.5,{3*sqrt(3)}) circle (3pt);
\filldraw (.5,{3*sqrt(3)}) circle (3pt);
\filldraw (3.5,{3*sqrt(3)}) circle (3pt);
\filldraw (6.5,{3*sqrt(3)}) circle (3pt);
\filldraw (8.5,{3*sqrt(3)}) circle (3pt);
\filldraw (11.5,{3*sqrt(3)}) circle (3pt);
\filldraw (-2,{(3.5)*sqrt(3)}) circle (3pt);
\filldraw (-1,{(3.5)*sqrt(3)}) circle (3pt);
\filldraw (1,{(3.5)*sqrt(3)}) circle (3pt);
\filldraw (5,{(3.5)*sqrt(3)}) circle (3pt);
\filldraw (7,{(3.5)*sqrt(3)}) circle (3pt);
\filldraw (10,{(3.5)*sqrt(3)}) circle (3pt);
\filldraw (12,{(3.5)*sqrt(3)}) circle (3pt);
\draw (-2.5,{4*sqrt(3)}) circle (3pt);
\draw (-1.5,{4*sqrt(3)}) circle (3pt);
\draw (-.5,{4*sqrt(3)}) circle (3pt);
\draw (3.5,{4*sqrt(3)}) circle (3pt);
\draw (6.5,{4*sqrt(3)}) circle (3pt);
\draw (7.5,{4*sqrt(3)}) circle (3pt);
\draw (10.5,{4*sqrt(3)}) circle (3pt);
\draw (12.5,{4*sqrt(3)}) circle (3pt);

\draw (4.5,-.4) node {$9$};
\draw (11.5,{2*sqrt(3)-.2}) node {$8$};
\draw (-2.4,{2*sqrt(3)-.2}) node {$8$};

\end{tikzpicture}
\caption{Left: A `half-hexagon' with sides of length $n \ge 1$ and $m \ge 1$.
The dotted line representing the upper boundary is considered to be `open'.
\newline
Right: An example tiling and its equivalent interlaced particle configuration
when $n=8$ and $m=9$. The unfilled circles represent the
deterministic lozenges/particles.}
\label{figTilHalfHex}
\end{figure}

Rescaling the sides of the `half-hexagon' by $\frac1n$, we consider
the asymptotic behaviour of the system as $n,m \to \infty$ under the
assumption that $\frac{m}n$ converges to a positive constant and the
empirical distribution of the deterministic lozenges/particles on the
top row converges weakly. We consider the asymptotic `shape' of such
systems. Petrov, \cite{Pet14}, studied the special case where the
particles on the top row are contained in a finite number of densely
packed blocks, and the empirical distribution converges to the
Lebesgue measure restricted to a finite number of closed disjoint
intervals. By adding deterministic lozenges/particles as for the
regular hexagon, this is equivalent to tiling those types of
polygons shown on the left of figure \ref{figPolygons}. The
results of this paper hold for any probability measure that can be
obtained as the weak limit of the empirical distribution.

We end this section by comparing and contrasting our results to those of
Kenyon {\em et al.}, \cite{Ken06} and \cite{Ken07}. The asymptotic frozen
boundaries of the polygons of figure \ref{figPolygons}, for example, can
be studied using the techniques of these papers. The boundaries are shown
to be {\em algebraic}. As stated above, this paper and \cite{Duse15b}
studies the frozen boundaries of a natural generalisation of those polygons
shown on the left of figure \ref{figPolygons}. The techniques of Kenyon
{\em et al.} do not cover such models, and our techniques do not cover the
polygons shown in the middle and on the right. In our case,
the asymptotic frozen boundaries are not necessarily algebraic.
We do, however, obtain parameterisations of the boundaries,
and we perform a detailed analysis of their local geometric properties.
Finally note, in \cite{Ken07}, the asymptotic frozen boundary
of the polygon in the middle is shown to be a cardioid. In \cite{Duse15d},
we consider a related situation, and study the asymptotic behaviour of
particles in a neighbourhood of a cusp in the frozen boundary. These
do not behave as a {\em Pearcey} point process, as previously expected.
We obtain a novel point process, which we call the {\em Cusp Airy}
process. This process can also appear for those polygons studied by
Petrov, i.e., those shown on the left of figure \ref{figPolygons}.
An example of such a polygon and cusp is given in section
\ref{secEx4}. The polygon and cusp in question can
be seen in figure \ref{figex4}.

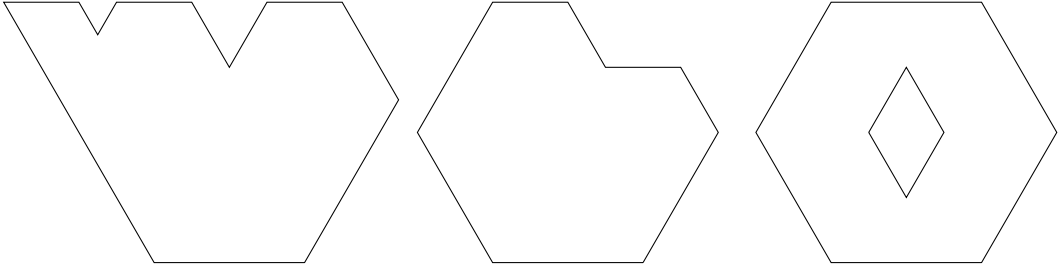
\begin{figure}[t]
\centering
\begin{tikzpicture}[xscale=0.25,yscale=0.25]

\draw(8,0)
--++(-8,0)
--++(-8,13.8564064605510)
--++(4,0)
--++(1,-1.73205080756888)
--++(1,1.73205080756888)
--++(4,0)
--++(2,-3.46410161513775)
--++(2,3.46410161513775)
--++(4,0)
--++(3,-5.19615242270663)
--++(-5,-8.66025403784439);

\draw(26,0)
--++(-8,0)
--++(-4,6.92820323027551)
--++(4,6.92820323027551)
--++(4,0)
--++(2,-3.46410161513775)
--++(4,0)
--++(2,-3.46410161513775)
--++(-4,-6.92820323027551);

\draw(44,0)
--++(-8,0)
--++(-4,6.92820323027551)
--++(4,6.92820323027551)
--++(8,0)
--++(4,-6.92820323027551)
--++(-4,-6.92820323027551);

\draw(38,6.92820323027551)
--++(2,3.46410161513775)
--++(2,-3.46410161513775)
--++(-2,-3.46410161513775)
--++(-2,3.46410161513775);

\end{tikzpicture}
\caption{Polygons of various shapes that can be tiled with lozenges.
\newline
Left: A `half-hexagon' with V-shaped cuts on the top boundary.
\newline
Center: A regular hexagon with a corner removed.
\newline
Right: A regular hexagon with a diamond shaped hole in the center.}
\label{figPolygons}
\end{figure}

\subsection{Determinantal random point processes}

In this section we give a brief introduction to determinantal random point
processes that suffices for our purposes. See Johansson, \cite{Joh06b},
for a more complete treatment.

Let $\Lambda$ be a Polish space. Fix $N \in \N \cup \{\infty\}$ and
$Y \subset \Lambda^N$, a space of configurations of $N$-particles of
$\Lambda$. Denote each $y \in Y$ as $y = (y_1,\ldots,y_N)$. Assume, for all
$y \in Y$ and compact Borel sets $B \subset \Lambda$, that the number of particles
from $y$ contained in $B$ is finite, i.e., $\# \{y_i \in B\} < \infty$. Let
$\mathcal{F}$ be the sigma-algebra generated by sets of the form
$\{y \in Y: \# \{y_i \in B\} = m \}$ for all $m \le N$ and Borel sets
$B \subset \Lambda$. A probability space of the form $(Y,\mathcal{F},\P)$ is
referred to as a {\em random point process}.

Given such a process, $m \leq N$, and $B \subset \Lambda^m$,
define $N_B^m : Y \to \N$ by,
\begin{equation*}
N_B^m (y) := \# \{(y_{i_1},\ldots,y_{i_m}) \in B : i_1 \ne \cdots \ne i_m \},
\end{equation*}
for all $y \in Y$. In words, $N_B^m (y)$ is the number of distinct
$m$-tuples of particles from $y$ that are contained in $B$. Then define a
measure on $\Lambda^m$ by $B \mapsto \E [ N_B^m ]$ for all Borel subsets
$B \subset \Lambda^m$. Assume that this is well-defined and finite whenever
$B$ is bounded. Then, given a reference measure $\lambda$ on $\Lambda$, the
density of the above measure with respect to $\lambda^m$, whenever it exists,
is referred to as the {\em $m^\text{th}$ correlation function}, $\rho_m$. That
is,
\begin{equation*}
\int_B \rho_m(x_1,\ldots,x_m) d\lambda[x_1] \ldots d\lambda[x_m] = \E [ N_B^m ],
\end{equation*}
for all Borel subsets $B \subset \Lambda^m$.

A random point process is called {\em determinantal} if all correlation functions
exist and there exists a function $K : \Lambda^2 \to \C$ for which
\begin{equation*}
\rho_m(x_1,\ldots,x_m) = \det[K(x_i,x_j)]_{i,j=1}^m,
\end{equation*}
for all $x_1,\ldots,x_m \in \Lambda$ and $m \leq N$. $K$ is called the {\em correlation
kernel} of the process.

\subsection{The determinantal structure of discrete Gelfand-Tsetlin patterns}
\label{sectdsodgtp}

For the remainder of the paper we restrict to the study of the interlaced
particle configurations that we introduced in section \ref{secTilingsHalf}.
Recall that adjacent rows are a distance of $\frac{\sqrt{3}}2$ apart, and
that particles on adjacent rows alternate between integer and half-integer
positions. Therefore, since it is more convenient to study configurations
of particles in $\Z^2$, we shift each row $r$ vertically by
$- (1 - \frac{\sqrt{3}}2) (n-r)$, and horizontally by $\frac12 (n-r)$. This
preserves the interlacing constraint, and the resulting configuration
is called a {\em Gelfand-Tsetlin pattern}, which we now define
rigorously:
\begin{definition}
\label{defGT}
A discrete Gelfand-Tsetlin pattern of depth $n$ is an $n$-tuple, denoted
$(y^{(1)},y^{(2)},\ldots,y^{(n)}) \in  \Z \times \Z^2 \times \cdots \times \Z^n$,
which satisfies the interlacing constraint
\begin{equation*}
y_1^{(r+1)} \; \ge \; y_1^{(r)} \; > \; y_2^{(r+1)} \; \ge \; y_2^{(r)}
\; > \cdots \ge \; y_r^{(r)} \; > \; y_{r+1}^{(r+1)},
\end{equation*}
for all $r \in \{1,\ldots,n-1\}$, denoted $y^{(r+1)} \succ y^{(r)}$. Equivalently
this can be considered as an interlaced configuration of $\frac12 n (n+1)$
particles in $\Z \times \{1,\ldots,n\}$ by placing a particle at position
$(u,r) \in \Z \times \{1,\ldots,n\}$ whenever $u$ is an element of $y^{(r)}$.
\end{definition}
A Gelfand-Tsetlin pattern of depth $4$ is shown on the
left of figure \ref{figGTAsyShape}.

For each $n\ge1$, fix $x^{(n)} \in \Z^n$ with $x_1^{(n)} > x_2^{(n)} > \cdots > x_n^{(n)}$.
Consider the uniform probability measure, $\nu_n$, on the set of discrete Gelfand-Tsetlin
patterns of depth $n$ with the particles on row $n$ in the deterministic positions defined
by $x^{(n)}$:
\begin{equation*}
\nu_n[(y^{(1)},\ldots,y^{(n)})]
:= \frac1{Z_n} \cdot \left\{
\begin{array}{rcl}
1 & ; &
\text{when} \; x^{(n)} = y^{(n)} \succ y^{(n-1)} \succ \cdots \succ y^{(1)}, \\
0 & ; & \text{otherwise},
\end{array}
\right.
\end{equation*}
where $Z_n > 0$ is a normalisation constant. This measure, and the equivalent
description of Gelfand-Tsetlin patterns given in definition \ref{defGT},
induces a random point process on interlaced configurations of particles in
$\Z \times \{1,\ldots,n\}$. This process is determinantal, a fact which
follows from the equivalent description of the system as perfect matchings
of hexagonal planar graphs (see, for example, Kenyon, \cite{Ken97}). This
observation does not, however, provide a convenient expression for the
correlation kernel of the process. In section \ref{secdofckfit}, we use the
Gelfand-Tsetlin description to find such an expression. We denote the
correlation kernel by $K_n : (\Z \times \{1,\ldots,n\})^2 \to \C$, a
function of pairs of particle positions. Ignoring the deterministic
particles on row $n$, interlacing implies that we need only consider
those particle positions, $(u,r), (v,s) \in \Z \times \{1,\ldots,n-1\}$,
which satisfy $u \ge x_n^{(n)}+n-r$ and $v \ge x_n^{(n)}+n-s$. For all
such $(u,r),(v,s)$, we show in section \ref{secdofckfit} that
\begin{equation}
\label{eqKnrusvFixTopLine}
K_n((u,r),(v,s)) = \widetilde{K}_n ((u,r),(v,s)) - \phi_{r,s}(u,v),
\end{equation}
where
\begin{align*}
\lefteqn{\widetilde{K}_n ((u,r),(v,s))} \\
& := \frac{(n-s)!}{(n-r-1)!} \sum_{k=1}^n 1_{(x_j^{(n)} \ge  u)} \sum_{l=v+s-n}^v
\frac{\prod_{j=u+r-n+1}^{u-1} (x_k^{(n)} - j)}{\prod_{j=v+s-n, \; j \neq l}^v (l - j)} \;
\prod_{i=1, \; i \neq k}^n \left( \frac{l - x_i^{(n)}}{x_k^{(n)} - x_i^{(n)}} \right),
\end{align*}
and
\begin{equation*}
\phi_{r,s} (u,v)
:= 1_{(v \ge u)} \cdot \left\{
\begin{array}{lll}
0 & ; & \text{when} \; s \le r, \\
1 & ; & \text{when} \; s = r+1, \\
\frac1{(s-r-1)!} \prod_{j=1}^{s-r-1} (v-u+s-r-j) & ; & \text{when} \; s > r+1.
\end{array}
\right.
\end{equation*}
This is a generalisation of Defosseux, \cite{Def08}, and Metcalfe,
\cite{Met13}, which consider a similar process on configurations in
$\R \times \{1,\ldots,n\}$. The kernel in \cite{Def08} and \cite{Met13}
is recovered from the above kernel using asymptotic arguments. The above
kernel was also independently obtained by Petrov, \cite{Pet14}. Our proof,
based on the methods used in \cite{Def08} and \cite{Met13}, is more
elementary than that of Petrov. We highlight the differences at the
beginning of section \ref{secdofckfit}.

\begin{figure}[t]
\centering
\begin{tikzpicture}[xscale=3/4,yscale=3/4]

\draw (0,4.5) node {$y_4^{(4)}$};
\draw (2,4.5) node {$y_3^{(4)}$};
\draw (4,4.5) node {$y_2^{(4)}$};
\draw (6,4.5) node {$y_1^{(4)}$};
\draw (1,3) node {$y_3^{(3)}$};
\draw (3,3) node {$y_2^{(3)}$};
\draw (5,3) node {$y_1^{(3)}$};
\draw (2,1.5) node {$y_2^{(2)}$};
\draw (4,1.5) node {$y_1^{(2)}$};
\draw (3,0) node {$y_1^{(1)}$};

\draw (.45,3.75) node [rotate=-55] {$<$};
\draw (1.45,3.75) node [rotate=55] {$\le$};
\draw (2.45,3.75) node [rotate=-55] {$<$};
\draw (3.45,3.75) node [rotate=55] {$\le$};
\draw (4.45,3.75) node [rotate=-55] {$<$};
\draw (5.45,3.75) node [rotate=55] {$\le$};
\draw (1.45,2.25) node [rotate=-55] {$<$};
\draw (2.45,2.25) node [rotate=55] {$\le$};
\draw (3.45,2.25) node [rotate=-55] {$<$};
\draw (4.45,2.25) node [rotate=55] {$\le$};
\draw (2.45,.75) node [rotate=-55] {$<$};
\draw (3.45,.75) node [rotate=55] {$\le$};


\draw (7.5,3.5)
--++ (5,0)
--++ (0,-2.5)
--++ (-2.5,0)
--++ (-2.5,2.5);

\draw (7.5,3.9) node {$(a,1)$};
\draw (12.5,3.9) node {$(b,1)$};
\draw (12.5,.6) node {$(b,0)$};
\draw (10,.6) node {$(a+1,0)$};

\draw (15+1/3,14/3)
--++ (7/3,0)
--++ (7/3,-7/3)
--++ (0,-7/3)
--++ (-7/3,0)
--++ (-7/3,7/3)
--++ (0,7/3);
\draw [dashed] (15+1/3,14/3)
--++ (-7/3,0)
--++ (7/3,-7/3);
\draw [dashed] (17+2/3,14/3)
--++ (7/3,0)
--++ (0,-7/3);

\draw (14+1/6,5.1) node {$\frac12$};
\draw (16+1/2,5.1) node {$\frac12$};
\draw (18+5/6,5.1) node {$\frac12$};
\draw (20.3,3+1/2) node {$\frac12$};
\draw (20.3,1+1/6) node {$\frac12$};
\draw (18+5/6,-.4) node {$\frac12$};
\draw (16+1/2,.7) node {$\frac1{\sqrt{2}}$};
\draw (14+1/6,3) node {$\frac1{\sqrt{2}}$};

\draw (15.1,3+1/2) node {$\frac12$};
\draw (19.1,3.8) node {$\frac1{\sqrt{2}}$};

\begin{axis}[hide axis, xmin=0,xmax=20, ymin=0,ymax=4.66666666667, samples=30, smooth, x=1cm, y=1cm]
\addplot [domain=0:2*pi, thin]
({13 + 14/3*(3/4 + (sqrt(3)/4)*(cos(deg(x))) + (1/4)*(1-sin(deg(x))))},
{14/3-(14/3)*(1-sin(deg(x)))/2});
\end{axis}

\end{tikzpicture}
\caption{Left: A visualisation of a Gelfand-Tsetlin pattern of depth $4$.
\newline
Middle: $\{(\chi,\eta) \in \R \times [0,1] : \chi \in [a + 1 - \eta,b] \}$,
where $a$ and $b$ satisfy $b-a>1$ (see hypothesis \ref{hypWeakConv}).
Equations (\ref{eqEsc1}) and (\ref{eqEsc2}) imply that the bulk of the
rescaled particles of the Gelfand-Tsetlin patterns lie asymptotically
in this region as $n \to \infty$.
\newline
Right: The shifted asymptotic shape of the rescaled regular hexagon. The areas
enclosed by the dashed lines represent the added regions of deterministic
lozenges/particles, as described in section \ref{secTilingsHalf}.}
\label{figGTAsyShape}
\end{figure}
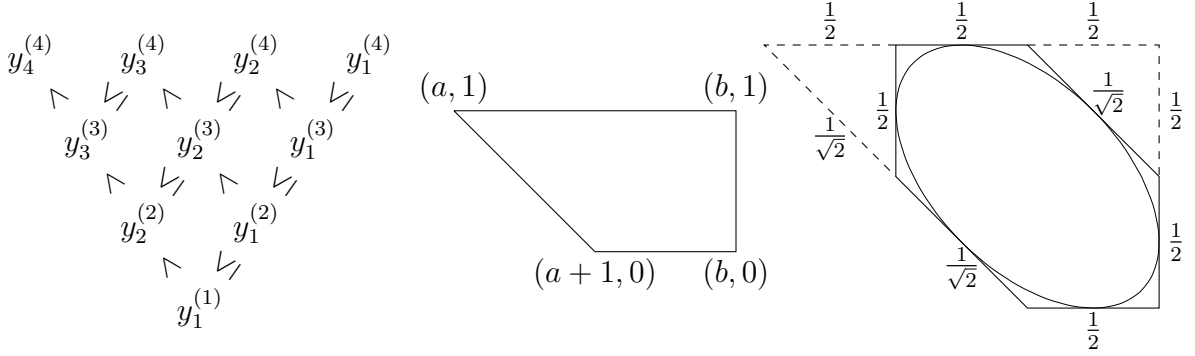

\subsection{Motivation and statement of main results}
\label{secmasomr}

In this paper we study the asymptotic behaviour of the determinantal system
introduced in the previous section as $n \to \infty$, under the assumption
that the (rescaled) empirical distribution of the deterministic particles
on row $n$ converges weakly to a measure with compact support:
\begin{hyp}
\label{hypWeakConv}
Let $\mu$ be a probability measure on $\R$ with compact support,
$\supp(\mu) \subset [a,b]$. Assume that $b-a>1$, $\{a,b\} \subset \supp(\mu)$,
and
\begin{equation*}
\frac1n \sum_{i=1}^n \delta_{x_i^{(n)}/n} \to \mu
\end{equation*}
as $n \to \infty$, in the sense of weak convergence of measures.
\end{hyp}

\begin{rem}
\label{remmu}
Recalling that $\{x_1^{(n)},x_2^{(n)},\ldots,x_n^{(n)}\} \subset \Z$, it trivially follows
that measures, $\mu$, which satisfy hypothesis \ref{hypWeakConv} are absolutely
continuous with respect to Lebesgue measure, $\l$. Moreover, the density of $\mu$
takes values in $[0,1]$, and so $\mu \le \l$. Finally, $\supp(\mu)$ and
$\supp(\l-\mu)$ contain no isolated singletons.
\end{rem}

For each $n$, we now rescale the vertical and horizontal positions of the particles
of the Gelfand-Tsetlin pattern of depth $n$ by $\frac1n$. Recall that
$\{x_1^{(n)},x_2^{(n)},\ldots,x_n^{(n)}\}$ are the (deterministic) particles
on the top row of the Gelfand-Tsetlin pattern of depth $n$. Note, hypothesis
\ref{hypWeakConv} easily gives,
\begin{equation}
\label{eqEsc1}
\# \bigg\{ x \in \bigg\{ \frac{x_1^{(n)}}n, \frac{x_2^{(n)}}n, \ldots, \frac{x_n^{(n)}}n \bigg\}
: x \in [a,b] \bigg\} = n + o(n),
\end{equation}
for all $n$ sufficiently large. Thus, at most $o(n)$ of the rescaled particles
on the top row are not contained in the interval $[a,b]$ for all $n$ sufficiently
large. Next, fix any $\eta \in [0,1]$, and any $\{m_n\}_{n\ge1} \subset \Z$
with $m_n \in \{1,2,\ldots,n-1\}$ and $\frac{m_n}n = \eta + o(1)$ for all $n$
sufficiently large. Also, let
$\{ y_1^{(m_n)}, y_2^{(m_n)}, \ldots, y_{m_n}^{(m_n)} \}$ denote the
(random) particles on row $m_n$ of the Gelfand-Tsetlin pattern of depth
$n$. Equation (\ref{eqEsc1}) and the interlacing constraint then give,
\begin{equation}
\label{eqEsc2}
\# \bigg\{ y \in \bigg\{ \frac{y_1^{(m_n)}}n, \frac{y_2^{(m_n)}}n,
\ldots, \frac{y_{m_n}^{(m_n)}}n \bigg\} : y \in [a+1-\eta,b] \bigg\} = m_n + o(n),
\end{equation}
for all $n$ sufficiently large. Thus, at most $o(n)$ of the rescaled particles
on row $m_n$ are not contained in the interval $[a+1-\eta,b]$ for all $n$ sufficiently
large. Thus, since $\eta$ is any value in $[0,1]$, the bulk of the rescaled
particles of the Gelfand-Tsetlin patterns lie asymptotically in
$\{(\chi,\eta) \in \R \times [0,1] : \chi \in [a + 1 - \eta,b] \}$ as $n \to \infty$.
This geometric subset of $\R^2$ is shown in the middle of figure \ref{figGTAsyShape}.
The example of the regular hexagon is shown on the right of figure
\ref{figGTAsyShape}. This figure is obtained from figure \ref{figArctic} by
performing the shift described at the beginning of section \ref{sectdsodgtp}.
We recover this figure in section \ref{secEx3}, using our techniques.
 
The local asymptotic behaviour of particles near a point, $(\chi,\eta)$,
in the shape in the middle of figure \ref{figGTAsyShape}, can be examined
by considering the asymptotic behaviour of $K_n((u_n,r_n),(v_n,s_n))$ as
$n \to \infty$, where $\{(u_n,r_n)\}_{n\ge1}$ and $\{(v_n,s_n)\}_{n\ge1}$
are sequences in $\Z^2$ which satisfy:
\begin{hyp}
\label{hypunrnsnvn}
Fix $(\chi,\eta)$ in the shape in middle of figure
\ref{figGTAsyShape}, i.e., $(\chi,\eta) \in [a,b] \times [0,1]$
with $b \ge \chi \ge \chi + \eta - 1 \ge a$. Assume that
$\frac1n (u_n,r_n) \to (\chi,\eta)$ and $\frac1n (v_n,s_n) \to (\chi,\eta)$
as $n \to \infty$.
\end{hyp}

The asymptotic behaviour of $K_n((u_n,r_n),(v_n,s_n))$ can be examined
using steepest descent techniques. To see this, first note that equation
(\ref{eqKnrusvFixTopLine}) and the Residue Theorem give,
\begin{equation}
\label{eqKnrnunsnvn1}
K_n((u_n,r_n),(v_n,s_n))
= \left( \frac{(n-s_n)!}{(n-r_n-1)!} \; \frac{n^{n-r_n-1}}{n^{n-s_n}} \right) J_n
- \phi_{r_n,s_n}(u_n,v_n),
\end{equation}
where, dropping the superscript from $x^{(n)}$,
\begin{equation}
\label{eqJnrnunsnvn1}
J_n := \frac1{(2\pi i)^2} \int_{\g_n} dw \int_{\G_n} dz \;
\frac{\prod_{j=u_n+r_n-n+1}^{u_n-1}
(z - \frac{j}n)}{\prod_{j=v_n+s_n-n}^{v_n} (w - \frac{j}n)} \;
\frac1{w-z} \; \prod_{i=1}^n \left( \frac{w - \frac{x_i}n}{z - \frac{x_i}n} \right),
\end{equation}
for all $n \in \N$, where $\g_n$ and $\G_n$ are any counter-clockwise closed contours
that satisfy the requirements of figure \ref{figcontours1}. Also note that the integrand
can be written as
\begin{equation*}
\frac{\exp(n f_n(w) - n \tilde{f}_n(z))}{w-z},
\end{equation*}
for all
$w,z \in \C \setminus \R$, where
\begin{align*}
f_n(w)
& := \frac1n \sum_{i=1}^n \log \left( w - \frac{x_i}n \right) -
\frac1n \sum_{j=v_n+s_n-n}^{v_n} \log \left( w - \frac{j}n \right), \\
\tilde{f}_n(z)
& := \frac1n \sum_{i=1}^n \log \left( z - \frac{x_i}n \right) -
\frac1n \sum_{j=u_n+r_n-n+1}^{u_n-1} \log \left( z - \frac{j}n \right),
\end{align*}
and $\log$ denotes the principal logarithm. Finally, inspired by
hypotheses \ref{hypWeakConv} and \ref{hypunrnsnvn} we define
\begin{equation*}
f_{(\chi,\eta)} (w)
:= \int_a^b \log (w-x) \mu[dx] - \int_{\chi+\eta-1}^\chi \log (w-x) dx,
\end{equation*}
for all $w \in \C \setminus \R$. Note that $f_n(w) \to f_{(\chi,\eta)} (w)$
and $\tilde{f}_n(z) \to f_{(\chi,\eta)} (z)$ as $n \to \infty$ for all
$w,z \in \C \setminus \R$.

\begin{figure}[t]
\centering
\begin{tikzpicture}

\draw (0,0) ellipse (5cm and 2cm);
\draw (0,0) ellipse (3cm and 1.5cm);
\draw(-6,0) --++(12,0);

\draw[arrows=->,line width=1pt](.01,2)--(0,2);
\draw[arrows=->,line width=1pt](-.01,-2)--(0,-2);
\draw[arrows=->,line width=1pt](.01,1.5)--(0,1.5);
\draw[arrows=->,line width=1pt](-.01,-1.5)--(0,-1.5);

\draw (0,2.3) node {$\g_n$};
\draw (0,1.2) node {$\G_n$};

\end{tikzpicture}
\caption{$\G_n$ contains $\{\frac1n x_j : x_j \ge u_n\}$ and none of
$\{\frac1n x_j : x_j \le u_n+r_n-n \}$. $\g_n$ contains $\G_n$ and
$\{\frac1n v_n, \frac1n (v_n-1), \ldots, \frac1n (v_n+s_n-n)\}$.
Both contours are oriented counter-clockwise.}
\label{figcontours1}
\end{figure}
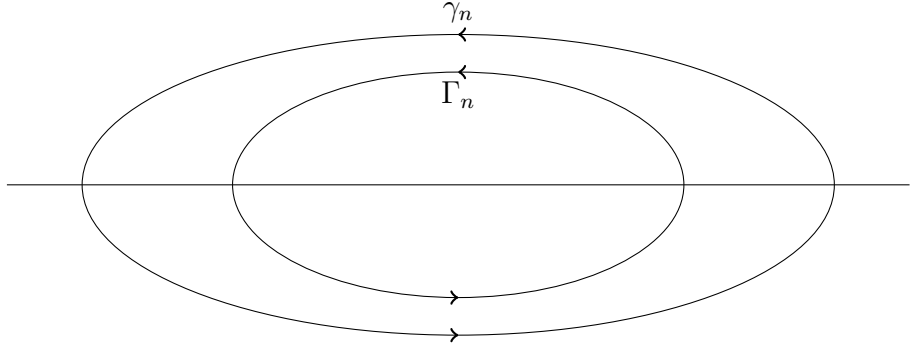

The above expression for the integrand, and the asymptotic behaviour of
the functions $f_n$ and $\tilde{f}_n$ in the exponent, suggest that the
contour integral representation of the kernel is amenable to steepest
descent techniques. More exactly, steepest descent analysis suggests that,
as $n \to \infty$, the asymptotic behaviour of $K_n((u_n,r_n),(v_n,s_n))$
depends on the behaviour of the roots of $f_{(\chi,\eta)}'$:
\begin{equation}
\label{eqf'0}
f_{(\chi,\eta)}' (w)
= \int_a^b \frac{\mu[dx]}{w-x} - \int_{\chi+\eta-1}^\chi \frac{dx}{w-x},
\end{equation}
for all $w \in \C \setminus \R$. Note also that hypotheses \ref{hypWeakConv} and
\ref{hypunrnsnvn} imply that $\mu[\chi,b] \ge 0$,
$(\l-\mu)[\chi+\eta-1,\chi] \ge 0$ and $\mu[a,\chi+\eta-1] \ge 0$,
where $\l$ is Lebesgue measure. It is therefore natural to write,
\begin{equation}
\label{eqf'}
f_{(\chi,\eta)}' (w)
= \int_\chi^b \frac{\mu[dx]}{w-x}
- \int_{\chi+\eta-1}^\chi \frac{(\l-\mu)[dx]}{w-x}
+ \int_a^{\chi+\eta-1} \frac{\mu[dx]}{w-x},
\end{equation}
for all $w \in \C \setminus \R$. Therefore $f_{(\chi,\eta)}'$ extends
analytically to
\begin{equation*}
\C \setminus \left (\supp(\mu |_{[\chi,b]})
\; \cup \; \supp((\l-\mu) |_{[\chi+\eta-1,\chi]})
\; \cup \; \supp(\mu |_{[a,\chi+\eta-1]}) \right).
\end{equation*}
The main result of section \ref{sectrofn'}, theorem
\ref{thmf'}, characterises all possible behaviours of the
roots of $f_{(\chi,\eta)}'$ in this domain.

\begin{rem}
\label{remCauTrans}
For simplicity of notation, we omit the subscript from $f_{(\chi,\eta)}$
whenever confusion is impossible. Moreover, note that the first term on
the right hand side of equation (\ref{eqf'0}) is the function,
\begin{equation*}
w \mapsto \int_a^b \frac{\mu[dx]}{w-x},
\end{equation*}
for all $w \in \C \setminus \R$. This function extends
analytically to $\C \setminus \supp(\mu)$, is conventionally called
the Cauchy transform of $\mu$, and is denoted by $C$.
\end{rem}

We now use the behaviours of the roots observed in theorem \ref{thmf'}
to divide the shape in the middle of figure \ref{figGTAsyShape} into
regions in which different asymptotic behaviours can be expected.
An important region is the following:
\begin{definition}
\label{defLiq}
The liquid region, $\LL$, is the set of all $(\chi,\eta)$ in the shape
in middle of figure \ref{figGTAsyShape}, i.e., $(\chi,\eta) \in [a,b] \times [0,1]$
and $b \ge \chi \ge \chi + \eta - 1 \ge a$, for which $f'$ has non-real roots.
\end{definition}
In section \ref{secGeo} we consider the geometric interpretation of this region.
First note that non-real roots of $f'$ occur in complex conjugate pairs. Theorem
\ref{thmf'} then implies that $(\chi,\eta) \in \LL$ if and only if $f'$ has exactly
$2$ roots in $\C \setminus \R$, counting multiplicities. More exactly, there exists a
unique $w_c \in \mathbb{H} := \{ w \in \C : \text{Im}(w) > 0 \}$ with
$f'(w_c) = f'(\overline{w_c}) = 0$. This defines a map from $\LL$ to $\mathbb{H}$.
In theorem \ref{thmwc} we show that this map is a homeomorphism (indeed, it is a
diffeomorphism, as we shall see in section \ref{secDiff}). Therefore $\LL$ is a
non-empty, open, connected set. Moreover, an explicit expression is obtained for
the inverse of the homeomorphism, denoted by,
\begin{equation*}
(\chi_\LL(\cdot),\eta_\LL(\cdot)) : \mathbb{H} \to \LL \subset [a,b] \times [0,1].
\end{equation*}

Steepest descent analysis and the above observations suggest, as
$n \to \infty$, that universal bulk asymptotic behaviour should be observed
whenever $(\chi,\eta) \in \LL$: Fixing $(\chi,\eta) \in \LL$, and choosing
the parameters $(u_n,r_n)$ and $(v_n,s_n)$ of hypothesis \ref{hypunrnsnvn}
appropriately, it should be possible to show that $K_n((u_n,r_n),(v_n,s_n))$
converges to the {\em Sine} kernel as $n \to \infty$. Note that an analogous
result was obtained in Metcalfe, \cite{Met13}, which considers similar
processes on configurations of particles in $\R \times \{1,\ldots,n\}$.
In the current situation, Petrov, \cite{Pet14}, confirmed universal bulk asymptotic
behaviour when the measure, $\mu$, of hypothesis \ref{hypWeakConv} is given by
the Lebesgue measure restricted to a finite number of closed disjoint intervals.

In section \ref{sectbotlr1} we use the above homeomorphism to
examine $\partial \LL$, the boundary of $\LL$. The main result of this
section, lemma \ref{lemBddEdge1}, defines a subset of $\partial \LL$
for any measure, $\mu$, of hypothesis \ref{hypWeakConv}. This is
defined by showing that $(\chi_\LL(\cdot),\eta_\LL(\cdot)) : \mathbb{H} \to \LL$
has a unique continuous extension to the open set $R \subset \R$ given by,
\begin{equation}
\label{eqR}
R := (\; \overline{(\R \setminus \supp(\mu)) \cup (\R \setminus \supp(\l-\mu))} \;)^\circ.
\end{equation}
Above we take the interior of the closure of the union. The extension is denoted by,
\begin{equation*}
(\chi_\EE(\cdot),\eta_\EE(\cdot)) : R \to [a,b] \times [0,1],
\end{equation*}
and an explicit expression is given. We conclude that
$(\chi_\EE(t),\eta_\EE(t)) \in \partial \LL$ for all $t \in R$,
and $(\chi_\EE(\cdot),\eta_\EE(\cdot)) : R \to \partial \LL$
is a smooth curve parameterised over $R$. Also, as we shall shortly
see, this map is injective. Finally we define:
\begin{definition}
\label{defEdge0}
The edge, $\EE \subset \partial \LL$, is the image of the smooth curve
$(\chi_\EE(\cdot),\eta_\EE(\cdot)) : R \to \partial \LL$.
\end{definition}

Lemma \ref{lemBddEdge1}, and the other results of section \ref{sectbotlr1},
give a complete description of $\partial \LL$ only when the measure $\mu$,
of hypothesis \ref{hypWeakConv}, is restricted to an interesting sub-class
of possible measures (see lemma \ref{lemBddEdge3}). When $\mu$ is not
restricted to the sub-class, these results only give a partial description.
Figure \ref{figexs}, below, depicts those sections of $\partial \LL$
obtained from these results for various examples of $\mu$. The examples are
examined in detail in section \ref{secEx}.

In section \ref{secedge} we construct an alternative description of the edge,
$\EE$, in terms of the behaviour of the roots of $f'$. This is analogous to
definition \ref{defLiq} for the liquid region, $\LL$:
\begin{definition}
\label{defEdge}
The edge, $\EE$, is the disjoint union
$\EE := \EE_\mu \cup \EE_{\l-\mu} \cup \EE_0 \cup \EE_1 \cup \EE_2$, where
\begin{itemize}
\item
$\EE_\mu$ is the set of all $(\chi,\eta)$ in the shape in the middle of
figure \ref{figGTAsyShape} for which $f'$ has a repeated root in
$\R \setminus [\chi+\eta-1,\chi]$.
\item
$\EE_{\l-\mu}$ is the set of all $(\chi,\eta)$ for which $f'$ has a
repeated root in $(\chi+\eta-1,\chi)$.
\item
$\EE_0$ is the set of all $(\chi,\eta)$ for which $\eta=1$ and $f'$ has a root
at $\chi \; (=\chi+\eta-1)$.
\item
$\EE_1$ is the set of all $(\chi,\eta)$ for which $\eta<1$ and $f'$ has a root at $\chi$.
\item
$\EE_2$ is the set of all $(\chi,\eta)$ for which $\eta<1$ and $f'$ has a root
at $\chi+\eta-1$.
\end{itemize}
\end{definition}
The fact that the above union is disjoint, and that $\LL$
and $\EE$ are disjoint, follows from corollary \ref{corf'}. To see
that the definitions are equivalent, note that theorem \ref{thmf'}
implies that $f'$ has at most one real-valued repeated root. Then,
starting with definition \ref{defEdge}, we can define a map from
$\EE$ to $\R$ by mapping $(\chi,\eta) \in \EE$ to the relevant root
of $f'$, i.e.,
\begin{itemize}
\item
to the unique real-valued repeated root whenever
$(\chi,\eta) \in \EE_\mu \cup \EE_{\l-\mu}$.
\item
to the root $\chi$ whenever $(\chi,\eta) \in \EE_0 \cup \EE_1$.
\item
to the root $\chi+\eta-1$ whenever $(\chi,\eta) \in \EE_0 \cup \EE_2$.
\end{itemize}
Theorem \ref{thmEdge} implies that this bijectively maps $\EE$ to $R$,
defined above. Moreover, the inverse of this map is the smooth curve of definition
\ref{defEdge0}, i.e., $(\chi_\EE(\cdot),\eta_\EE(\cdot)) : R \to \partial \LL$.
Thus the curve is a smooth and bijective map from $R$ to the edge, $\EE$, of
definition \ref{defEdge}. Therefore the definitions are trivially equivalent.

We end section \ref{secedge} with lemma \ref{lemEdge2}, which further
clarifies the equivalence of the above definitions. First recall
that $\mu \le \l$ (see remark \ref{remmu}), and note that
$\R \setminus \supp(\mu)$ and $\R \setminus \supp(\l-\mu)$ are disjoint
open sets. Equation (\ref{eqR}) thus gives
\begin{equation*}
R = (\R \setminus \supp(\mu)) \cup (\R \setminus \supp(\l-\mu)) \cup R_1 \cup R_2,
\end{equation*}
where
\begin{itemize}
\item
$R_1$ is the set of all $t \in \R$ for which there exists
an $\e>0$ with $(t,t+\e) \subset \R \setminus \supp(\mu)$
and $(t-\e,t) \subset \R \setminus \supp(\l-\mu)$.
\item
$R_2$ is the set of all $t \in \R$ for which there exists
an $\e>0$ with $(t,t+\e) \subset \R \setminus \supp(\l-\mu)$
and $(t-\e,t) \subset \R \setminus \supp(\mu)$.
\end{itemize}
Intuitively, one can think of $R_1$ as the set of all points
where the density of $\mu$ jumps from $1$ (on the left) to
$0$ (on the right), and $R_2$ as the set of all points where
it jumps from $0$ to $1$.
Next, fix $t \in R$, and define,
\begin{equation*}
(\chi,\eta) := (\chi_\EE(t),\eta_\EE(t))
\hspace{0.25cm} \text{and} \hspace{0.25cm}
f_t' := f_{(\chi, \eta)}',
\end{equation*}
where the last term is the function given in equation (\ref{eqf'}).
The equivalence of definitions \ref{defEdge0} and \ref{defEdge},
discussed above, implies that $t$ is a root of $f_t'$.
Then, letting $C : \C \setminus \supp(\mu) \to \C$ be
the Cauchy transform of $\mu$ (see remark \ref{remCauTrans}),
and letting $m_t\ge1$ denote the multiplicity of $t$ as
a root of $f_t'$, lemma \ref{lemEdge2} implies that the
following 9 cases exhaust all possibilities:
\begin{enumerate}
\item
$(\chi,\eta) \in \EE_\mu$,
$t \in \R \setminus \supp(\mu)$ with $C(t) \neq 0$, and $m_t=2$.
\item
$(\chi,\eta) \in \EE_\mu$,
$t \in \R \setminus \supp(\mu)$ with $C(t) \neq 0$, and $m_t=3$.
\item
$(\chi,\eta) \in \EE_{\l-\mu}$,
$t \in \R \setminus \supp(\l-\mu)$, and $m_t=2$.
\item
$(\chi,\eta) \in \EE_{\l-\mu}$,
$t \in \R \setminus \supp(\l-\mu)$, and $m_t=3$.
\item
$(\chi,\eta) \in \EE_0$,
$t \in \R \setminus \supp(\mu)$ with $C(t) = 0$, and $m_t=1$.
\item
$(\chi,\eta) \in \EE_1$, $t \in R_1$, and $m_t=1$.
\item
$(\chi,\eta) \in \EE_1$, $t \in R_1$, and $m_t=2$.
\item
$(\chi,\eta) \in \EE_2$, $t \in R_2$, and $m_t=1$.
\item
$(\chi,\eta) \in \EE_2$, $t \in R_2$, and $m_t=2$.
\end{enumerate}

In section \ref{seclgpotee} we investigate the local geometric properties
of the edge curve. First, fix $t \in R$, and define the (un-normalised)
orthogonal vectors, $\mathbf{x} = \mathbf{x}(t)$ and
$\mathbf{y} = \mathbf{y}(t)$, as in lemma \ref{lemLocGeo}. Also define
$(\chi,\eta) := (\chi_\EE(t),\eta_\EE(t))$. A Taylor
expansion of the edge curve then gives (see lemma \ref{lemLocGeo}),
\begin{equation*}
(\chi_\EE(s),\eta_\EE(s)) - (\chi,\eta)
= a(s) \mathbf{x} + b(s) \mathbf{y},
\end{equation*}
for all $s \in R$ sufficiently close to $t$, where
\begin{align*}
a(s) &= a_1 (s-t) + a_2 (s-t)^2 + O((s-t)^3), \\
b(s) &= b_1 (s-t)^2 + b_2 (s-t)^3 + O((s-t)^4),
\end{align*}
and $a_1 = a_1(t)$, $a_2 = a_2(t)$, $b_1 = b_1(t)$ and $b_2 = b_2(t)$
are known. Next, we investigate $a_1, a_2, b_1$
and $b_2$ for each of the exhaustive cases, (1-9) discussed above.
In lemma \ref{lemLineCusp} we show that:
\begin{itemize}
\item
$a_1 \neq 0$ and $b_1 \neq 0$ in cases (1, 3, 5, 6, 8).
\item
$a_1 = b_1 = 0$, $a_2 \neq 0$ and $b_2 \neq 0$ in cases (2, 4, 7, 9).
\end{itemize}
The above Taylor expansion then implies that the edge curve behaves like
a parabola in a neighbourhood of $(\chi,\eta)$ in cases
(1, 3, 5, 6, 8), with tangent vector $\mathbf{x}$ and normal vector
$\mathbf{y}$. Also the edge curve behaves like an algebraic cusp of
first order in a neighbourhood of $(\chi,\eta)$ in cases
(2, 4, 7, 9), and the vector $\mathbf{x}$ can be said to define the
`orientation' of the cusp. Also, since cases (1-9) are exhaustive, no
other behaviour is possible.

In the paper, \cite{Duse15c}, we use steepest descent techniques to
examine the edge asymptotic behaviour for cases (1-4). Recall that in
case (1), $t \in \R \setminus \supp(\mu)$ is a root
of $f_t'$ of multiplicity $2$, and the edge curve behaves locally like
a parabola in a neighbourhood of $(\chi_\EE(t),\eta_\EE(t))$. In case (2),
$t \in \R \setminus \supp(\mu)$ is a root of $f_t'$ of multiplicity $3$,
and the edge curve behaves like an algebraic cusp of first order in a
neighbourhood of $(\chi_\EE(t),\eta_\EE(t))$. In \cite{Duse15c}, we
confirm universal edge asymptotic behaviour for these cases: As $n \to \infty$,
choosing the parameters $(u_n,r_n)$ and $(v_n,s_n)$ of hypothesis
\ref{hypunrnsnvn} appropriately, $K_n((u_n,r_n),(v_n,s_n))$ converges to
the {\em Airy} kernel in case (1), and the {\em Pearcey} kernel in case
(2). We also show a similar result in cases (3)
and (4), except now the asymptotic behaviour of the correlation
kernel of the `holes' is examined, rather than that of the particles.

In the paper, \cite{Duse15d}, we use steepest descent techniques
to examine the edge asymptotic behaviour for cases (7) and (9).
In these cases, $t \in R_1 \cup R_2$ is a root of $f_t'$ of
multiplicity $2$, and the edge curve behaves like an algebraic
cusp of first order. Normally, edge universality implies the Pearcey
point process at cusps, but this does not occur in these cases.
As stated at the end of section \ref{secTilingsHalf},
we obtain a novel point process, which we call the
{\em Cusp-Airy} process.

Finally, we consider some examples of the measure, $\mu$,
of hypothesis \ref{hypWeakConv}. Letting $\varphi : \R \to [0,1]$ denote
the density of $\mu$, we consider:
\begin{enumerate}
\item[(a)]
$\varphi(x) = \frac12$ for all $x \in [-1,1]$.
\item[(b)]
$\varphi(x) = \frac12$ for all $x \in [0,1] \cup [2,3]$.
\item[(c)]
$\varphi(x) = 1$ for all $x \in [0,\frac12] \cup [1,\frac32]$.
\item[(d)]
$\varphi (x) = 1$ for all $x \in [0,\frac13] \cup [1,\frac43] \cup [c,c+\frac13]$,
where $c := \frac1{12} (23 + \sqrt{217}) > \frac43$.
\item[(e)]
$\varphi (x) = 1-x$ for all $x \in [0,1]$, and $\varphi (x) = 1+x$ for all
$x \in [-1,0]$.
\item[(f)]
$\varphi (x) = \frac{15}{16} (x-1)^2 (x+1)^2$ for all $x \in [-1,1]$.
\end{enumerate}
For all other values of $x \in \R$ in the above examples, we define
$\varphi(x) := 0$. These examples are examined in detail in section
\ref{secEx}. For each example, we give explicit expressions for the
edge curve (see definition \ref{defEdge0}). Moreover, we identify those
sections of $\partial \LL$ obtained from the results of section
\ref{sectbotlr1}. Finally, we identify which of the exhaustive
cases, (1-9) above, exist for each example. We summarise the results
of section \ref{secEx}, below.

Consider example (a). We state that the results of section \ref{sectbotlr1}
give a complete description of $\partial \LL$ in this case. This is depicted
on the top left of figure \ref{figexs}. Moreover, the edge is that part of
$\partial \LL$ excluding the (closed) straight line between $(-1,1)$ and
$(1,1)$, and the point of tangency with the lower boundary. Finally, all points of the edge
satisfy case (1) of (1-9). See section \ref{secEx1} for more details, and figure
\ref{figex1} for a more detailed depiction of $\partial \LL$. Example (a)
arises, for example, when we restrict $n$ in hypothesis \ref{hypWeakConv}
to be odd, and take
\begin{equation*}
x^{(n)} := \{n-1, n-3, \ldots, 2,0,-2, \ldots, -(n-3), -(n-1) \},
\end{equation*}
for all such $n$. In words, every second particle position in the top row
of the Gelfand-Tsetlin patterns is occupied.

Consider example (b). We state that the results of section \ref{sectbotlr1}
give a complete description of $\partial \LL$ in this case. This is depicted
on the top right of figure \ref{figexs}. Moreover, the edge is that part of
$\partial \LL$ excluding the (closed) straight line between $(0,1)$ and
$(1,1)$, the (closed) straight line between $(2,1)$ and $(3,1)$, and the
point of tangency with the lower boundary. Finally, the cusps in the edge satisfy case (2) of
(1-9), the point of tangency with the upper boundary satisfies case (5), and all other points of
the edge satisfy case (1). See section \ref{secEx2} for more details, and
figure \ref{figex2} for a more detailed depiction of $\partial \LL$.

\begin{figure}
\centering
\mbox{\includegraphics{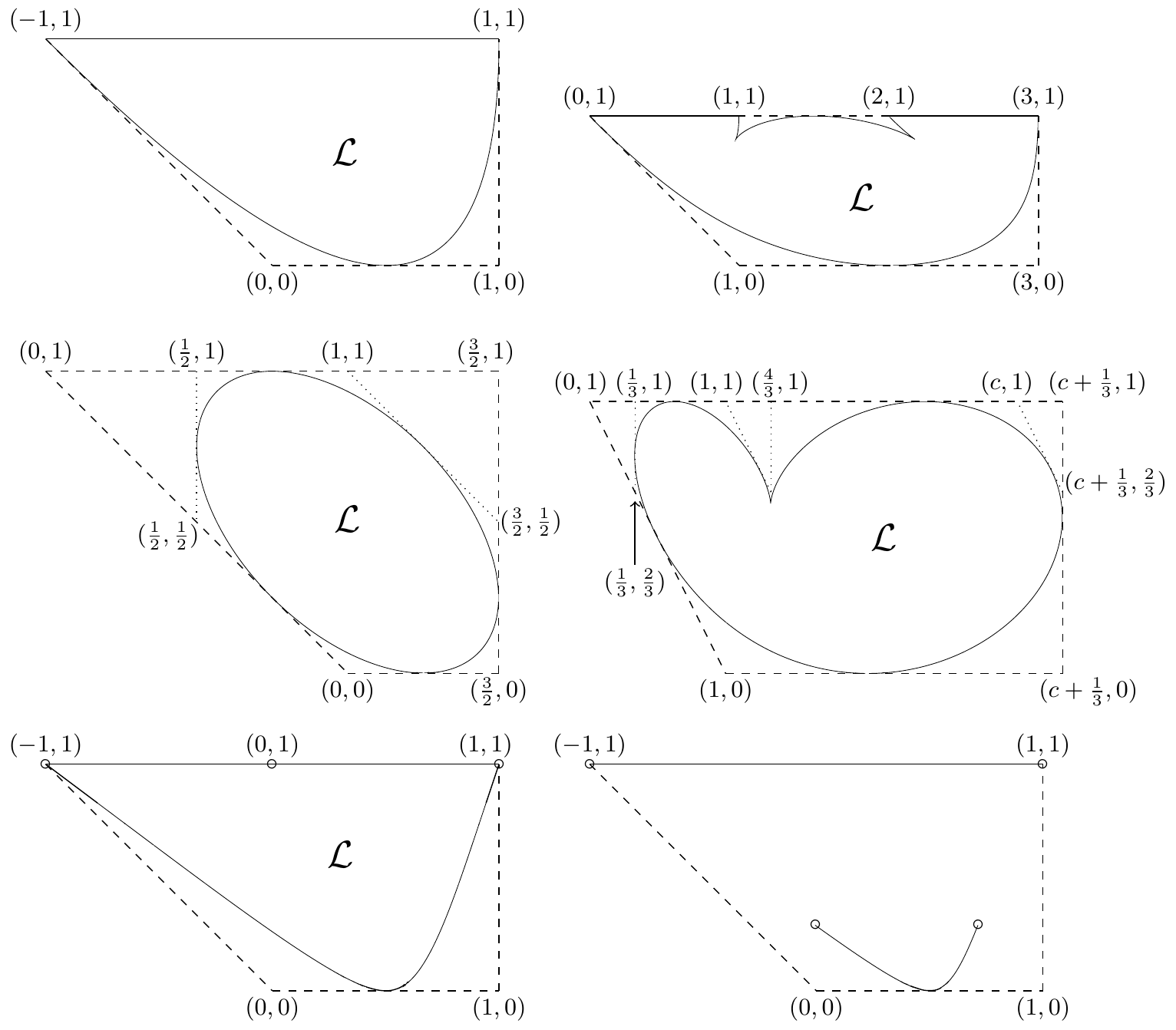}}
\caption{Those sections of $\partial \LL$ obtained from the results of
section \ref{sectbotlr1} for the examples, (a-f), defined above.
The dashed lines represent the shape in the middle of figure \ref{figGTAsyShape},
and the solid lines represent $\partial \LL$. Example (a) is on
the top left, (b) is on the top right, (c) is in the middle left,
(d) is in the middle right, (e) is on the bottom left, and (f) is on
the bottom right. In example (d), the vertical direction has been
scaled by 2 for clarity.}
\label{figexs}
\end{figure}

Consider example (c). We state that the results of section \ref{sectbotlr1}
give a complete description of $\partial \LL$ in this case. This is depicted
in the middle left of figure \ref{figexs}. Moreover, the edge is that part of
$\partial \LL$ excluding only the point of tangency with the lower boundary.
Finally, the edge contains examples from cases (1,3,5,6,8). See section
\ref{secEx3} for more details, and figure \ref{figex3} for a more detailed
depiction of $\partial \LL$. Example (c) arises, for example, when
we restrict $n$ in hypothesis \ref{hypWeakConv} to be even, and take
\begin{equation*}
x^{(n)} := \bigg\{ \frac{3n}2, \frac{3n}2-1, \frac{3n}2-2, \ldots, n+1 \bigg\}
\bigcup \bigg\{\frac{n}2, \frac{n}2-1, \frac{n}2-2, \ldots, 1\bigg\},
\end{equation*}
for all such $n$. In words, the particles on the top row of the
Gelfand-Tsetlin patterns exist in $2$ densely packed blocks.
This is the situation for the random systems of Gelfand-Tsetlin patterns
which equivalently describe random tilings of the regular
hexagon (see sections \ref{secTilingsHalf} and \ref{sectdsodgtp}).
Finally note that the shifted asymptotic shape of the frozen
boundary of the regular hexagon (see right of figure \ref{figGTAsyShape})
is identical to $\partial \LL$ as shown below.

Consider example (d). We state that the results of section \ref{sectbotlr1}
give a complete description of $\partial \LL$ in this case. This is depicted
in the middle right of figure \ref{figexs}. Moreover, the edge is that part of
$\partial \LL$ excluding only the point of tangency with the lower boundary. Finally, the edge
contains examples from cases (1,3,5,6,7,8). See section \ref{secEx4}
for more details, and figure \ref{figex4} for a more detailed depiction of
$\partial \LL$. In particular, we emphasise that the cusp in the edge
satisfies case (7) of (1-9). Thus, as stated above, the asymptotic behaviour
at this cusp is not governed by the Pearcey point process, but by the novel
{\em Cusp-Airy} process. We prove this result in the paper \cite{Duse15d}.

Consider example (e). We state that the results of section \ref{sectbotlr1}
do not give a complete description of $\partial \LL$ in this case. More
exactly, the points $\{-1,0,1\} \subset \supp(\mu)$ do not satisfy any of
the conditions of lemma \ref{lemBddEdge2}. Consequently, we show in section
\ref{secEx5} that the results of section \ref{sectbotlr1} give those parts
of $\partial \LL$ shown on the bottom left of figure \ref{figexs}, excluding
the circled points. We then show via direct calculation that we get a complete
description of $\partial \LL$ by adding these points. The edge is that part
of $\partial \LL$ excluding the (closed) straight line between $(-1,1)$ and
$(1,1)$, and the point of tangency with the lower boundary. Finally, all points of the edge
satisfy case (1). See section \ref{secEx5} for more details, and figure
\ref{figex5} for a more detailed depiction of $\partial \LL$.

Consider example (f). We state that the results of section \ref{sectbotlr1}
do not give a complete description of $\partial \LL$ in this case. More
exactly, the points $\{-1,1\} \subset \supp(\mu)$ do not satisfy any of
the conditions of lemma \ref{lemBddEdge2}. Consequently, we show in
section \ref{secEx6} that the results of section \ref{sectbotlr1} give
only those parts of $\partial \LL$ shown on the bottom right of figure
\ref{figexs}, excluding the circled points. This is clearly not
a complete description of $\partial \LL$, since $\LL$ is a connected
set. A complete description of $\partial \LL$, in this case, is beyond
the scope of this paper, since the technicalities involved in extending
lemma \ref{lemBddEdge2} to cover this situation are highly non-trivial.
In the paper \cite{Duse15b}, we make heavy use of the theory of
singular integrals to examine this and other, surprisingly subtle,
situations. The edge, in this case, is the lower part of $\partial \LL$
excluding the circled points, and the point of tangency with the lower boundary. Moreover, all
points of the edge satisfy case (1). See section \ref{secEx6} for more
details, and figure \ref{figex6} for a more detailed depiction of
$\partial \LL$.

We end this section by noting that none of the above examples have
points of the edge which satisfy either cases (4) or (9) of the
exhaustive cases, (1-9), listed above. However, we note that case
(4) is of a similar nature to case (2), and case (9) is of a similar
nature to case (7).

\section{Geometry}
\label{secGeo}

In this section we consider the geometric properties of the liquid region
given in definition \ref{defLiq}.

\subsection{The liquid region, $\LL$}

Recall (see definition \ref{defLiq}) that the liquid region, $\LL$,
is the set of all $(\chi,\eta) \in [a,b] \times [0,1]$ with
$b \ge \chi \ge \chi + \eta - 1 \ge a$, for which the following
function has non-real roots (see equation (\ref{eqf'0})):
\begin{equation}
\label{eqf'2}
f_{(\chi,\eta)}'(w) = C(w) + \log(w - \chi) - \log(w-\chi-\eta+1),
\end{equation}
for all $w \in \C \setminus \R$, where $\log$ is principal value and $C$ is
the Cauchy transform of $\mu$,
\begin{equation}
\label{eqCauTrans}
C(w) := \int_a^b \frac{\mu[dx]}{w-x}.
\end{equation}
We denote this simply by $f'$ where no confusion is possible.
First note that non-real roots of $f'$ occur in complex conjugate pairs. Theorem
\ref{thmf'} then implies that $(\chi,\eta) \in \LL$ if and only if $f'$ has
exactly $2$ roots in $\C \setminus \R$, counting multiplicities. More exactly
there are $2$ roots of multiplicity $1$, a unique root in
$\mathbb{H} := \{ w \in \C : \text{Im}(w) > 0 \}$, and its complex
conjugate.

\begin{thm}
\label{thmwc}
Let $W_\LL : \LL \to \mathbb{H}$ map $(\chi,\eta) \in \LL$
to the corresponding root of $f'$ in $\mathbb{H}$. This is a homeomorphism with
inverse $w \mapsto (\chi_\LL(w),\eta_\LL(w))$ for all $w \in \mathbb{H}$, given by,
\begin{equation*}
\chi_\LL(w) := w + \frac{(w - \bar{w}) (e^{C(\bar{w})}-1)}{e^{C(w)} - e^{C(\bar{w})}}
\hspace{0.25cm} \text{and} \hspace{0.25cm}
\eta_\LL(w) :=
1 + \frac{(w - \bar{w}) (e^{C(w)}-1) (e^{C(\bar{w})}-1) }{e^{C(w)} - e^{C(\bar{w})}},
\end{equation*}
where $\bar{w}$ is the complex conjugate of $w$.
\end{thm}

\begin{proof}
We first show:
\begin{enumerate}
\item
$\LL$ is non-empty.
\item
$\LL$ is open.
\item
$W_\LL : \LL \to \mathbb{H}$ is continuous.
\item
$W_\LL : \LL \to \mathbb{H}$ is injective.
\end{enumerate}
The {\em invariance of domain theorem} then implies that $W_\LL(\LL)$ is open
and $W_\LL : \LL \to W_\LL(\LL)$ is a homeomorphism. We complete
the result by showing:
\begin{enumerate}
\setcounter{enumi}{4}
\item
$W_\LL : \LL \to W_\LL(\LL)$ has inverse
$w \mapsto (\chi_\LL(w),\eta_\LL(w))$ for all $w \in W_\LL(\LL)$.
\item
$W_\LL(\LL) = \mathbb{H}$.
\end{enumerate}

Consider (1). Fixing $w \in \mathbb{H}$ and defining
$(\chi,\eta) := (\chi_\LL(w),\eta_\LL(w))$, where $\chi_\LL$ and $\eta_\LL$
are defined as in the statement of the lemma, we show that: 
\begin{enumerate}
\item[(1a)]
$f'(w) = 0$.
\item[(1b)]
$(\chi,\eta) \in [a,b] \times [0,1]$ with
$b \ge \chi \ge \chi + \eta - 1 \ge a$ whenever $|w|$ is chosen to be
sufficiently large.
\end{enumerate}
Thus $(\chi,\eta) \in \LL$ by definition whenever $|w|$ is chosen to be
sufficiently large, as required.

Consider (1a). First note, since $(\chi,\eta) = (\chi_\LL(w),\eta_\LL(w))$,
the definitions of $\chi_\LL$ and $\eta_\LL$ give $w - \chi - \eta + 1 =
(w - \chi) e^{C(w)}$. Equation (\ref{eqf'2}) then gives
$f'(w) = C(w) + \log(w - \chi) - \log((w - \chi) e^{C(w)})$,
where $\log$ is principal value. (1a) thus follows if we can show that
$\log((w - \chi) e^{C(w)}) = \log(w - \chi) + \log(e^{C(w)})$. We prove
this by showing that $\text{Arg}(w-\chi) \in (0,\pi)$
and $\text{Arg}(e^{C(w)}) \in (-\pi,0)$.

First note that it is trivial to see that $\text{Arg}(w-\chi) \in (0,\pi)$,
since $w \in \mathbb{H}$ and $\chi \in \R$. Next note that
equation (\ref{eqCauTrans}) gives,
\begin{equation*}
\text{Im}(C(w)) = - \int_a^b \frac{\text{Im}(w) \mu[dx]}{(\text{Re}(w)-x)^2 + \text{Im}(w)^2}.
\end{equation*}
Thus, since $\text{Im}(w) > 0$, and since $\mu \le \l$ (see remark \ref{remmu}),
\begin{equation*}
0 > \text{Im}(C(w)) > - \int_{-\infty}^\infty
\frac{\text{Im}(w) dx}{(\text{Re}(w)-x)^2 + \text{Im}(w)^2} = -\pi.
\end{equation*}
Therefore $\text{Arg}(e^{C(w)}) = \text{Im}(C(w)) \in (-\pi,0)$. This proves (1a).

Consider (1b). Recalling that $\chi = \chi_\LL(w)$ and $\eta = \eta_\LL(w)$, write
\begin{align*}
\chi & =
\frac{w(e^{\frac12(C(w) - C(\bar{w}))}-1) - \bar{w}(e^{-\frac12(C(w) - C(\bar{w}))}-1)
- (w-\bar{w}) (e^{-\frac12(C(w) + C(\bar{w}))}-1)}{2\sinh(\frac12(C(w) - C(\bar{w})))}, \\
\eta & =
1 + \frac{2 (w - \bar{w}) \sinh(\frac12 C(w)) \sinh(\frac12 C(\bar{w})) }
{\sinh(\frac12(C(w) - C(\bar{w})))}.
\end{align*}
Also Taylor expansions of the Cauchy transform in equation (\ref{eqCauTrans}) give
\begin{align*}
C(w)
&= \frac1{w} + \frac{\mu_1}{w^2} + \frac{\mu_2}{w^3} + O \left( |w|^{-4} \right), \\
C(w) - C(\bar{w})
&= \left( \frac1{w} - \frac1{\bar{w}} \right)
\left( 1 + \mu_1 \left( \frac1{w} + \frac1{\bar{w}} \right)  +
\mu_2 \left( \frac1{w^2} + \frac1{|w|^2} + \frac1{\bar{w}^2} \right)
+ O \left( |w|^{-3} \right) \right),
\end{align*}
where $\mu_1 := \int_a^b x \mu[dx]$ and $\mu_2 := \int_a^b x^2 \mu[dx]$. Combine the
above to get
\begin{equation*}
\chi = \mu_1 + \frac12 + O \left( |w|^{-1} \right)
\hspace{0.5cm} \text{and} \hspace{0.5cm}
\eta = \left( \mu_2 - \mu_1^2 - \frac1{12} \right) \frac1{|w|^2}
+ O \left( |w|^{-3} \right).
\end{equation*}
Finally recall that hypothesis \ref{hypWeakConv} and remark \ref{remmu}
imply that $\mu$ is a probability measure on $[a,b]$, that $b-a>1$, and
that $\mu \le \l$. Also $a \in \supp(\mu)$, and so $\mu[a,a+\e] > 0$
for all $\e > 0$. Therefore,
\begin{equation*}
\mu_1
= \int_a^b x \mu[dx]
< \int_{b-1}^b x dx
= b-\frac12.
\end{equation*}
Similarly $b \in \supp(\mu)$, and so
\begin{equation*}
\mu_1
= \int_a^b x \mu[dx]
> \int_a^{a+1} x dx
= a+\frac12.
\end{equation*}
Similarly,
\begin{equation*}
\mu_2 - \mu_1^2 = \frac12 \iint_a^b (x-y)^2 \mu[dx] \mu[dy] >
\frac12 \iint_0^1 (x-y)^2 dx dy = \frac1{12}.
\end{equation*}
Combine the above to get $(\chi,\eta) \in (a+1,b) \times (0,1)$
whenever $|w|$ is chosen be to sufficiently large. This proves (1b).

Consider (2). Fix $(\chi_1,\eta_1) \in \LL$ and
$(\chi_2,\eta_2) \in [a,b] \times [0,1]$ and define
\begin{itemize}
\item 
$f_1'(w) := C(w) + \log(w - \chi_1) - \log(w-\chi_1-\eta_1+1)$,
\item
$f_2'(w) := C(w) + \log(w - \chi_2) - \log(w-\chi_2-\eta_2+1)$,
\end{itemize}
for all $w \in \C \setminus \R$, where $\log$ is principal value. Let
$w_1 := W_\LL(\chi_1,\eta_1)$, where $W_\LL$ is defined in the statement of the lemma.
Then $w_1$ is the unique root of $f_1'$ in $\mathbb{H}$. We now show that
$(\chi_2,\eta_2) \in \LL$ whenever $|\chi_1-\chi_2|$ and $|\eta_1-\eta_2|$
are sufficiently small, i.e., that $f_2'$ has a root in $\mathbb{H}$ for all
such $(\chi_2,\eta_2)$. Fix $\e>0$ for which $B(w_1,\e) \subset \mathbb{H}$. Then,
since $w_1$ is the unique root of $f_1'$ in $\mathbb{H}$, the extreme value theorem
gives,
\begin{equation*}
\inf_{w \in \partial B(w_1,\e)} |f_1'(w)| > 0.
\end{equation*}
Also $|f_1'(w) - f_2'(w)| \le |\log(w - \chi_1) - \log(w - \chi_2)| + |\log(w-\chi_1-\eta_1+1)
- \log(w-\chi_2-\eta_2+1)|$. Thus, whenever $|\chi_1-\chi_2|$ and $|\eta_1-\eta_2|$ are
sufficiently small, $|f_1'(w)| > |f_1'(w) - f_2'(w)|$ for all
$w \in \partial B(w_1,\e)$. Rouch\'{e}s Theorem thus implies that $f_2'$ has a root
in $B(w_1,\e) \subset \mathbb{H}$.

Consider (3). Consider the same setup used in step (2). Recall that
$w_1 = W_\LL(\chi_1,\eta_1)$ and that $f_2'$ has a root in
$B(w_1,\e) \subset \mathbb{H}$ whenever $|\chi_1-\chi_2|$ and $|\eta_1-\eta_2|$
are sufficiently small. The above definition of $W_\LL$ then implies that
$W_\LL(\chi_2,\eta_2)$ is the observed root of $f_2'$, and so
$|W_\LL(\chi_1,\eta_1) - W_\LL(\chi_2,\eta_2)| < \e$ whenever $|\chi_1-\chi_2|$
and $|\eta_1-\eta_2|$ are sufficiently small. Finally note that we can repeat
the same analysis with $\e$ replaced by any $\d \in (0,\e)$. Therefore $W_\LL$
is continuous, as required.

Consider (4). Fix $(\chi_1,\eta_1), (\chi_2,\eta_2) \in \LL$ with
$W_\LL(\chi_1,\eta_1) = W_\LL(\chi_2,\eta_2) = w \in \mathbb{H}$.
Equation (\ref{eqf'2}) and the above definition of $W_\LL$ then gives
\begin{equation*}
C(w)
= - \log(w - \chi_1) + \log(w-\chi_1-\eta_1+1)
= - \log(w - \chi_2) + \log(w-\chi_2-\eta_2+1).
\end{equation*}
Exponentiating and simplifying gives
$(\eta_2-\eta_1) w = (1-\eta_1) \chi_2 - (1-\eta_2) \chi_1$.
Then $w \in \R$ whenever $\eta_1 \neq \eta_2$, which contradicts
$w \in \mathbb{H}$. Thus $\eta_1 = \eta_2$. Also, recalling that
$\LL$ is open, $\eta_1 = \eta_2 \in (0,1)$. This finally gives
$\chi_1 = \chi_2$, as required.

Consider (5). Fix $(\chi,\eta) \in \LL$ and let
$w := W_\LL(\chi,\eta) \in W_\LL(\LL)$. Equation (\ref{eqf'2}) and
the above definition of $W_\LL$ then give
$C(w) + \log(w - \chi) - \log(w-\chi-\eta+1) = 0$. Exponentiating and
simplifying gives $1-\eta = (w-\chi) (e^{C(w)}-1)$. Complex conjugation gives,
\begin{equation*}
1-\eta = (w-\chi) (e^{C(w)}-1) = (\bar{w}-\chi) (e^{C(\bar{w})}-1).
\end{equation*}
Solving gives $(\chi,\eta) = (\chi_\LL(w),\eta_\LL(w))$, as required.

Consider (6). Recall that $W_\LL(\LL)$ is open and that
$W_\LL : \LL \to W_\LL(\LL)$ is a homeomorphism with inverse
$w \mapsto (\chi_\LL(w),\eta_\LL(w))$. Assume that $W_\LL(\LL)$ is a
proper subset of $\mathbb{H}$, i.e., that there exists a point
$w \in \partial W_\LL(\LL)$ with $w \in \mathbb{H} \setminus W_\LL(\LL)$.
Choose a sequence $\{w_n\}_{n\ge1} \subset W_\LL(\LL)$ with $w_n \to w$ as
$n \to \infty$, and let $(\chi_n,\eta_n) := (\chi_\LL(w_n),\eta_\LL(w_n))$
for all $n\ge1$. Note that we can always choose so that
$\{(\chi_n,\eta_n)\}_{n\ge1}$ is convergent as $n \to \infty$,
$(\chi_n,\eta_n) \to (\chi,\eta)$ say. Also note equation (\ref{eqf'2})
and the above definition of $W_\LL$ gives
$C(w_n) + \log(w_n-\chi_n) - \log(w_n-\chi_n-\eta_n+1) = 0$
for all $n\ge1$. Letting $n \to \infty$ we get
$C(w) + \log(w-\chi) - \log(w-\chi-\eta+1) = 0$, and so
$(\chi,\eta) \in \LL$ and $w = W_\LL(\chi,\eta)$. This contradicts
the assumption that $w \in \mathbb{H} \setminus W_\LL(\LL)$, and so
$W_\LL(\LL) = \mathbb{H}$, as required.
\end{proof}

It can furthermore be shown that $W_\LL : \LL \to \mathbb{H}$ is a
diffeomorphism (see section \ref{secDiff}).

\subsection{The boundary of the liquid region, $\partial \LL$}
\label{sectbotlr1}

Note the following consequence of theorem \ref{thmwc}:
\begin{cor}
\label{corBndryLim}
$\LL$ is a non-empty, open, simply connected set in the shape in the middle
of figure \ref{figGTAsyShape}. Moreover, $\partial \LL$ is the set of all $(\chi,\eta)$
for which there exists a sequence, $\{w_n\}_{n\ge1} \subset \mathbb{H}$, with
$(\chi_\LL(w_n),\eta_\LL(w_n)) \to (\chi,\eta)$ as $n \to \infty$,
and either $|w_n| \to \infty$ or $w_n \to t \in \R$ as $n \to \infty$.
\end{cor}
In this section we use the above result to examine $\partial \LL$. First note:
\begin{lem}
\label{lemBddEdge0}
$(\frac12 + \int_a^b x \mu[dx],0) \in \partial \LL$. Moreover
$(\chi_\LL(w_n),\eta_\LL(w_n)) \to (\frac12 + \int_a^b x \mu[dx],0)$ as
$n \to \infty$ for all $\{w_n\}_{n\ge1} \subset \mathbb{H}$ with
$|w_n| \to \infty$.
\end{lem}

\begin{proof}
Fix $\{w_n\}_{n\ge1} \subset \mathbb{H}$ with
$|w_n| \to \infty$ as $n \to \infty$. The proof
of step (1b) of theorem \ref{thmwc} then gives
$(\chi_\LL(w_n),\eta_\LL(w_n)) \to (\frac12 + \int_a^b x \mu[dx],0)$,
and corollary \ref{corBndryLim} gives
$(\frac12 + \int_a^b x \mu[dx],0) \in \partial \LL$.
\end{proof}

Next recall that $\mu \le \l$ (see remark \ref{remmu}), and note that
$\R \setminus \supp(\mu)$ and $\R \setminus \supp(\l-\mu)$ are disjoint
open sets. Also recall (see equation (\ref{eqR})) that
\begin{equation*}
R := (\; \overline{(\R \setminus \supp(\mu)) \cup (\R \setminus \supp(\l-\mu))} \;)^\circ.
\end{equation*}
Hypothesis \ref{hypWeakConv} and remark \ref{remmu} then give,
\begin{equation}
\label{eqR2}
R = (\R \setminus \supp(\mu)) \cup (\R \setminus \supp(\l-\mu)) \cup R_1 \cup R_2,
\end{equation}
where,
\begin{itemize}
\item
$R_1$ is the set of all
$t \in \partial (\R \setminus \supp(\mu)) \cap \partial (\R \setminus \supp(\l-\mu))$
for which there exists an interval, $I := (t_2,t_1)$, with $t \in I$,
$(t,t_1) \subset \R \setminus \supp(\mu)$ and $(t_2,t) \subset \R \setminus \supp(\l-\mu)$.
\item
$R_2$ is the set of all
$t \in \partial (\R \setminus \supp(\mu)) \cap \partial (\R \setminus \supp(\l-\mu))$
for which there exists an interval, $I := (t_2,t_1)$, with $t \in I$,
$(t,t_1) \subset \R \setminus \supp(\l-\mu)$ and $(t_2,t) \subset \R \setminus \supp(\mu)$.
\end{itemize}
Indeed, $R_1 \cap R_2 = \emptyset$ and $R_1 \cup R_2 = \partial (\R \setminus \supp(\mu))
\cap \partial (\R \setminus \supp(\l-\mu))$. Finally, note the following technical result:
\begin{lem}
\label{lemAnalExt}
Let $C : \C \setminus \R \to \C$ denote the Cauchy transform of $\mu$ (see equation
(\ref{eqCauTrans})). Then,
\begin{enumerate}
\item[(a)]
$C : \C \setminus \R \to \C$ has a unique analytic extension to
$\C \setminus \supp(\mu)$. Moreover, also denoting the extension
by $C$,
\begin{equation*}
C(w) = \int_a^b \frac{\mu[dx]}{w-x},
\end{equation*}
for all $w \in \C \setminus \supp(\mu)$. Finally, $e^{C(t)} > 0$
and $C'(t) < 0$ for all $t \in \R \setminus \supp(\mu)$.
\item[(b)]
$e^{C(\cdot)} : \C \setminus \R \to \C$ and
$C' : \C \setminus \R \to \C$ have unique analytic extensions to
$\C \setminus \supp(\l-\mu)$. Moreover, also denoting the extensions
by $e^{C(\cdot)}$ and $C'$, and fixing any interval
$I = (t_2,t_1) \subset \R \setminus \supp(\l-\mu)$,
\begin{equation*}
e^{C(w)} = e^{C_I(w)} \left( \frac{w-t_2}{w-t_1} \right)
\hspace{0.5cm} \text{and} \hspace{0.5cm}
C'(w) = C_I'(w) - \frac1{w-t_1} + \frac1{w-t_2},
\end{equation*}
for all $w \in (\C \setminus \R) \cup I$, where
$C_I(w) := \int_{[a,b] \setminus I} \frac{\mu[dx]}{w-x}$. Finally,
$e^{C(t)} < 0$, $C'(t) > 0$ and $\frac{d}{dw} e^{C(w)} |_{w=t} = e^{C(t)} C'(t)$
for all $t \in \R \setminus \supp(\l-\mu)$.
\item[(c)]
Fix $t \in R_1$, and $I = (t_2,t_1)$ with $t \in I$,
$(t,t_1) \subset \R \setminus \supp(\mu)$ and
$(t_2,t) \subset \R \setminus \supp(\l-\mu)$ (see equation (\ref{eqR2})).
Then $e^{-C(\cdot)} : \C \setminus \R \to \C$ has a unique analytic
extension to $(\C \setminus \R) \cup I$. Moreover, also denoting the extension
by $e^{-C(\cdot)}$,
\begin{equation*}
e^{-C(w)} = e^{-C_I(w)} \left( \frac{w-t}{w-t_2} \right),
\end{equation*}
for all $w \in (\C \setminus \R) \cup I$, where
$C_I(w) := \int_{[a,b] \setminus I} \frac{\mu[dx]}{w-x}$. Finally, $e^{-C(t)} = 0$
and $\frac{d}{dw} e^{-C(w)} |_{w=t} = e^{-C_I(t)} (t-t_2)^{-1}$.
\item[(d)]
Fix $t \in R_2$, and $I = (t_2,t_1)$ with $t \in I$,
$(t,t_1) \subset \R \setminus \supp(\l-\mu)$ and
$(t_2,t) \subset \R \setminus \supp(\mu)$ (see equation (\ref{eqR2})).
Then $e^{C(\cdot)} : \C \setminus \R \to \C$ has a unique analytic
extension to $(\C \setminus \R) \cup I$. Moreover, also denoting the extension
by $e^{C(\cdot)}$,
\begin{equation*}
e^{C(w)} = e^{C_I(w)} \left( \frac{w-t}{w-t_1} \right),
\end{equation*}
for all $w \in (\C \setminus \R) \cup I$, where
$C_I(w) := \int_{[a,b] \setminus I} \frac{\mu[dx]}{w-x}$. Finally,
$e^{C(t)} = 0$ and $\frac{d}{dw} e^{C(w)} |_{w=t} = e^{C_I(t)} (t-t_1)^{-1}$.
\end{enumerate}
\end{lem}

\begin{proof}
Consider (a). The required analytic extension easily follows from equation
(\ref{eqCauTrans}). Also,
\begin{equation*}
C(w) = \int_a^b \frac{\mu[dx]}{w-x}
\hspace{0.5cm} \text{and} \hspace{0.5cm}
C'(w) = - \int_a^b \frac{\mu[dx]}{(w-x)^2},
\end{equation*}
for all $w \in \C \setminus \supp(\mu)$. Thus $C(t) \in \R$, $e^{C(t)} > 0$
and $C'(t) < 0$ for all $t \in \R \setminus \supp(\mu)$.

Consider (b). Fixing $I = (t_2,t_1) \subset \R \setminus \supp(\l-\mu)$,
equation (\ref{eqCauTrans}) gives $C(w) = C_I(w) - \log(w-t_1) + \log(w-t_2)$
for all $w \in \C \setminus \R$, where $\log$ is principal value and 
$C_I(w) = \int_{[a,b] \setminus I} \frac{\mu[dx]}{w-x}$. The required
analytic extensions easily follow. Also, it easily follows from the expressions
of the extensions that $e^{C(t)} < 0$ and
$\frac{d}{dw} e^{C(w)} |_{w=t} = e^{C(t)} C'(t)$ for all $t \in I$. It remains
to show that $C'(t) < 0$ for all $t \in I$. Note, for all such $t$,
\begin{equation*}
C'(t) = - \left( \int_a^{t_2} + \int_{t_1}^b \right) \frac{\mu[dx]}{(t-x)^2}
- \frac1{t-t_1} + \frac1{t-t_2}.
\end{equation*}
Then, recalling that $\mu \le \l$ (see remark \ref{remmu}),
\begin{equation*}
C'(t)
\ge - \left( \int_{t_2-p_2}^{t_2} + \int_{t_1}^{t_1+p_1} \right) \frac{dx}{(t-x)^2}
- \frac1{t-t_1} + \frac1{t-t_2},
\end{equation*}
where $p_1 := \mu[t_1,b]$ and $p_2 := \mu[a,t_2]$. Integrating finally gives,
\begin{equation*}
C'(t) \ge \frac1{t-t_2+p_2} - \frac1{t-t_1-p_1}
= \frac{-1}{(t-t_2+p_2)(t-t_1-p_1)},
\end{equation*}
where the final part follows since $p_1 + p_2 + (t_1-t_2)
= \mu[t_1,b] + \mu[a,t_2] + \mu[t_2,t_1] = 1$. Therefore
$C'(t) > 0$ for all $t \in I = (t_2,t_1)$.

Consider (c). Fixing $t \in R_1$ and $I = (t_2,t_1)$ as in the statement,
equation (\ref{eqCauTrans}) gives $C(w) = C_I(w) - \log(w-t) + \log(w-t_2)$
for all $w \in \C \setminus \R$, where $\log$ is principal value. The required
analytic extension easily follows. Also, it easily follows from the expression
of the extension that $e^{-C(t)} = 0$ and
$\frac{d}{dw} e^{-C(w)} |_{w=t} = e^{-C_I(t)} (t-t_2)^{-1}$. Similarly for (d).
\end{proof}

The various analytic extensions of the previous result then give:
\begin{lem}
\label{lemBddEdge1}
$(\chi_\EE(t),\eta_\EE(t)) \in \partial \LL$ for all
$t \in (\R \setminus \supp(\mu)) \cup (\R \setminus \supp(\l-\mu))$, where
\begin{align*}
\chi_\EE(t) := t + \frac{e^{C(t)}-1}{e^{C(t)} C'(t)}
\hspace{0.5cm} \text{and} \hspace{0.5cm}
\eta_\EE(t) := 1 + \frac{(e^{C(t)}-1)^2}{e^{C(t)} C'(t)}.
\end{align*}
Also $(\chi_\EE(t),\eta_\EE(t)) \in \partial \LL$ for all $t \in R_1 \cup R_2$,
where
\begin{itemize}
\item
$\chi_\EE(t) := t$ and
$\eta_\EE(t) := 1 - e^{C_I(t)} (t-t_2)$ for all $t \in R_1$,
\item
$\chi_\EE(t) := t - e^{-C_I(t)} (t-t_1)$ and
$\eta_\EE(t) := 1 + e^{-C_I(t)} (t-t_1)$ for all $t \in R_2$.
\end{itemize}
Moreover, $(\chi_\LL(w_n),\eta_\LL(w_n)) \to (\chi_\EE(t),\eta_\EE(t))$ as
$n \to \infty$ for all $t \in R$ and $\{w_n\}_{n\ge1} \subset \mathbb{H}$ with
$w_n \to t$. (Above, whenever $ t \in R_1 \cup R_2$, $I := (t_2,t_1)$ is
chosen as in lemma \ref{lemAnalExt},
$C_I(w) := \int_{[a,b] \setminus I} \frac{\mu[dx]}{w-x}$,
and the result is independent of the choice of $I$.)
\end{lem}

\begin{proof}
Fix $t \in R$ and $\{w_n\}_{n\ge1} \subset \mathbb{H}$ with $w_n \to t$ as
$n \to \infty$. We shall show that
$(\chi_\LL(w_n),\eta_\LL(w_n)) \to (\chi_\EE(t),\eta_\EE(t))$ as
$n \to \infty$. Corollary \ref{corBndryLim} then gives
$(\chi_\EE(t),\eta_\EE(t)) \in \partial \LL$, as required.

When $t \in \R \setminus \supp(\mu)$, write (see theorem \ref{thmwc}),
\begin{align*}
\chi_\LL(w_n) &= w_n + (e^{C(\overline{w_n})}-1)
\frac{w_n - \overline{w_n}}{e^{C(w_n)} - e^{C(\overline{w_n})}}, \\
\eta_\LL(w_n) &= 1 + (e^{C(w_n)}-1) (e^{C(\overline{w_n})}-1)
\frac{w_n - \overline{w_n}}{e^{C(w_n)} - e^{C(\overline{w_n})}},
\end{align*}
for all $n$. Note, as $n \to \infty$, the analytic extension of part
(a) of lemma \ref{lemAnalExt} gives, 
\begin{equation*}
e^{C(w_n)} \to e^{C(t)},
\hspace{0.5cm}
e^{C(\overline{w_n})} \to e^{C(t)}
\hspace{0.5cm} \text{and} \hspace{0.5cm}
\frac{e^{C(w_n)} - e^{C(\overline{w_n})}}{w_n - \overline{w_n}}
\to e^{C(t)} C'(t).
\end{equation*}
Therefore $(\chi_\LL(w_n),\eta_\LL(w_n)) \to (\chi_\EE(t),\eta_\EE(t))$
when $t \in \R \setminus \supp(\mu)$. Similarly when
$t \in \R \setminus \supp(\l-\mu)$, except now we use the analytic
extensions of part (b) of lemma \ref{lemAnalExt}.

Also write $\chi_\LL(w_n)$ and $\eta_\LL(w_n)$ as above when $t \in R_2$.
When $t \in R_1$, write,
\begin{align*}
\chi_\LL(w_n) &=
w_n - e^{-C(w_n)} (1-e^{-C(\overline{w_n})})
\frac{w_n - \overline{w_n}}{e^{-C(w_n)} - e^{-C(\overline{w_n})}}, \\
\eta_\LL(w_n) &= 1 - (1-e^{-C(w_n)}) (1-e^{-C(\overline{w_n})})
\frac{w_n - \overline{w_n}}{e^{-C(w_n)} - e^{-C(\overline{w_n})}},
\end{align*}
for all $n$. The analytic extensions of parts (c) and (d) of lemma
\ref{lemAnalExt} then easily give
$(\chi_\LL(w_n),\eta_\LL(w_n)) \to (\chi_\EE(t),\eta_\EE(t))$ as
$n \to \infty$ when $t \in R_1 \cup R_2$, as required.
\end{proof}

\begin{rem}
\label{remSmoothPara}
For emphasis we note that $R \subset \R$ is an open set, and that lemma
\ref{lemBddEdge1} implies that
$(\chi_\EE(\cdot),\eta_\EE(\cdot)) : R \to \partial \LL$ is the unique
continuous extensions to $R$ of
$(\chi_\LL(\cdot),\eta_\LL(\cdot)) : \mathbb{H} \to \LL$. Thus
$(\chi_\EE(\cdot),\eta_\EE(\cdot))$
is a smooth curve parameterised over $R$. In definition
\ref{defEdge0} we defined the edge, $\EE \subset \partial \LL$, as the
image of this curve. This is why we have chosen
the subscript $\EE$. Definition \ref{defEdge} equivalently defines the
edge in terms of the behaviour of the roots of the function $f'$ of
equation (\ref{eqf'}). This equivalence is shown in
corollary \ref{corEdgeDefEquiv} of theorem \ref{thmEdge}.
\end{rem}

It remains to consider the asymptotic behaviour of $(\chi_\LL(w_n),\eta_\LL(w_n))$
as $n \to \infty$, when $t \in \R \setminus R$ and
$\{w_n\}_{n\ge1} \subset \mathbb{H}$ with $w_n \to t$. This question
is surprisingly subtle however, and will be examined in greater detail
in \cite{Duse15b}. Lemma \ref{lemBddEdge2}, below, is a sub-result
of that paper. However, lemmas \ref{lemBddEdge0}, \ref{lemBddEdge1} and
\ref{lemBddEdge2} give a complete description of $\partial \LL$
when the measure, $\mu$, of hypothesis \ref{hypWeakConv} is restricted
to an interesting sub-class of all possible measures (see lemma
\ref{lemBddEdge3}). In section \ref{secEx} we depict $\LL$ and
$\partial \LL$ for a number of examples in this sub-class.

\begin{lem}
\label{lemBddEdge2}
$(t,1) \in \partial \LL$ for all $t \in \R \setminus R =
\supp(\mu) \cap \supp(\l-\mu) \cap (\R \setminus (R_1 \cup R_2))$
whenever there exists an $\e>0$ for which one of the following
cases is satisfied:
\begin{enumerate}
\item
$\sup_{x \in (t-\e,t+\e)} \varphi (x) < 1$ and $\inf_{x \in (t-\e,t+\e)} \varphi (x) > 0$.
\item
$\sup_{x \in (t-\e,t)} \varphi (x) < 1$, $\inf_{x \in (t-\e,t)} \varphi (x) > 0$
and $\varphi(x) = 0$ for all $x \in (t,t+\e)$.
\item
$\sup_{x \in (t-\e,t)} \varphi (x) < 1$, $\inf_{x \in (t-\e,t)} \varphi (x) > 0$
and $\varphi(x) = 1$ for all $x \in (t,t+\e)$.
\item
$\sup_{x \in (t,t+\e)} \varphi (x) < 1$, $\inf_{x \in (t,t+\e)} \varphi (x) > 0$
and $\varphi(x) = 0$ for all $x \in (t-\e,t)$.
\item
$\sup_{x \in (t,t+\e)} \varphi (x) < 1$, $\inf_{x \in (t,t+\e)} \varphi (x) > 0$
and $\varphi(x) = 1$ for all $x \in (t-\e,t)$.
\end{enumerate}
Above $\varphi : [a,b] \to [0,1]$ denotes the density of $\mu$
(remark \ref{remmu} shows that $\varphi$ is well-defined). Moreover
$(\chi_\LL(w_n),\eta_\LL(w_n)) \to (t,1)$ as $n \to \infty$ for all
$\{w_n\}_{n\ge1} \subset \mathbb{H}$ with $w_n \to t$.
\end{lem}

\begin{proof}
Fix $t \in \R \setminus R$ which satisfies one of the cases (1-5) in
the statement of the lemma. Also fix $\e>0$ as in the statement, and
$\{w_n\}_{n\ge1} \subset \mathbb{H}$ with $w_n \to t$ as $n \to \infty$.
Denote $u_n := \text{Re}(w_n)$, $v_n := \text{Im}(w_n)$,
$R_n := \text{Re}(C(w_n))$, and $I_n := - \text{Im}(C(w_n))$, where $C$ is
the Cauchy transform of $\mu$ (see equation (\ref{eqCauTrans})). Therefore,
\begin{equation}
\label{eqlemBddEdge30}
R_n = \int_a^b \frac{(u_n-x) \varphi(x) dx}{(u_n-x)^2 + v_n^2}
\hspace{0.5cm} \text{and} \hspace{0.5cm}
I_n = \int_a^b \frac{v_n \varphi(x) dx}{(u_n-x)^2 + v_n^2},
\end{equation}
for all $n$. Note that $u_n \to t$ and $v_n \searrow 0$ as $n \to \infty$,
and that $\pi > I_n > 0$ for all $n$. Also write (see theorem \ref{thmwc}),
\begin{align}
\label{eqlemBddEdge31}
\chi_\LL(w_n)
&= u_n - \frac{v_n \cos(I_n) - v_n e^{-R_n}}{\sin(I_n)}, \\
\label{eqlemBddEdge32}
\eta_\LL(w_n)
&= 1 - \frac{v_n e^{R_n} - 2 v_n \cos(I_n) + v_n e^{-R_n}}{\sin(I_n)},
\end{align}
for all $n$. We shall show that $(\chi_\LL(w_n),\eta_\LL(w_n)) \to (t,1)$ as
$n \to \infty$. Corollary \ref{corBndryLim} then gives $(t,1) \in \partial \LL$,
as required. We consider cases (1) and (2), and note (3-5) are similar.

Consider (1). Letting
$\varphi^+ := \sup_{x \in (t-\e,t+\e)} \varphi (x) < 1$ and
$\varphi^- := \inf_{x \in (t-\e,t+\e)} \varphi (x) > 0$, we show
that there exists a constant $c>0$ such that, for all $n$ sufficiently large,
\begin{enumerate}
\item[(1a)]
$\pi - \frac{\pi}2 (1-\varphi^+) > I_n > \frac{\pi}2 \varphi^-$.
\item[(1b)]
$v_n e^{|R_n|} < c v_n^{1 - \varphi^+ + \varphi^-}$.
\end{enumerate}
Equations (\ref{eqlemBddEdge31}) and (\ref{eqlemBddEdge32}) then easily
give $(\chi_\LL(w_n),\eta_\LL(w_n)) \to (t,1)$ as $n \to \infty$, as required.

Consider (1a). Recall that $\varphi(x) \in [0,1]$ for all $x \in [a,b]$,
$\varphi^+ = \sup_{x \in (t-\e,t+\e)} \varphi (x) < 1$
and $\varphi^- = \inf_{x \in (t-\e,t+\e)} \varphi (x) > 0$. Defining
$\varphi(x) := 0$ for all $x \in \R \setminus [a,b]$, equation
(\ref{eqlemBddEdge30}) then gives,
\begin{align*}
I_n
&\le \int_{a-\e}^{t-\e} \frac{v_n (1) dx}{(u_n-x)^2 + v_n^2}
+ \int_{t-\e}^{t+\e} \frac{v_n (\varphi^+) dx}{(u_n-x)^2 + v_n^2}
+ \int_{t+\e}^{b+\e} \frac{v_n (1) dx}{(u_n-x)^2 + v_n^2} \\
&= - \left. \arctan \left( \frac{u_n-x}{v_n} \right) \right|_{a-\e}^{t-\e}
- \varphi^+ \left. \arctan \left( \frac{u_n-x}{v_n} \right) \right|_{t-\e}^{t+\e}
- \left. \arctan \left( \frac{u_n-x}{v_n} \right) \right|_{t+\e}^{b+\e},
\end{align*}
for all $n$. Thus, since $u_n \to t \in [a,b]$ and $v_n \searrow 0$ as $n \to \infty$,
$I_n \le \pi \varphi^+ + O(v_n)$. Similarly,
\begin{equation*}
I_n
\ge \int_{a-\e}^{t-\e} (0) dx
+ \int_{t-\e}^{t+\e} \frac{v_n (\varphi^-) dx}{(u_n-x)^2 + v_n^2}
+ \int_{t+\e}^{b+\e} (0) dx,
\end{equation*}
for all $n$. Proceed as before to get $I_n \ge \pi \varphi^- + O(v_n)$.
(1a) follows since $\varphi^+ < 1$ and $\varphi^- > 0$.

Consider (1b). Recall that $\varphi : \R \to [0,1]$, $\varphi(x) = 0$ for all
$x \in \R \setminus [a,b]$, $\varphi^+ = \sup_{x \in (t-\e,t+\e)} \varphi (x) < 1$
and $\varphi^- = \inf_{x \in (t-\e,t+\e)} \varphi (x) > 0$. Then, choosing
$n$ sufficiently large that $u_n \in (t-\e,t+\e)$, equation
(\ref{eqlemBddEdge30}) gives,
\begin{align*}
\lefteqn{R_n} \\
&\le \int_{a-\e}^{t-\e} \frac{(u_n-x) (1) dx}{(u_n-x)^2 + v_n^2}
+ \int_{t-\e}^{u_n} \frac{(u_n-x) (\varphi^+) dx}{(u_n-x)^2 + v_n^2}
+ \int_{u_n}^{t+\e} \frac{(u_n-x) (\varphi^-) dx}{(u_n-x)^2 + v_n^2}
+ \int_{t+\e}^{b+\e} (0) dx \\
&= - \left. \frac12 \log ( (u_n-x)^2 + v_n^2 ) \right|_{a-\e}^{t-\e}
- \left. \frac{\varphi^+}2 \log ( (u_n-x)^2 + v_n^2 ) \right|_{t-\e}^{u_n}
- \left. \frac{\varphi^-}2 \log ( (u_n-x)^2 + v_n^2 ) \right|_{u_n}^{t+\e}.
\end{align*}
Thus, since $u_n \to t$ and $v_n \searrow 0$ as $n \to \infty$,
$R_n \le -(\varphi^+ - \varphi^-) \log(v_n) + O(1)$. Similarly,
\begin{equation*}
R_n
\ge \int_{a-\e}^{t-\e} (0) dx
+ \int_{t-\e}^{u_n} \frac{(u_n-x) (\varphi^-) dx}{(u_n-x)^2 + v_n^2}
+ \int_{u_n}^{t+\e} \frac{(u_n-x) (\varphi^+) dx}{(u_n-x)^2 + v_n^2}
+ \int_{t+\e}^{b+\e} \frac{(u_n-x) (1) dx}{(u_n-x)^2 + v_n^2},
\end{equation*}
for all $n$ sufficiently large. Proceed as before to get
$R_n \ge (\varphi^+ - \varphi^-) \log(v_n) + O(1)$. Combining both
inequalities gives $|R_n| \le -(\varphi^+ - \varphi^-) \log(v_n) + O(1)$
for all $n$ sufficiently large, and (1b) then easily follows.

Consider (2). Now, letting
$\varphi^+ := \sup_{x \in (t-\e,t)} \varphi (x) < 1$ and
$\varphi^- := \inf_{x \in (t-\e,t)} \varphi (x) > 0$, we show
that there exists a constant $c>0$ such that, for all $n$ sufficiently large,
\begin{enumerate}
\item[(2a)]
$\pi - \frac{\pi}2 (1-\varphi^+) > I_n > \frac14 \varphi^-$ when $u_n-t \le v_n$,
and $\pi - \frac{\pi}2 (1-\varphi^+) > I_n > \frac14 \varphi^- \frac{v_n}{u_n-t}$
when $u_n-t \ge v_n$.
\item[(2b)]
$v_n e^{|R_n|} < c v_n^{1 - \varphi^+}$ when $u_n-t \le v_n$,
and $v_n e^{|R_n|} < c v_n (u_n-t)^{- \varphi^+}$ when $u_n-t \ge v_n$.
\end{enumerate}
Also note that $\frac14 > \frac14 \varphi^- \frac{v_n}{u_n-t} > 0$
when $u_n-t \ge v_n$, and so $\sin (\frac14 \varphi^- \frac{v_n}{u_n-t})
> \frac18 \varphi^- \frac{v_n}{u_n-t}$. Thus, since $u_n \to t$ and
$v_n \searrow 0$ as $n \to \infty$, equations (\ref{eqlemBddEdge31}) and
(\ref{eqlemBddEdge32}) give $(\chi_\LL(w_n),\eta_\LL(w_n)) \to (t,1)$,
as required.

Consider (2a). Recall that $\varphi(x) \in [0,1]$ for all $x \in [a,b]$,
$\varphi^+ = \sup_{x \in (t-\e,t)} \varphi (x) < 1$,
$\varphi^- = \inf_{x \in (t-\e,t)} \varphi (x) > 0$,
$\varphi(x) = 0$ for all $x \in (t,t+\e)$, and $\varphi(x) = 0$ for all
$x \in \R \setminus [a,b]$. Equation (\ref{eqlemBddEdge30}) then gives,
for all $n$,
\begin{align*}
I_n
&\le \int_{a-\e}^{t-\e} \frac{v_n (1) dx}{(u_n-x)^2 + v_n^2}
+ \int_{t-\e}^t \frac{v_n (\varphi^+) dx}{(u_n-x)^2 + v_n^2}
+ \int_{t+\e}^{b+\e} \frac{v_n (1) dx}{(u_n-x)^2 + v_n^2}, \\
I_n
&\ge \int_{t-\e}^t \frac{v_n (\varphi^-) dx}{(u_n-x)^2 + v_n^2}.
\end{align*}
Then, since $u_n \to t \in [a,b]$ and $v_n \searrow 0$ as
$n \to \infty$, we can proceed as in (1a) to get,
\begin{equation*}
\frac{\pi}2 \varphi^+ - \varphi^+ \arctan \left( \frac{u_n-t}{v_n} \right) + O(v_n)
\;\; \ge \;\; I_n \;\; \ge \;\;
\frac{\pi}2 \varphi^- - \varphi^- \arctan \left( \frac{u_n-t}{v_n} \right) + O(v_n),
\end{equation*}
for all $n$. Recall that $\arctan : \R \to \R$ is strictly increasing with
$\arctan(x) \in (-\frac{\pi}2,\frac{\pi}2)$ for all $x \in \R$ and
$\arctan(x) < \frac{\pi}2 - \frac1{2x}$ for all $x \ge 1$.
Therefore,
\begin{align*}
\pi \varphi^+ + O(v_n)
&\;\; \ge \;\; I_n \;\; \ge \;\;
\varphi^- \left( \frac{\pi}2 - \arctan (1) \right) + O(v_n)
&\text{ when } u_n-t \le v_n, \\
\pi \varphi^+ + O(v_n)
&\;\; \ge \;\; I_n \;\; \ge \;\;
\frac12 \varphi^- \frac{v_n}{u_n-t} + O(v_n)
&\text{ when } u_n-t \ge v_n,
\end{align*}
(2a) follows since $\varphi^+ < 1$ and $\varphi^- > 0$, and
since $u_n \to t$ and $v_n \searrow 0$ as $n \to \infty$.

Consider (2b). Recall that $\varphi(x) \in [0,1]$ for all $x \in [a,b]$,
$\varphi^+ = \sup_{x \in (t-\e,t)} \varphi (x) < 1$,
$\varphi^- = \inf_{x \in (t-\e,t)} \varphi (x) > 0$,
$\varphi(x) = 0$ for all $x \in (t,t+\e)$, and $\varphi(x) = 0$ for all
$x \in \R \setminus [a,b]$. Choose $n$ sufficiently large that
$u_n \in (t-\e,t+\e)$. Then, when $u_n\le t$, equation (\ref{eqlemBddEdge30})
gives,
\begin{align*}
R_n
&\le \int_{a-\e}^{t-\e} \frac{(u_n-x) (1) dx}{(u_n-x)^2 + v_n^2}
+ \int_{t-\e}^{u_n} \frac{(u_n-x) (\varphi^+) dx}{(u_n-x)^2 + v_n^2}
+ \int_{u_n}^t \frac{(u_n-x) (\varphi^-) dx}{(u_n-x)^2 + v_n^2}, \\
R_n
&\ge \int_{t-\e}^{u_n} \frac{(u_n-x) (\varphi^-) dx}{(u_n-x)^2 + v_n^2}
+ \int_{u_n}^t \frac{(u_n-x) (\varphi^+) dx}{(u_n-x)^2 + v_n^2}
+ \int_{t+\e}^{b+\e} \frac{(u_n-x) (1) dx}{(u_n-x)^2 + v_n^2}.
\end{align*}
Also, when $u_n \ge t$,
\begin{align*}
R_n
&\le \int_{a-\e}^{t-\e} \frac{(u_n-x) (1) dx}{(u_n-x)^2 + v_n^2}
+ \int_{t-\e}^t \frac{(u_n-x) (\varphi^+) dx}{(u_n-x)^2 + v_n^2}, \\
R_n
&\ge \int_{t-\e}^t \frac{(u_n-x) (\varphi^-) dx}{(u_n-x)^2 + v_n^2}
+ \int_{t+\e}^{b+\e} \frac{(u_n-x) (1) dx}{(u_n-x)^2 + v_n^2}.
\end{align*}
Then, since $u_n \to t \in [a,b]$ and $v_n \searrow 0$ as
$n \to \infty$, we can proceed as in (1b) to get,
\begin{align*}
|R_n|
&\le -(\varphi^+ - \varphi^-) \log(v_n)
- \frac12 \varphi^- \log((u_n-t)^2 + v_n^2) + O(1)
&\text{ when } u_n \le t, \\
|R_n|
&\le - \frac12 \varphi^+ \log((u_n-t)^2 + v_n^2) + O(1)
&\text{ when } u_n \ge t.
\end{align*}
Finally note that $\frac12 \log((u_n-t)^2 + v_n^2) \ge \log(v_n)$
and $\frac12 \log((u_n-t)^2 + v_n^2) \ge \log|u_n-t|$. Therefore,
for all $n$ sufficiently large, $|R_n| \le -\varphi^+ \log(v_n) + O(1)$
when $u_n -t \le v_n$, and $|R_n| \le - \varphi^+ \log(u_n-t) + O(1)$
when $u_n -t \ge v_n$. (2b) then easily follows.
\end{proof}

We end this section with the following trivial result:
\begin{lem}\leavevmode
\label{lemBddEdge3}
\begin{enumerate}
\item
Assume that the measure, $\mu$, of hypothesis \ref{hypWeakConv} is of the form
$\mu = \sum_k \l_{A_i}$, where $\{A_1, A_2, ... \}$ are mutually disjoint
intervals. Then $R = \R$, and lemmas \ref{lemBddEdge0} and \ref{lemBddEdge1}
completely describe
$\partial \LL$.
\item
Assume that $\mu$ is such that $\R \setminus R$ is non-empty, and each
$t \in \R \setminus R$ satisfies one of the cases (1-5) of lemma
\ref{lemBddEdge2}. Then lemmas \ref{lemBddEdge0}, \ref{lemBddEdge1} and
\ref{lemBddEdge2} completely describe $\partial \LL$.
\end{enumerate}
\end{lem}

\begin{proof}
The fact that $\R = R$ in part (1) follows from equation (\ref{eqR}).
Next, recall that corollary \ref{corBndryLim} implies that
$\partial \LL$ is the set of all $(\chi,\eta)$ for which there exists a
$\{w_n\}_{n\ge1} \subset \mathbb{H}$ with
$(\chi_\LL(w_n),\eta_\LL(w_n)) \to (\chi,\eta)$ as $n \to \infty$,
and either $|w_n| \to \infty$ or $w_n \to t \in \R$ as $n \to \infty$.
Lemma \ref{lemBddEdge0} covers all sequences with $|w_n| \to \infty$,
and lemma \ref{lemBddEdge1} covers all sequences with $w_n \to t \in R$.
Finally, in part (2), lemma \ref{lemBddEdge2} covers all sequences with
$w_n \to t \in \R \setminus R$, since each such $t$ satisfies one of the
cases of lemma \ref{lemBddEdge2} by assumption. The result trivially
follows.
\end{proof}

\subsection{The edge, $\EE$}
\label{secedge}

In this section we prove an analogous result for the edge, $\EE$, to theorem
\ref{thmwc} for the liquid region, $\LL$. Recall in theorem \ref{thmwc}, we
map $\LL$ to $\mathbb{H}$ by mapping $(\chi,\eta) \in \LL$ to the
unique root of $f'_{(\chi,\eta)}$ in $\mathbb{H}$. This map is a
homeomorphism with inverse $w \mapsto (\chi_\LL(w),\eta_\LL(w))$. Recall also
(see remark \ref{remSmoothPara}), that
$(\chi_\EE(\cdot),\eta_\EE(\cdot)) : R \to \partial \LL$ is the smooth
curve, parameterised over the open set $R \subset \R$ (see equations (\ref{eqR})
and (\ref{eqR2})), which is the unique continuous extensions to $R$ of
$(\chi_\LL(\cdot),\eta_\LL(\cdot)) : \mathbb{H} \to \LL$. The main
result of this section, theorem \ref{thmEdge}, uses definition \ref{defEdge}
for $\EE$ to define a map from $\EE$ to $\R$ which is analogous to the map
from $\LL$ to $\mathbb{H}$ discussed above. This is shown to
bijectively map $\EE$ to $R$ with inverse $t \mapsto (\chi_\EE(t),\eta_\EE(t))$.
Then, in corollary \ref{corEdgeDefEquiv}, we use theorem \ref{thmEdge} to prove the
equivalence of definitions \ref{defEdge0} and \ref{defEdge} for $\EE$.
Lemma \ref{lemEdge2} further explores this equivalence.

Again we denote $f'_{(\chi,\eta)}$ simply by $f'$ where no confusion is possible.
Recall definition \ref{defEdge}: $\EE$ is the disjoint union,
$\EE := \EE_\mu \cup \EE_{\l-\mu} \cup \EE_0 \cup \EE_1 \cup \EE_2$, where
\begin{itemize}
\item
$\EE_\mu$ is the set of all $(\chi,\eta) \in [a,b] \times [0,1]$
with $b \ge \chi \ge \chi+\eta-1 \ge a$ for which $f'$ has a
repeated root in $\R \setminus [\chi+\eta-1,\chi]$.
\item
$\EE_{\l-\mu}$ is the set of all $(\chi,\eta)$ for which $f'$ has a
repeated root in $(\chi+\eta-1,\chi)$.
\item
$\EE_0$ is the set of all $(\chi,\eta)$ for which $\eta=1$ and $f'$ has a root
at $\chi \; (=\chi+\eta-1)$.
\item
$\EE_1$ is the set of all $(\chi,\eta)$ for which $\eta<1$ and $f'$ has a root at $\chi$.
\item
$\EE_2$ is the set of all $(\chi,\eta)$ for which $\eta<1$ and $f'$ has a root
at $\chi+\eta-1$.
\end{itemize}
The fact that the above union is disjoint follows from corollary \ref{corf'} of
theorem \ref{thmf'}. Also, theorem \ref{thmf'} implies that $f'$ has at most one
real-valued repeated root.

\begin{thm}
\label{thmEdge}
Define $W_\EE : \EE \to \R$ by mapping to the corresponding root of $f'$:
\begin{itemize}
\item
$W_\EE(\chi,\eta)$ is the unique real-valued repeated root whenever
$(\chi,\eta) \in \EE_\mu \cup \EE_{\l-\mu}$.
\item
$W_\EE(\chi,\eta)$ is the root $\chi$ whenever $(\chi,\eta) \in \EE_0 \cup \EE_1$.
\item
$W_\EE(\chi,\eta)$ is the root $\chi+\eta-1$ whenever $(\chi,\eta) \in \EE_0 \cup \EE_2$.
\end{itemize}
Then $W_\EE(\EE) = R$ and $W_\EE : \EE \to R$ is a bijection with inverse
$t \mapsto (\chi_\EE(t),\eta_\EE(t))$.
\end{thm}

\begin{proof}
Let $C : \C \setminus \supp(\mu) \to \C$ be the analytic extension of
the Cauchy transform defined in part (a) of lemma \ref{lemAnalExt}.
Also define $R_1$ and $R_2$ as in equation (\ref{eqR2}). We show:
\begin{enumerate}
\item
$W_\EE(\EE_\mu) = R_\mu := \{t \in \R \setminus \supp(\mu) : C(t) \neq 0\}$
and $W_\EE : \EE_\mu \to R_\mu$ is a bijection with inverse
$t \mapsto (\chi_\EE(t),\eta_\EE(t))$.
\item
$W_\EE(\EE_{\l-\mu}) = R_{\l-\mu} := \R \setminus \supp(\l-\mu)$ and
$W_\EE : \EE_{\l-\mu} \to R_{\l-\mu}$ is a bijection with inverse
$t \mapsto (\chi_\EE(t),\eta_\EE(t))$.
\item
$W_\EE(\EE_0) = R_0 := \{t \in \R \setminus \supp(\mu) : C(t) = 0\}$ and
$W_\EE : \EE_0 \to R_0$ is a bijection with inverse
$t \mapsto (\chi_\EE(t),\eta_\EE(t))$.
\item
$W_\EE(\EE_1) = R_1$ and $W_\EE : \EE_1 \to R_1$ is a bijection with
inverse $t \mapsto (\chi_\EE(t),\eta_\EE(t))$.
\item
$W_\EE(\EE_2) = R_2$ and $W_\EE : \EE_2 \to R_2$ is a bijection with
inverse $t \mapsto (\chi_\EE(t),\eta_\EE(t))$.
\end{enumerate}
Note, equation (\ref{eqR2}) implies that $R$ is the disjoint union,
$R = R_\mu \cup R_{\l-\mu} \cup R_0 \cup R_1 \cup R_2$. Parts (1-5)
thus easily give the required result. In fact they are a stronger statement.

Consider (1). We prove this by showing:
\begin{enumerate}
\item[(1a)]
Fix $(\chi,\eta) \in \EE_\mu$ and let $t := W_\EE(\chi,\eta)$. Then
$t \in R_\mu$ and $(\chi,\eta) = (\chi_\EE(t),\eta_\EE(t))$.
\item[(1b)]
Fix $t \in R_\mu$ and let $(\chi,\eta) := (\chi_\EE(t),\eta_\EE(t))$. Then
$(\chi,\eta) \in \EE_\mu$ and $W_\EE(\chi,\eta) = t$.
\end{enumerate}

Consider (1a). Note, equation (\ref{eqf'}) and the definition of $\EE_\mu$ imply
that $t \in (\R \setminus [\chi+\eta-1,\chi]) \setminus \supp(\mu)$. Also note,
equations (\ref{eqf'}) and (\ref{eqCauTrans}) give,
\begin{equation}
\label{eqf'Rmu}
f'(w) = C(w) + \log(w - \chi) - \log(w-\chi-\eta+1),
\end{equation}
for all $w \in \C \setminus \R$, where $\log$ is principal value. This has a
trivial analytic extension to
$(\C \setminus [\chi+\eta-1,\chi]) \setminus \supp(\mu)$. Also, parts (a) and
(b) of corollary \ref{corf'} imply that $(\chi,\eta) \in (a,b) \times (0,1)$
and $b > \chi > \chi+\eta-1 > a$. Finally, since $t$ is a repeated
root of $f'$, $e^{f'(t)} = 1$ and $f''(t) = 0$. Therefore,
\begin{equation}
\label{eqlemEdge0}
e^{C(t)} \left( \frac{t-\chi}{t-\chi-\eta+1} \right) = 1
\hspace{0.5cm} \text{and} \hspace{0.5cm}
C'(t) + \frac1{t-\chi} - \frac1{t-\chi-\eta+1} = 0.
\end{equation}
The first part gives $C(t) \neq 0$, since $\eta<1$, and so
$t \in R_\mu$. Also, the first part gives,
\begin{equation*}
\frac1{t-\chi} = \frac{e^{C(t)} - 1}{1-\eta}
\hspace{0.5cm} \text{and} \hspace{0.5cm}
\frac1{t-\chi-\eta+1} = \frac{e^{C(t)} - 1}{(1-\eta) e^{C(t)}}.
\end{equation*}
Substitute into the second part and solve to get
$(\chi,\eta) = (\chi_\EE(t),\eta_\EE(t))$ (see lemma
\ref{lemBddEdge1}), as required.

Consider (1b). First note, lemma \ref{lemBddEdge1} implies that
$(\chi,\eta) \in \partial \LL$, and so
$(\chi,\eta) \in [a,b] \times [0,1]$ with
$b \ge \chi \ge \chi+\eta-1 \ge a$. Also, this lemma gives,
\begin{equation}
\label{eqlemEdge1}
\eta = 1 + \frac{(e^{C(t)}-1)^2}{e^{C(t)} C'(t)},
\hspace{0.5cm}
\chi =  t + \frac{e^{C(t)}-1}{e^{C(t)} C'(t)},
\hspace{0.5cm}
\chi+\eta-1 =  t + \frac{e^{C(t)}-1}{C'(t)}.
\end{equation}
Note that $e^{C(t)} \neq 1$ ($C(t) \neq 0$ since $t \in R_\mu$),
$e^{C(t)} >0$ and $C'(t) < 0$ (see part (a) of lemma (\ref{lemAnalExt})).
Thus the second term on the right hand side of the expression for
$\eta$ is negative, and so $\eta < 1$. Also, the second terms on the right
hand side of the expressions for $\chi$ and $\chi+\eta-1$ are non-zero and
have the same sign, and so $t \in \R \setminus [\chi+\eta-1,\chi]$.
Note, equation (\ref{eqf'Rmu}) holds as above,
for all $w \in (\C \setminus [\chi+\eta-1,\chi]) \setminus \supp(\mu)$.
Substitute the above expressions for $\chi$ and $\eta$ into $f'$ to get
$f'(t) = 0$. Similarly $f''(t) = 0$. Thus $(\chi,\eta) \in \EE_\mu$
by definition, and $W_\EE(\chi,\eta) = t$, as required.

Consider (2). We prove this by showing:
\begin{enumerate}
\item[(2a)]
Fix $(\chi,\eta) \in \EE_{\l-\mu}$ and let $t := W_\EE(\chi,\eta)$. Then
$t \in R_{\l-\mu}$ and $(\chi,\eta) = (\chi_\EE(t),\eta_\EE(t))$.
\item[(2b)]
Fix $t \in R_{\l-\mu}$ and let $(\chi,\eta) := (\chi_\EE(t),\eta_\EE(t))$. Then
$(\chi,\eta) \in \EE_{\l-\mu}$ and $W_\EE(\chi,\eta) = t$.
\end{enumerate}

Consider (2a). Note, equation (\ref{eqf'}) and the definition of $\EE_{\l-\mu}$
give $t \in (\chi+\eta-1,\chi) \setminus \supp(\l-\mu) \subset R_{\l-\mu}$,
and so it remains to show that $(\chi,\eta) = (\chi_\EE(t),\eta_\EE(t))$.
Fix $I = (t_2,t_1) \subset (\chi+\eta-1,\chi) \setminus \supp(\l-\mu)$
with $t \in I$. Equations (\ref{eqf'}) and (\ref{eqCauTrans}) then give,
\begin{equation}
\label{eqf'Rl-mu}
f'(w) = C_I(w) + \log(w-\chi) - \log(w-t_1) + \log(w-t_2) - \log(w-\chi-\eta+1),
\end{equation}
for all $w \in \C \setminus \R$, where $\log$ is principal value
and $C_I(w) := \int_{[a,b] \setminus I} \frac{\mu[dx]}{w-x}$. This has a
trivial analytic extension to $(\C \setminus \R) \cup I$.
Also, since $t$ is a repeated root of $f'$, $e^{f'(t)} = 1$ and $f''(t) = 0$.
Equation (\ref{eqlemEdge0}) again holds, where now $e^{C(t)}$ and $C'(t)$ are
defined by the analytic extensions of part (b) of lemma (\ref{lemAnalExt}).
We can then proceed as for $\EE_\mu$ to get
$(\chi,\eta) = (\chi_\EE(t),\eta_\EE(t))$, as required.

Consider (2b). First note, lemma \ref{lemBddEdge1} implies that
$(\chi,\eta) \in \partial \LL$, and so
$(\chi,\eta) \in [a,b] \times [0,1]$ with $b \ge \chi \ge \chi+\eta-1 \ge a$.
Equation (\ref{eqlemEdge1}) again holds, where now $e^{C(t)}$ and $C'(t)$
are defined by the analytic extensions of part (b) of lemma (\ref{lemAnalExt}).
Therefore $e^{C(t)} < 0$ and $C'(t) > 0$. Thus the second term on the right
hand side of the expression for $\eta$ is negative, and so $\eta < 1$.
Also the second term on the right hand side of the expression for $\chi$ is
positive, and the second term on the right hand side of the expression for
$\chi+\eta-1$ is negative, and so $t \in (\chi+\eta-1,\chi)$. Therefore,
fixing $I = (t_2,t_1) \subset (\chi+\eta-1,\chi) \setminus \supp(\l-\mu)$
with $t \in I$, equation (\ref{eqf'Rmu}) holds as above, for all
$w \in (\C \setminus \R) \cup I$. This easily gives $f'(t) = f''(t) = 0$.
Thus $(\chi,\eta) \in \EE_{\l-\mu}$ by definition, and $W_\EE(\chi,\eta) = t$, as
required.

Consider (3). We prove this by showing:
\begin{enumerate}
\item[(3a)]
Fix $(\chi,\eta) \in \EE_0$ and let $t := W_\EE(\chi,\eta)$. Then
$t \in R_0$ and $(\chi,\eta) = (\chi_\EE(t),\eta_\EE(t))$.
\item[(3b)]
Fix $t \in R_0$ and let $(\chi,\eta) := (\chi_\EE(t),\eta_\EE(t))$. Then
$(\chi,\eta) \in \EE_0$ and $W_\EE(\chi,\eta) = t$.
\end{enumerate}

Consider (3a). The definitions of $\EE_0$ and $W_\EE$ give
$(\chi,\eta) = (t,1)$. Moreover $f'$ has a root at $t=\chi$, and so equation
(\ref{eqf'}) gives $t \in \R \setminus \supp(\mu)$. Also equations (\ref{eqf'})
and (\ref{eqCauTrans}) give $f'(w) = C(w)$ for all $w \in \C \setminus \R$.
This has a trivial analytic extension to $\C \setminus \supp(\mu)$, and so
$f'(t) = C(t)$. Thus, since $f'$ has a root at $t=\chi$, $C(t) = 0$ and
lemma \ref{lemBddEdge1} gives $(\chi_\EE(t),\eta_\EE(t)) =(t,1)$. Thus
$t\in R_0$ and $(\chi,\eta) = (\chi_\EE(t),\eta_\EE(t))$, as required.
Consider (3b). The definition of $R_0$ gives $t \in \R \setminus \supp(\mu)$
and $C(t) = 0$. Lemma \ref{lemBddEdge1} thus gives $(\chi,\eta) = (t,1)$.
Then, as before, $f'(w) = C(w)$ for all $w \in \C \setminus \supp(\mu)$, and
so $f'(t) = C(t) = 0$. Therefore $(\chi,\eta) \in \EE_0$ by definition, and
$W_\EE(\chi,\eta) = t$, as required.

Consider (4). We prove this by showing:
\begin{enumerate}
\item[(4a)]
Fix $(\chi,\eta) \in \EE_1$ and let $t := W_\EE(\chi,\eta)$. Then
$t \in R_1$ and $(\chi,\eta) = (\chi_\EE(t),\eta_\EE(t))$.
\item[(4b)]
Fix $t \in R_1$ and let $(\chi,\eta) := (\chi_\EE(t),\eta_\EE(t))$. Then
$(\chi,\eta) \in \EE_1$ and $W_\EE(\chi,\eta) = t$.
\end{enumerate}

Consider (4a). Note, the definitions of $\EE_1$ and
$W_\EE$ give $\eta<1$ and $t = \chi$. Moreover, $f'$ has a root at $t=\chi$, and
so equation (\ref{eqf'}) shows that there exists an interval, $I = (t_2,t_1)$,
with $t = \chi \in I$, $(t,t_1) \subset \R \setminus \supp(\mu)$
and $(t_2,t) \subset (\chi+\eta-1,\chi) \setminus \supp(\l-\mu)$.
Equation (\ref{eqR2}) then implies that $t \in R_1$. Moreover, equations
(\ref{eqf'}) and (\ref{eqCauTrans}) give,
\begin{equation}
\label{eqf'R1}
f'(w) = C_I(w) + \log(w-t_2) - \log(w-\chi-\eta+1),
\end{equation}
for all $w \in \C \setminus \R$, where $\log$ is principal value. This has a
trivial analytic extension to $(\C \setminus \R) \cup I$. Thus, since $t = \chi$
and $f'(\chi) = 0$, $C_I(t) + \log(t-t_2) - \log(1-\eta) = 0$, where now $\log$
denotes natural logarithm. Solving gives $\eta = 1 - (t-t_2) e^{C_I(t)}$. Lemma
(\ref{lemBddEdge1}) then gives $(\chi,\eta) = (\chi_\EE(t),\eta_\EE(t))$, as
required.

Consider (4b). Note,
lemma \ref{lemBddEdge1} implies that $(\chi,\eta) \in \partial \LL$, and so
$(\chi,\eta) \in [a,b] \times [0,1]$ with $b \ge \chi \ge \chi+\eta-1 \ge a$.
Also, this lemma gives $\chi =  t$ and $\eta =  1 - e^{C_I(t)} (t-t_2) < 1$.
Also, since $t \in R_1$, equation (\ref{eqR2}) implies that we can fix
$I = (t_2,t_1)$ with $t = \chi \in I$, $(t,t_1) \subset \R \setminus \supp(\mu)$
and $(t_2,t) \subset (\chi+\eta-1,\chi) \setminus \supp(\l-\mu)$. Therefore
equation (\ref{eqf'R1}) holds as above, for all $w \in (\C \setminus \R) \cup I$.
This easily gives $f'(t) = f'(\chi) = 0$. Thus $(\chi,\eta) \in \EE_1$ by
definition, and $W_\EE(\chi,\eta) = t$, as required.

Consider (5). We prove this by showing:
\begin{enumerate}
\item[(5a)]
Fix $(\chi,\eta) \in \EE_2$ and let $t := W_\EE(\chi,\eta)$. Then
$t \in R_2$ and $(\chi,\eta) = (\chi_\EE(t),\eta_\EE(t))$.
\item[(5b)]
Fix $t \in R_2$ and let $(\chi,\eta) := (\chi_\EE(t),\eta_\EE(t))$. Then
$(\chi,\eta) \in \EE_2$ and $W_\EE(\chi,\eta) = t$.
\end{enumerate}

Consider (5a). Note, the definitions of $\EE_2$ and $W_\EE$ give $\eta<1$
and $t = \chi+\eta-1$. Moreover, $f'$ has a root at $t = \chi+\eta-1$, and so
equation (\ref{eqf'}) shows that there exists an interval, $I = (t_2,t_1)$,
with $t = \chi+\eta-1 \in I$,
$(t,t_1) \subset (\chi+\eta-1,\chi) \setminus \supp(\l-\mu)$ and
$(t_2,t) \subset \R \setminus \supp(\mu)$. Equation (\ref{eqR2}) then implies
that $t \in R_2$. Moreover, equations (\ref{eqf'}) and (\ref{eqCauTrans}) give,
\begin{equation}
\label{eqf'R2}
f'(w) = C_I(w) + \log(w-\chi) - \log(w-t_1),
\end{equation}
for all $w \in \C \setminus \R$, where $\log$ is principal value. This has a
trivial analytic extension to $(\C \setminus \R) \cup I$. Thus, since
$t = \chi+\eta-1$ and $f'(\chi+\eta-1) = 0$,
$C_I(t) + \log(1-\eta) - \log(t_1-t) = 0$, where now $\log$ denotes natural
logarithm. Solving gives $\eta = 1 - (t_1-t) e^{-C_I(t)}$, and so
$\chi = t+1-\eta = t + (t_1-t) e^{-C_I(t)}$.
Lemma (\ref{lemBddEdge1}) then gives $(\chi,\eta) = (\chi_\EE(t),\eta_\EE(t))$, as
required. Finally, (5b) follows in a similar way to (4b).
\end{proof}

Theorem \ref{thmEdge} only uses definition \ref{defEdge} for $\EE$, and shows
that $(\chi_\EE(\cdot),\eta_\EE(\cdot)) : R \to \partial \LL$ bijectively
maps $R$ to $\EE$. Definition \ref{defEdge0} defines $\EE$ as the image
space of this map. Therefore:
\begin{cor}
\label{corEdgeDefEquiv}
Definitions \ref{defEdge0} and \ref{defEdge} of the edge, $\EE$, are equivalent.
\end{cor}

We end this section with a result which further clarifies the above equivalence:
\begin{lem}
\label{lemEdge2}
Define $R_\mu, R_{\l-\mu}, R_0, R_1, R_2$ as in parts (1-5) in the proof
of theorem \ref{thmEdge}. Also, for all
$t \in R = R_\mu \cup R_{\l-\mu} \cup R_0 \cup R_1 \cup R_2$, define the function,
\begin{equation*}
f_t' := f_{(\chi_\EE(t), \eta_\EE(t))}',
\end{equation*}
where the right hand side is defined in equation (\ref{eqf'}). Then,
\begin{enumerate}
\item[(a)]
$(\chi_\EE(t), \eta_\EE(t)) \in \EE_\mu$ if and only if $t \in R_\mu$.
Moreover, in this case, $t$ is a root of $f_t'$ of multiplicity either
$2$ or $3$.
\item[(b)]
$(\chi_\EE(t), \eta_\EE(t)) \in \EE_{\l-\mu}$ if and only if $t \in R_{\l-\mu}$.
Moreover, in this case, $t$ is a root of $f_t'$ of multiplicity either
$2$ or $3$.
\item[(c)]
$(\chi_\EE(t), \eta_\EE(t)) \in \EE_0$ if and only if $t \in R_0$. Moreover,
in this case, the functions $f_t'$ and $C$ are equal, and $t$ is a root of
$f_t'$ of multiplicity $1$.
\item[(d)]
$(\chi_\EE(t), \eta_\EE(t)) \in \EE_1$ if and only if $t \in R_1$. Moreover,
in this case, $t$ is a root of $f_t'$ of multiplicity either $1$ or $2$.
\item[(e)]
$(\chi_\EE(t), \eta_\EE(t)) \in \EE_2$ if and only if $t \in R_2$. Moreover,
in this case, $t$ is a root of $f_t'$ of multiplicity either $1$ or $2$.
\end{enumerate}
\end{lem}

\begin{proof}
Consider (a). The fact that $(\chi_\EE(t), \eta_\EE(t)) \in \EE_\mu$ if and
only if $t \in R_\mu$ follows from part (1) of theorem \ref{thmEdge}. Also,
fixing $t \in R_\mu$ and defining
$(\chi,\eta) := (\chi_\EE(t), \eta_\EE(t)) \in \EE_\mu$, theorem \ref{thmEdge}
gives $t = W_\EE(\chi,\eta)$. The definitions of $\EE_\mu$ and $W_\EE$ then
imply that $t$ is a repeated root of $f_t'$, i.e., $t$ has multiplicity at
least $2$. Finally, theorem \ref{thmf'} implies that $t$ has multiplicity
at most $3$, as required. Part (b) follows similarly.

Consider (c). The fact that $(\chi_\EE(t), \eta_\EE(t)) \in \EE_0$ if and
only if $t \in R_0$ follows from part (3) of theorem \ref{thmEdge}. Also,
fixing $t \in R_0$ and defining
$(\chi,\eta) := (\chi_\EE(t), \eta_\EE(t)) \in \EE_0$, theorem \ref{thmEdge}
gives $t = W_\EE(\chi,\eta)$. The definitions of $\EE_0$ and $W_\EE$ then
imply that $\eta = 1$ and $t$ is a root of $f_t'$. Also, since $\eta=1$,
equations (\ref{eqf'}) and (\ref{eqCauTrans}) give $f_t'=C$. Finally,
$f_t''(t) = C'(t) = - \int_a^b \frac{\mu[dx]}{(t-x)^2} < 0$,
and so $t$ is a root of multiplicity $1$, as required.

Consider (d). The fact that $(\chi_\EE(t), \eta_\EE(t)) \in \EE_1$ if
and only if $t \in R_1$ follows from part (4) of theorem \ref{thmEdge}.
Also, fixing $t \in R_1$ and defining
$(\chi,\eta) := (\chi_\EE(t), \eta_\EE(t)) \in \EE_1$, theorem \ref{thmEdge}
gives $t = W_\EE(\chi,\eta) = \chi$. The definitions of $\EE_1$ and $W_\EE$
then imply that $t = \chi$ is a root of $f_t'$. It remains to show that
$t = \chi$ has multiplicity at most $2$. Note that part (b) of corollary
\ref{corf'} implies that $\eta>0$. Then, using the notation of theorem \ref{thmf'},
this theorem implies that $t = \chi \in J$ whenever $S_2$ is non-empty.
The result, whenever $S_2$ is non-empty, then follows from parts (1) and
(9) of theorem \ref{thmf'}. Whenever $S_2$ is empty,
the result follows from parts (11) and (12). This covers all possibilities.
Part (e) follows similarly.
\end{proof}

\subsection{Local geometric properties of the edge curve}
\label{seclgpotee}

Recall (see definition \ref{defEdge0}) that $\EE$ is the image of the smooth
curve defined by $t \mapsto (\chi_\EE(t),\eta_\EE(t))$ for all $t \in R$, where
$R = (\R \setminus \supp(\mu)) \cup (\R \setminus \supp(\l-\mu)) \cup R_1 \cup R_2$,
and $R_1,R_2$ are defined in equation (\ref{eqR2}). In this section we investigate
the local geometric properties of the edge curve. We show that the curve behaves
locally either like a parabola with negative curvature, or a cusp of first order,
and this behaviour is completely determined by the multiplicities of the roots
as presented in lemma \ref{lemEdge2}.

We begin by decomposing the image curve into orthogonal components and
performing a Taylor expansion:
\begin{lem}
\label{lemLocGeo}
Fix $t \in R$ and define the (un-normalised) orthogonal vectors,
$\mathbf{x} = \mathbf{x}(t)$ and $\mathbf{y} = \mathbf{y}(t)$ as,
\begin{itemize}
\item
$\mathbf{x} := (1,e^{C(t)}-1)$ and $\mathbf{y} := (e^{C(t)}-1,-1)$ when
$t \in (\R \setminus \supp(\mu)) \cup (\R \setminus \supp(\l-\mu))$.
\item
$\mathbf{x} := (0,1)$ and $\mathbf{y} := (1,0)$ when $t \in R_1$.
\item
$\mathbf{x} := (1,-1)$ and $\mathbf{y} := (1,1)$ when $t \in R_2$.
\end{itemize}
Then,
\begin{equation}
\label{eqrtvt}
(\chi_\EE(s),\eta_\EE(s)) - (\chi_\EE(t),\eta_\EE(t))
= a(s) \mathbf{x} + b(s) \mathbf{y},
\end{equation}
for all $s \in R$ sufficiently close to $t$, where
\begin{align*}
a(s) &= a_1 (s-t) + a_2 (s-t)^2 + O((s-t)^3), \\
b(s) &= b_1 (s-t)^2 + b_2 (s-t)^3 + O((s-t)^4),
\end{align*}
and $a_1 = a_1(t)$, $a_2 = a_2(t)$, $b_1 = b_1(t)$ and $b_2 = b_2(t)$ are defined
in the proof. (Above, when $t \in (\R \setminus \supp(\mu)) \cup (\R \setminus \supp(\l-\mu))$,
$e^{C(t)}$ is defined by the analytic extensions of parts (a) and (b) of lemma
\ref{lemAnalExt}.)
\end{lem}

\begin{proof}
First suppose that $t \in (\R \setminus \supp(\mu)) \cup (\R \setminus \supp(\l-\mu))$.
Fix  an interval, $I := (t_2,t_1)$, with $t \in I$, $I \subset \R \setminus \supp(\mu)$
whenever $t \in \R \setminus \supp(\mu)$, and $I \subset \R \setminus \supp(\l-\mu)$
whenever $t \in \R \setminus \supp(\l-\mu)$. Then, solving equation (\ref{eqrtvt}) gives,
\begin{align*}
(1 + (e^{C(t)}-1)^2) a(s)
&= (\chi_\EE(s) - \chi_\EE(t)) + (\eta_\EE(s) - \eta_\EE(t)) (e^{C(t)}-1), \\
(1 + (e^{C(t)}-1)^2) b(s)
&= (\chi_\EE(s) - \chi_\EE(t)) (e^{C(t)}-1) - (\eta_\EE(s) - \eta_\EE(t)),
\end{align*}
for all $s \in I$. Also, lemma \ref{lemBddEdge1} and parts (a) and (b) of lemma
\ref{lemAnalExt} give
\begin{equation}
\label{eqchi'eta'}
\chi_\EE'(t) = - \frac{ C''(t) (e^{C(t)}-1) - C'(t)^2 (e^{C(t)}+1)}{e^{C(t)} C'(t)^2}
\hspace{0.3cm} \text{and} \hspace{0.3cm}
\eta_\EE'(t) = (e^{C(t)}-1) \chi_\EE'(t).
\end{equation}
The second part of this equation and Taylor expansions give the required
result with,
\begin{itemize}
\item
$a_1 = \chi_\EE'(t)$.
\item
$2 a_2 = \chi_\EE''(t) + (1 + (e^{C(t)}-1)^2)^{-1} e^{C(t)}
(e^{C(t)}-1) C'(t) \chi_\EE'(t)$.
\item
$2 b_1 = - (1 + (e^{C(t)}-1)^2)^{-1} e^{C(t)} C'(t) \chi_\EE'(t)$.
\item
$6 b_2 = - (1 + (e^{C(t)}-1)^2)^{-1} e^{C(t)}
[2 C'(t) \chi_\EE''(t) + (C'(t)^2 + C''(t)) \chi_\EE'(t)]$.
\end{itemize}

Next suppose that $t \in R_1$. Fix an interval, $I = (t_2,t_1)$, as in
equation (\ref{eqR2}), i.e., with $t \in I$,
$(t,t_1) \subset \R \setminus \supp(\mu)$ and
$(t_2,t) \subset \R \setminus \supp(\l-\mu)$. Then, solving equation
(\ref{eqrtvt}),
\begin{equation*}
a(s) = \eta_\EE(s) - \eta_\EE(t)
\hspace{0.5cm} \text{and} \hspace{0.5cm}
b(s) = \chi_\EE(s) - \chi_\EE(t),
\end{equation*}
for all $s \in I$. Also, part (c) of lemma \ref{lemAnalExt} gives,
\begin{equation*}
e^{-C(s)} = e^{-C_I(s)} \left( \frac{s-t}{s-t_2} \right),
\end{equation*}
for all $s \in I$, where
$C_I(s) := \int_{[a,b] \setminus I} \frac{\mu[dx]}{s-x}$.
Lemma \ref{lemBddEdge1} then gives,
\begin{equation*}
\eta_\EE(s) = 1 - (t-t_2) e^{C_I(s)} + (s-t) h_1(s)
\hspace{0.5cm} \text{and} \hspace{0.5cm}
\chi_\EE(s) = t + (s-t)^2 h_2(s),
\end{equation*}
for all $s \in I$, where
\begin{align*}
h_1(s) &:= \frac{(t-t_2) (s-t_2) e^{2C_I(s)} C_I'(s) +
(s+t-2t_2) e^{2C_I(s)} - 2 (s-t_2) e^{C_I(s)} + (s-t)}
{-(t-t_2) e^{C_I(s)} + (s-t) (s-t_2) e^{C_I(s)} C_I'(s)}, \\
h_2(s) &:= \frac{(s-t_2) e^{C_I(s)} C_I'(s) + e^{C_I(s)} - 1}
{-(t-t_2) e^{C_I(s)} + (s-t) (s-t_2) e^{C_I(s)} C_I'(s)}.
\end{align*}
Taylor expansions then give the required result with,
\begin{itemize}
\item
$a_1 = - (t-t_2) e^{C_I(t)} C_I'(t) + h_1(t)$.
\item
$2 a_2 = - (t-t_2) e^{C_I(t)} C_I'(t)^2 - (t-t_2) e^{C_I(t)} C_I''(t) + 2 h_1'(t)$.
\item
$b_1 = h_2(t)$.
\item
$b_2 = h_2'(t)$.
\end{itemize}

Next suppose that $t \in R_2$. Fix an interval, $I = (t_2,t_1)$, as in
equation (\ref{eqR2}), i.e., with $t \in I$,
$(t,t_1) \subset \R \setminus \supp(\l-\mu)$ and
$(t_2,t) \subset \R \setminus \supp(\mu)$.
Then, solving equation (\ref{eqrtvt}),
\begin{align*}
2 a(s) &= (\chi_\EE(s) - \eta_\EE(s)) - (\chi_\EE(t) - \eta_\EE(t)), \\
2 b(s) &= (\chi_\EE(s) + \eta_\EE(s)) - (\chi_\EE(t) + \eta_\EE(t)),
\end{align*}
for all $s \in I$. Also part (d) of lemma \ref{lemAnalExt} gives,
\begin{equation*}
e^{C(s)} = e^{C_I(s)} \left( \frac{s-t}{s-t_1} \right),
\end{equation*}
for all $s \in I$. Lemma \ref{lemBddEdge1} then gives,
\begin{align*}
\chi_\EE(s) - \eta_\EE(s) &= t - 1 - 2 (t-t_1) e^{-C_I(s)} + (s-t) h_3(s), \\
\chi_\EE(s) + \eta_\EE(s) &= t + 1 + (s-t)^2 h_4(s),
\end{align*}
for all $s \in I$, where
\begin{align*}
h_3(s) &:= 1 + \frac{ 2 (t-t_1) (s-t_1) e^{-C_I(s)} C_I'(s) - 2 (s+t-2t_1) e^{-C_I(s)} + 3 (s-t_1)
- (s-t) e^{C_I(s)}}{t-t_1 + (s-t)(s-t_1) C_I'(s)}, \\
h_4(s) &:= \frac{(s-t_1) C_I'(s) + e^{C_I(s)} - 1}{t-t_1 + (s-t)(s-t_1) C_I'(s)}.
\end{align*}
Taylor expansions then give the required result with,
\begin{itemize}
\item
$2 a_1 = 2 (t-t_1) e^{-C_I(t)} C_I'(t) + h_3(t)$.
\item
$2 a_2 = - (t-t_1) e^{-C_I(t)} C_I'(t)^2 + (t-t_1) e^{-C_I(t)} C_I''(t) + h_3'(t)$.
\item
$2 b_1 = h_4(t)$.
\item
$2 b_2 = h_4'(t)$.
\end{itemize}
\end{proof}

The local geometric behaviour of the edge curve, as examined in lemma
\ref{lemLocGeo}, depends on the parameters $a_1, a_2, b_1, b_2$.
The next lemma further clarifies this by classifying the situations in
which these terms are zero or non-zero. We show that these situations are
completely determined by the multiplicities of the roots as presented in
lemma \ref{lemEdge2}:

\begin{lem}
\label{lemLineCusp}
Fix $t \in R = R_\mu \cup R_{\l-\mu} \cup R_0 \cup R_1 \cup R_2$, and define
$f_t'$ as in lemma \ref{lemEdge2}. Then the following 9 cases exhaust
all possibilities:
\begin{enumerate}
\item
$(\chi_\EE(t),\eta_\EE(t)) \in \EE_\mu$, where $t \in R_\mu$ is a root of $f_t'$
of multiplicity $2$.
\item
$(\chi_\EE(t),\eta_\EE(t)) \in \EE_\mu$, where $t \in R_\mu$ is a root of $f_t'$
of multiplicity $3$.
\item
$(\chi_\EE(t),\eta_\EE(t)) \in \EE_{\l-\mu}$, where $t \in R_{\l-\mu}$ is a root
of $f_t'$ of multiplicity $2$.
\item
$(\chi_\EE(t),\eta_\EE(t)) \in \EE_{\l-\mu}$, where $t \in R_{\l-\mu}$ is a root
of $f_t'$ of multiplicity $3$.
\item
$(\chi_\EE(t),\eta_\EE(t)) \in \EE_0$, where $t \in R_0$ is a root of $f_t'$
of multiplicity $1$.
\item
$(\chi_\EE(t),\eta_\EE(t)) \in \EE_1$, where $t \in R_1$ is a root of $f_t'$
of multiplicity $1$.
\item
$(\chi_\EE(t),\eta_\EE(t)) \in \EE_1$, where $t \in R_1$ is a root of $f_t'$
of multiplicity $2$.
\item
$(\chi_\EE(t),\eta_\EE(t)) \in \EE_2$, where $t \in R_2$ is a root of $f_t'$
of multiplicity $1$.
\item
$(\chi_\EE(t),\eta_\EE(t)) \in \EE_2$, where $t \in R_2$ is a root of $f_t'$
of multiplicity $2$.
\end{enumerate}
Moreover, defining $a_1 = a_1(t)$, $a_2 = a_2(t)$,
$b_1 = b_1(t)$ and $b_2 = b_2(t)$ as in lemma \ref{lemLocGeo},
\begin{itemize}
\item
$a_1 \neq 0$ and $b_1 \neq 0$ in cases (1, 3, 5, 6, 8).
\item
$a_1 = b_1 = 0$, $a_2 \neq 0$ and $b_2 \neq 0$ in cases (2, 4, 7, 9).
\end{itemize}
\end{lem}

\begin{proof}
The fact that cases (1-9) exhaust all possibilities follows trivially
from lemma \ref{lemEdge2}. Next, using the various analytic extensions of
lemma \ref{lemAnalExt}, we will show:
\begin{enumerate}
\item[(a)]
In cases (1, 2, 3, 4), i.e., when $t \in R_\mu \cup R_{\l-\mu}$ is a root of
multiplicity $2$ or $3$,
\begin{equation*}
a_1 = - \left( \frac{e^{C(t)}-1}{e^{C(t)} C'(t)^2} \right) f_t'''(t) \\
\hspace{0.5cm} \text{and} \hspace{0.5cm}
b_1 = \frac12 \left( \frac{e^{C(t)}-1}{C'(t)(1 + (e^{C(t)}-1)^2)} \right) f_t'''(t).
\end{equation*}
\item[(b)]
In cases (2, 4), i.e., when $t \in R_\mu \cup R_{\l-\mu}$ is a root of multiplicity $3$,
\begin{equation*}
a_2 = - \frac12 \left( \frac{e^{C(t)}-1}{e^{C(t)} C'(t)^2} \right) f_t^{(4)} (t) \\
\hspace{0.5cm} \text{and} \hspace{0.5cm}
b_2 = \frac13 \left( \frac{e^{C(t)}-1}{C'(t)(1 + (e^{C(t)}-1)^2)} \right) f_t^{(4)} (t).
\end{equation*}
\item[(c)]
In case (5), i.e., when $t \in R_0$ is a root of multiplicity $1$,
\begin{equation*}
a_1 = 2
\hspace{0.5cm} \text{and} \hspace{0.5cm}
b_1 = - C'(t) = - f_t''(t) = \int_a^b \frac{\mu[dx]}{(t-x)^2}.
\end{equation*}
\item[(d)]
In cases (6, 7), i.e., when $t \in R_1$ is a root of multiplicity $1$ or $2$,
\begin{equation*}
a_1 = - 2 (t-t_2) e^{C_I(t)} f_t''(t)
\hspace{0.5cm} \text{and} \hspace{0.5cm}
b_1 = - f_t''(t).
\end{equation*}
\item[(e)]
In case (7), i.e., when $t \in R_1$ is a root of multiplicity $2$,
\begin{equation*}
a_2 = - \frac32 (t-t_2) e^{C_I(t)} f_t'''(t)
\hspace{0.5cm} \text{and} \hspace{0.5cm}
b_2 = - f_t'''(t).
\end{equation*}
\item[(f)]
In cases (8, 9), i.e., when $t \in R_2$ is a root of multiplicity $1$ or $2$,
\begin{equation*}
a_1 = 2 (t-t_1) e^{-C_I(t)} f_t''(t)
\hspace{0.5cm} \text{and} \hspace{0.5cm}
b_1 = \frac12 f_t''(t).
\end{equation*}
\item[(g)]
In case (9), i.e., when $t \in R_2$ is a root of multiplicity $2$,
\begin{equation*}
a_2 = \frac32 (t-t_1) e^{-C_I(t)} f_t'''(t)
\hspace{0.5cm} \text{and} \hspace{0.5cm}
b_2 = \frac12 f_t'''(t).
\end{equation*}
\end{enumerate}
The result trivially follows from (a-g). For example, in
case (1), (a) implies that $a_1 \neq 0$ and $a_2 \neq 0$. Also, in case
(2), (a) and (b) imply that $a_1 = b_1 = 0$, $a_2 \neq 0$ and $b_2 \neq 0$,
etc.

For simplicity of notation, set $(\chi,\eta) := (\chi_\EE(t),\eta_\EE(t))$
for the remainder of this lemma. Consider (a) when $t \in R_\mu$.
Following the proof of part (1) of theorem \ref{thmEdge}, $\eta<1$,
$t \in (\R \setminus [\chi+\eta-1,\chi]) \setminus \supp(\mu)$, and
equation (\ref{eqf'Rmu}) holds for all
$w \in (\C \setminus [\chi+\eta-1,\chi]) \setminus \supp(\mu)$. Also equation
(\ref{eqlemEdge1}) holds where $C(t)$ and $C'(t)$ are defined by the
analytic extensions of part (a) of lemma \ref{lemAnalExt}. Differentiate
equation (\ref{eqf'Rmu}) twice, take $w=t$, and substitute $t-\chi$ and
$t-\chi-\eta+1$ from equation (\ref{eqlemEdge1}) to get,
\begin{equation}
\label{eqlemLineCusp2}
f_t'''(t) = C''(t) - C'(t)^2 \frac{e^{C(t)}+1}{e^{C(t)}-1},
\end{equation}
when $t \in R_\mu$ (note $e^{C(t)}-1 \neq 0$, since $C(t) \neq 0$ by
definition). Equation (\ref{eqchi'eta'}) then gives,
\begin{equation}
\label{eqlemLineCusp3}
\chi_\EE'(t) = - \left( \frac{e^{C(t)}-1}{e^{C(t)} C'(t)^2} \right) f_t'''(t).
\end{equation}
Part (a), when $t \in R_\mu$, then follows from the expressions
of $a_1$ and $b_1$ given in the proof of lemma \ref{lemLocGeo}.

Consider (a) when $t \in R_{\l-\mu}$. Following the proof of part (2) of theorem
\ref{thmEdge}, $\eta<1$ and $t \in (\chi+\eta-1,\chi) \setminus \supp(\l-\mu)$.
Also, fixing $I = (t_2,t_1) \subset (\chi+\eta-1,\chi) \setminus \supp(\l-\mu)$
with $t \in I$, equation (\ref{eqf'Rl-mu}) holds for all
$w \in (\C \setminus \R) \cup I$. Also, equation (\ref{eqlemEdge1}) holds, where
$e^{C(t)}$ and $C'(t)$ are defined by the analytic extensions of part (b)
of lemma (\ref{lemAnalExt}). Differentiate equation (\ref{eqf'Rl-mu})
and use the analytic extension for $C'$ to get,
\begin{equation}
\label{eqlemLineCusp4}
f_t''(w) = C'(w) + \frac1{w - \chi} - \frac1{w-\chi-\eta+1},
\end{equation}
for all $w \in (\C \setminus \R) \cup I$. Differentiate again, take $w=t$,
and substitute $t-\chi$ and $t-\chi-\eta+1$ from equation (\ref{eqlemEdge1})
to show that equation (\ref{eqlemLineCusp2}) holds in this case also.
Finally, we can proceed as above to show part (a) when $t \in R_{\l-\mu}$.

Consider (b). First note that $f_t'''(t) = 0$, since $t$ is a root of $f_t'$ of
multiplicity $3$. Equation (\ref{eqlemLineCusp2}) thus gives,
\begin{equation*}
C''(t) = C'(t)^2 \frac{e^{C(t)}+1}{e^{C(t)}-1}.
\end{equation*}
Next, differentiate the first part of equation (\ref{eqchi'eta'}),
and substitute the above expression of $C''(t)$ to get,
\begin{equation*}
\chi_\EE''(t) = - \frac{e^{C(t)}-1}{e^{C(t)} C'(t)^2}
\left( C'''(t) - 2 C'(t)^3 \; \frac{e^{2C(t)} + e^{C(t)} + 1}{(e^{C(t)}-1)^2} \right).
\end{equation*}
Finally, differentiate equation (\ref{eqf'Rmu}) three times (when $t \in R_\mu$)
or differentiate equation (\ref{eqlemLineCusp4}) twice (when $t \in R_{\l-\mu}$),
take $w=t$, and substitute $t-\chi$ and $t-\chi-\eta+1$ from equation
(\ref{eqlemEdge1}) to get,
\begin{equation*}
f_t^{(4)}(t) = C'''(t) - 2 C'(t)^3 \; \frac{e^{2C(t)} + e^{C(t)} + 1}{(e^{C(t)}-1)^2}.
\end{equation*}
Therefore
\begin{equation*}
\chi_\EE''(t) = - \left( \frac{e^{C(t)}-1}{e^{C(t)} C'(t)^2} \right) f_t^{(4)}(t).
\end{equation*}
Also $\chi_\EE'(t) = 0$ since $f_t'''(t) = 0$ (see equation (\ref{eqlemLineCusp3})).
Part (b) then follows from the expressions of $a_2$ and $b_2$ given in the proof
of lemma \ref{lemLocGeo}.

Consider (c). Recall that $t \in R_0$, i.e., $t \in \R \setminus \supp(\mu)$
and $C(t) = 0$. Following the proof of part (3) of theorem \ref{thmEdge},
$(\chi,\eta) = (t,1)$, and $f_t'(w) = C(w)$ for all
$w \in \C \setminus \supp(\mu)$. Therefore $C(t) = 0$, $f_t'(t) = C(t)$, and
equation (\ref{eqCauTrans}) gives,
\begin{equation*}
f_t''(t) = C'(t) = - \int_a^b \frac{\mu[dx]}{(t-x)^2} < 0.
\end{equation*}
Equation (\ref{eqchi'eta'}) then gives $\chi_\EE'(t) = 2$. Part (c) then
follows from the expressions of $a_1$ and $b_1$ given in the proof of
lemma \ref{lemLocGeo}.

Consider (d). Following the proof of part (4) of theorem \ref{thmEdge},
$\eta<1$ and $t = \chi$. Also, fixing $I = (t_2,t_1)$ with
$t = \chi \in I$, $(t,t_1) \subset \R \setminus \supp(\mu)$ and
$(t_2,t) \subset (\chi+\eta-1,\chi) \setminus \supp(\l-\mu)$, equation
(\ref{eqf'R1}) holds for all $w \in (\C \setminus \R) \cup I$.
Differentiate this equation, take $w = t$, and substitute $\chi = t$ and
$\eta = 1 - e^{C_I(t)} (t-t_2)$ (see lemma \ref{lemBddEdge1}) to get,
\begin{equation*}
f_t''(t) = C_I'(t) + \frac1{t-t_2} - \frac1{(t-t_2) e^{C_I(t)}}.
\end{equation*}
Part (d) then follows from the expressions of $a_1$ and $b_1$ given in
the proof of lemma \ref{lemLocGeo}.

Consider (e). First note that $f_t''(t) = 0$, since $t$ is a root of $f_t'$ of
multiplicity $2$. The above expression of $f_t''(t)$ thus gives,
\begin{equation*}
C_I'(t) = \frac1{(t-t_2) e^{C_I(t)}} - \frac1{t-t_2}.
\end{equation*}
Next, substitute the above expression of $C_I'(t)$ into the expressions
of $a_2$ and $b_2$ given in the proof of lemma \ref{lemLocGeo} to get,
\begin{align*}
a_2
&= - \frac32 (t-t_2) e^{C_I(t)}
\left( C_I''(t) - \frac1{(t-t_2)^2} + \frac1{(t-t_2)^2 e^{2 C_I(t)}} \right), \\
b_2
&= - \left( C_I''(t) - \frac1{(t-t_2)^2} + \frac1{(t-t_2)^2 e^{2 C_I(t)}} \right).
\end{align*}
Finally, differentiate equation (\ref{eqf'R1}) twice, take $w = t$, and substitute
$\chi = t$ and $\eta = 1 - e^{C_I(t)} (t-t_2)$ to get,
\begin{equation*}
f_t'''(t) = C_I''(t) - \frac1{(t-t_2)^2} + \frac1{(t-t_2)^2 e^{2 C_I(t)}}.
\end{equation*}
Part (e) easily follows.

Consider (f). Following the proof of part (5) of theorem \ref{thmEdge}, $\eta<1$
and $t = \chi+\eta-1$. Also, fixing $I = (t_2,t_1)$ with $t = \chi+\eta-1 \in I$,
$(t,t_1) \subset (\chi+\eta-1,\chi) \setminus \supp(\l-\mu)$ and
$(t_2,t) \subset \R \setminus \supp(\mu)$, equation (\ref{eqf'R2}) holds for all
$w \in (\C \setminus \R) \cup I$. Differentiate this equation, take $w = t$, and
substitute $\chi = t - e^{-C_I(t)} (t-t_1)$ and $\eta = 1 + e^{-C_I(t)} (t-t_1)$
(see lemma \ref{lemBddEdge1}) to get,
\begin{equation*}
f_t''(t) = C_I'(t) + \frac1{(t-t_1) e^{-C_I(t)}} - \frac1{t-t_1}.
\end{equation*}
Part (f) then follows from the expressions of $a_1$ and $b_1$ given in
the proof of lemma \ref{lemLocGeo}.

Consider (g). First note that $f_t''(t) = 0$, since $t$ is a root of $f_t'$ of
multiplicity $2$. The above expression of $f_t''(t)$ thus gives,
\begin{equation*}
C_I'(t) = \frac1{t-t_1} - \frac1{(t-t_1) e^{-C_I(t)}}.
\end{equation*}
Next, substitute the above expression of $C_I'(t)$ into the expressions
of $a_2$ and $b_2$ given in the proof of lemma \ref{lemLocGeo} to get,
\begin{align*}
a_2
&= \frac32 (t-t_1) e^{-C_I(t)}
\left( C_I''(t) - \frac1{(t-t_1)^2 e^{- 2 C_I(t)}} + \frac1{(t-t_1)^2} \right), \\
b_2
&= \frac12 \left( C_I''(t) - \frac1{(t-t_1)^2 e^{- 2 C_I(t)}} + \frac1{(t-t_1)^2} \right).
\end{align*}
Finally, differentiate equation (\ref{eqf'R2}) twice, take $w = t$, and substitute
$\chi = t - e^{-C_I(t)} (t-t_1)$ and $\eta = 1 + e^{-C_I(t)} (t-t_1)$ to get,
\begin{equation*}
f_t'''(t) = C_I''(t) - \frac1{(t-t_1)^2 e^{- 2 C_I(t)}} + \frac1{(t-t_1)^2}.
\end{equation*}
Part (g) easily follows.
\end{proof}

We end this section with a result which clarifies and summaries all possible cases:
\begin{lem}
\label{lemGeoObv}
Fix $t \in R = R_\mu \cup R_{\l-\mu} \cup R_0 \cup R_1 \cup R_2$, and define
$f_t'$ as in lemma \ref{lemEdge2}. Also define, as in lemma \ref{lemLocGeo},
$\mathbf{x} = \mathbf{x}(t)$, $\mathbf{y} = \mathbf{y}(t)$, $a_1 = a_1(t)$,
$a_2 = a_2(t)$, $b_1 = b_1(t)$ and $b_2 = b_2(t)$. Consider the exhaustive
cases, (1-9), of lemma \ref{lemLineCusp}:

\begin{flushleft}
{\bf Case (1):} When $t \in R_\mu$ is a root of $f_t'$
of multiplicity $2$,
\end{flushleft}
\begin{itemize}
\item
$(\chi_\EE(t),\eta_\EE(t))$, given by lemma \ref{lemBddEdge1}, is in the
interior of the shape in the middle of figure \ref{figGTAsyShape}.
\item
$(\chi_\EE(s),\eta_\EE(s)) - (\chi_\EE(t),\eta_\EE(t))
= a(s) \mathbf{x} + b(s) \mathbf{y} \; $ for all $s \in R$ close to $t$.
\item
$\mathbf{x} = (1,e^{C(t)}-1)$ and $\mathbf{y} = (e^{C(t)}-1,-1)$.
\item
$a(s) = a_1 (s-t) + O((s-t)^2)$ and $b(s) = b_1 (s-t)^2 + O((s-t)^3)$.
\item
$a_1 \neq 0$ and $b_1 \neq 0$ are given in part (a) of the proof of
lemma \ref{lemLineCusp}.
\item
The edge curve behaves like a parabola with negative curvature
in a neighbourhood of $(\chi_\EE(t),\eta_\EE(t))$, with tangent
vector $\mathbf{x}$ and normal vector $\mathbf{y}$.
\end{itemize}

\begin{flushleft}
{\bf Case (2):} The set of all $t \in R_\mu$ for which $t$ is a
root of $f_t'$ of multiplicity $3$ is a discrete subset of $(a,b)$, i.e.,
it is composed of isolated singletons. Moreover, in this case,
\end{flushleft}
\begin{itemize}
\item
$(\chi_\EE(t),\eta_\EE(t))$, given by lemma \ref{lemBddEdge1}, is in the
interior of the shape in the middle of figure \ref{figGTAsyShape}.
\item
$(\chi_\EE(s),\eta_\EE(s)) - (\chi_\EE(t),\eta_\EE(t))
= a(s) \mathbf{x} + b(s) \mathbf{y} \; $ for all $s \in R$ close to $t$.
\item
$\mathbf{x} = (1,e^{C(t)}-1)$ and $\mathbf{y} = (e^{C(t)}-1,-1)$.
\item
$a(s) = a_2 (s-t)^2 + O((s-t)^3)$ and $b(s) = b_2 (s-t)^3 + O((s-t)^4)$.
\item
$a_2 \neq 0$ and $b_2 \neq 0$ are given in part (b) of the proof of
lemma \ref{lemLineCusp}.
\item
The edge curve behaves like a cusp of first order in a neighbourhood of
$(\chi_\EE(t),\eta_\EE(t))$.
\end{itemize}

\begin{flushleft}
{\bf Case (3):} When $t \in R_{\l-\mu}$ is a root of $f_t'$ of
multiplicity $2$, the edge curve behaves similarly to case (1).
\end{flushleft}

\begin{flushleft}
{\bf Case (4):} The set of all $t \in R_{\l-\mu}$ for which $t$ is a
root of $f_t'$ of multiplicity $3$ is a discrete subset of $(a,b)$.
Moreover, in this case, the edge curve behaves similarly to case (2).
\end{flushleft}

\begin{flushleft}
{\bf Case (5):} $R_0$, the set of all $t \in \R \setminus \supp(\mu)$
with $C(t) = 0$, is a discrete subset of $(a,b)$.
Moreover, in this case,
\end{flushleft}
\begin{itemize}
\item
$(\chi_\EE(t),\eta_\EE(t)) = (t,1)$.
\item
$(\chi_\EE(s),\eta_\EE(s)) - (\chi_\EE(t),\eta_\EE(t))
= a(s) \mathbf{x} + b(s) \mathbf{y} \; $ for all $s \in R$ close to $t$.
\item
$\mathbf{x} = (1,0)$ and $\mathbf{y} = (0,-1)$.
\item
$a(s) = a_1 (s-t) + O((s-t)^2)$ and $b(s) = b_1 (s-t)^2 + O((s-t)^3)$.
\item
$a_1 \neq 0$ and $b_1 \neq 0$ are given in part (c) of the proof of
lemma \ref{lemLineCusp}.
\item
The edge curve behaves like a parabola with negative curvature
in a neighbourhood of $(\chi_\EE(t),\eta_\EE(t))$, with tangent
vector $\mathbf{x}$ and normal vector $\mathbf{y}$.
\end{itemize}

\begin{flushleft}
{\bf Case (6):} $R_1$ is a discrete subset of $[a,b]$. Moreover,
when $t \in R_1$ is a root of $f_t'$ of multiplicity $1$,
\end{flushleft}
\begin{itemize}
\item
$(\chi_\EE(t),\eta_\EE(t))$, given by lemma \ref{lemBddEdge1}, satisfy
$\chi_\EE(t) = t$ and $\eta_\EE(t) \in (0,1)$.
\item
$(\chi_\EE(s),\eta_\EE(s)) - (\chi_\EE(t),\eta_\EE(t))
= a(s) \mathbf{x} + b(s) \mathbf{y} \; $ for all $s \in R$ close to $t$.
\item
$\mathbf{x} = (0,1)$ and $\mathbf{y} = (1,0)$.
\item
$a(s) = a_1 (s-t) + O((s-t)^2)$ and $b(s) = b_1 (s-t)^2 + O((s-t)^3)$.
\item
$a_1 \neq 0$ and $b_1 \neq 0$ are given in part (d) of the proof of
lemma \ref{lemLineCusp}.
\item
The edge curve behaves like a parabola with negative curvature
in a neighbourhood of $(\chi_\EE(t),\eta_\EE(t))$, with tangent
vector $\mathbf{x}$ and normal vector $\mathbf{y}$.
\end{itemize}

\begin{flushleft}
{\bf Case (7):} $R_1$ is a discrete subset of $[a,b]$. Moreover,
when $t \in R_1$ is a root of $f_t'$ of multiplicity $2$,
\end{flushleft}
\begin{itemize}
\item
$(\chi_\EE(t),\eta_\EE(t))$, given by lemma \ref{lemBddEdge1}, satisfy
$\chi_\EE(t) = t$ and $\eta_\EE(t) \in (0,1)$.
\item
$(\chi_\EE(s),\eta_\EE(s)) - (\chi_\EE(t),\eta_\EE(t))
= a(s) \mathbf{x} + b(s) \mathbf{y} \; $ for all $s \in R$ close to $t$.
\item
$\mathbf{x} = (0,1)$ and $\mathbf{y} = (1,0)$.
\item
$a(s) = a_2 (s-t)^2 + O((s-t)^3)$ and $b(s) = b_2 (s-t)^3 + O((s-t)^4)$.
\item
$a_2 \neq 0$ and $b_2 \neq 0$ are given in part (e) of the proof of
lemma \ref{lemLineCusp}.
\item
The edge curve behaves like a cusp of first order in a neighbourhood of
$(\chi_\EE(t),\eta_\EE(t))$.
\end{itemize}

\begin{flushleft}
{\bf Case (8):} $R_2$ is a discrete subset of $[a,b]$. Moreover,
when $t \in R_2$ is a root of $f_t'$ of multiplicity $1$,
\end{flushleft}
\begin{itemize}
\item
$(\chi_\EE(t),\eta_\EE(t))$, given by lemma \ref{lemBddEdge1}, satisfy
$\chi_\EE(t)+\eta_\EE(t)-1 = t$ and $\eta_\EE(t) \in (0,1)$.
\item
$(\chi_\EE(s),\eta_\EE(s)) - (\chi_\EE(t),\eta_\EE(t))
= a(s) \mathbf{x} + b(s) \mathbf{y} \; $ for all $s \in R$ close to $t$.
\item
$\mathbf{x} = (1,-1)$ and $\mathbf{y} = (1,1)$.
\item
$a(s) = a_1 (s-t) + O((s-t)^2)$ and $b(s) = b_1 (s-t)^2 + O((s-t)^3)$.
\item
$a_1 \neq 0$ and $b_1 \neq 0$ are given in part (f) of the proof of
lemma \ref{lemLineCusp}.
\item
The edge curve behaves like a parabola with negative curvature
in a neighbourhood of $(\chi_\EE(t),\eta_\EE(t))$, with tangent
vector $\mathbf{x}$ and normal vector $\mathbf{y}$.
\end{itemize}

\begin{flushleft}
{\bf Case (9):} $R_2$ is a discrete subset of $[a,b]$. Moreover,
when $t \in R_2$ is a root of $f_t'$ of multiplicity $2$,
\end{flushleft}
\begin{itemize}
\item
$(\chi_\EE(t),\eta_\EE(t))$, given by lemma \ref{lemBddEdge1}, satisfy
$\chi_\EE(t)+\eta_\EE(t)-1 = t$ and $\eta_\EE(t) \in (0,1)$.
\item
$(\chi_\EE(s),\eta_\EE(s)) - (\chi_\EE(t),\eta_\EE(t))
= a(s) \mathbf{x} + b(s) \mathbf{y} \; $ for all $s \in R$ close to $t$.
\item
$\mathbf{x} = (1,-1)$ and $\mathbf{y} = (1,1)$.
\item
$a(s) = a_2 (s-t)^2 + O((s-t)^3)$ and $b(s) = b_2 (s-t)^3 + O((s-t)^4)$.
\item
$a_2 \neq 0$ and $b_2 \neq 0$ are given in part (g) of the proof of
lemma \ref{lemLineCusp}.
\item
The edge curve behaves like a cusp of first order in a neighbourhood of
$(\chi_\EE(t),\eta_\EE(t))$.
\end{itemize}
\end{lem}

\begin{proof}
Consider (1). First note, part (a) of corollary \ref{corf'} implies that
$(\chi_\EE(t),\eta_\EE(t))$ is in the interior of the shape in the middle
of figure \ref{figGTAsyShape}, since, following the notation of this
corollary, $S_1,S_2,S_3$ are all non-empty. Moreover, lemma \ref{lemLocGeo}
gives the required Taylor expansion, and the Taylor expansion easily
implies that the edge curve behaves like a parabola. It remains to show that
the curvature, $k(t)$, is negative at $(\chi_\EE(t),\eta_\EE(t))$, where
\begin{equation*}
k(t) := \frac{\chi_\EE'(t) \eta_\EE''(t) - \eta_\EE'(t) \chi_\EE''(t)}
{(\chi_\EE'(t)^2 + \eta_\EE'(t)^2)^\frac32}.
\end{equation*}
Note,
denoting $\mathbf{x} = (x_1,x_2)$ and $\mathbf{y} = (y_1,y_2)$, the
Taylor expansion gives,
\begin{equation*}
\chi_\EE'(t) \eta_\EE''(t) - \eta_\EE'(t) \chi_\EE''(t)
= 2 a_1 b_1 (x_1 y_2 - x_2 y_1).
\end{equation*}
Substitute the expressions for $a_1$ and $b_1$ from part (a) of lemma
\ref{lemLineCusp}, and substitute $\mathbf{x} = (1,e^{C(t)}-1)$ and
$\mathbf{y} = (e^{C(t)}-1,-1)$ to get,
\begin{equation*}
\chi_\EE'(t) \eta_\EE''(t) - \eta_\EE'(t) \chi_\EE''(t)
= \frac{(e^{C(t)}-1)^2 f_t'''(t)^2}{e^{C(t)} C'(t)^3}.
\end{equation*}
Part (a) of lemma \ref{lemAnalExt} gives $e^{C(t)} > 0$ and
$C'(t) < 0$. Also $e^{C(t)}-1 \neq 0$ since $t \in R_\mu$, and
$f_t'''(t) \neq 0$ since $t$ is a root of $f_t'$ of multiplicity $2$. Therefore
$k(t) < 0$, as required. Similarly for case (3).

Consider (2). Note, given $t \in R_\mu$ for which
$t$ is a root of $f_t'$ of multiplicity $3$, equation
(\ref{eqlemLineCusp2}) implies that $t$ is a root of the function
$w \mapsto C''(w) (e^{C(w)}-1) - C'(w)^2 (e^{C(w)}+1)$. Also lemma
\ref{lemAnalExt} implies that this function is well-defined and
analytic for $w \in (\R \setminus \supp(\mu)) \cup
(\R \setminus \supp(\l-\mu))$. Thus, since $R_\mu
\subset (\R \setminus \supp(\mu)) \cup (\R \setminus \supp(\l-\mu))$,
and since the roots of an analytic function are isolated, the set of
all $t \in R_\mu$ for which $t$ is a root of $f_t'$
of multiplicity $3$ is a discrete set. Also, theorem \ref{thmf'} implies
that this set is contained in $(a,b)$, since, using the notation of
this theorem, roots of multiplicity $3$ are only possible in $K$,
and $K \subset (a,b)$. The remainder of the result for case (2)
follows using similar methods to those used for case (1).
Similarly for case (4).

Consider (5). Part (a) of lemma \ref{lemAnalExt} implies that $C$ has
a unique analytic extension to $\C \setminus \supp(\mu)$. Recall that the
roots of an analytic function are isolated. It thus follows that $R_0$,
the set of all $t \in \R \setminus \supp(\mu)$ with $C(t) = 0$, is a
discrete set. Also, $R_0 \subset (a,b)$, since $\{a,b\} \subset \supp(\mu)$
(see hypothesis \ref{hypWeakConv}), and since $C(t) > 0$ for all $t > b$
and $C(t) < 0$ for all $t < a$ (see part (a) of lemma \ref{lemAnalExt}).
Finally, since $C(t) = 0$ for all $t \in R_0$, lemma \ref{lemBddEdge1}
gives $(\chi_\EE(t),\eta_\EE(t)) = (t,1)$. The remainder of the result
for case (5) follows using similar methods to those used for case (1).

Consider (6). The fact that $R_1$ is a discrete subset of $[a,b]$
follows from remark \ref{remmu} and the definition of $R_1$ given
in equation (\ref{eqR2}). Also, lemma \ref{lemBddEdge1} implies that
$\chi_\EE(t) = t$ and $\eta_\EE(t) < 1$ for all $t \in R_1$. Finally,
part (b) of corollary \ref{corf'} implies that $\eta_\EE(t) > 0$. The
remainder of the result for case (6) follows using similar methods to
those used for case (1). Similarly for cases (7,8,9).
\end{proof}

\subsection{Examples}
\label{secEx}

In this section we consider some examples of the measure, $\mu$, of
hypothesis \ref{hypWeakConv}, and apply the results of the previous
sections to study $\LL$ and $\EE$ in each case. We give expressions
for the Cauchy transform (obtained from equation (\ref{eqCauTrans})),
the set $R = R_\mu \cup R_{\l-\mu} \cup R_0 \cup R_1 \cup R_2$
(obtained from the definitions of the sets given in parts (1-5) of
theorem \ref{thmEdge}, and the edge curve (obtained from lemma
\ref{lemBddEdge1}). 

\begin{rem}
\label{remNotation}
Notation: In this section, $\log$ always denotes principal
value, and $\varphi : [a,b] \to [0,1]$ denotes the density of
$\mu$. Also, whenever we say that a point $(\chi,\eta) \in \EE$
corresponds to a root of multiplicity $k$, we mean that
$t := W_\EE(\chi,\eta)$ is a root of $f_t'$ of multiplicity $k$
(see lemma \ref{lemEdge2}). Finally, when describing $\partial \LL$,
we will often say:
\begin{itemize}
\item
$p_1 \to^\text{i} p_2$
\item
$p_1 \to^\text{e} p_2$
\end{itemize}
In both cases $p_1$ and $p_2$ are distinct points of $\partial \LL$. In the
first case we mean that section of $\partial \LL$ with end-points $p_1$ and
$p_2$, traversed clockwise, including the end-points. In the second case we
mean that section of $\partial \LL$ with end-points $p_1$ and $p_2$,
traversed clockwise, excluding the end-points.
\end{rem}

\subsubsection{\bf Example 1:}
\label{secEx1}

Take $\varphi(x) = \frac12$ for all $x \in [-1,1]$. Then,
\begin{align*}
2 C(w) &= \log(w+1) - \log(w-1)
\text{ for all } w \in \C \setminus \R, \\
R &= \R \setminus [-1,1], \\
R_\mu &= R,
\; R_{\l-\mu} = \emptyset,
\; R_0 = \emptyset,
\; R_1 = \emptyset,
\; R_2 = \emptyset, \\
\chi_\EE(t) &= t - \sqrt{|t+1|} \; |t-1| \;
( \sqrt{|t+1|} - \sqrt{|t-1|} )
\text{ for all } t \in R = \R \setminus [-1,1], \\
\eta_\EE(t) &= 1 - \sqrt{|t+1| \; |t-1|} \;
( \sqrt{|t+1|} - \sqrt{|t-1|} )^2
\text{ for all } t \in R = \R \setminus [-1,1].
\end{align*}

Note that each $t \in [-1,1]$ satisfies one of the cases of
lemma \ref{lemBddEdge2}. Part (2) of lemma \ref{lemBddEdge3} thus
implies that lemmas \ref{lemBddEdge0}, \ref{lemBddEdge1}, and
\ref{lemBddEdge2} give a complete description of $\partial \LL$.
Plotting this gives figure \ref{figex1}. The dashed lines represent the
shape in the middle of figure \ref{figGTAsyShape}, and the solid lines
represent $\partial \LL$. In words, following $\partial \LL$ in a
clockwise direction from the bottom, and using the notation described
in remark \ref{remNotation}:
\begin{itemize}
\item
$p_0 = (\frac12,0) = (\frac12 + \int_a^b x \varphi(x) dx, 0)$.
\item
$p_0 \to^\text{e} (-1,1)$ is in $\EE_\mu$ and given by
$t \mapsto (\chi_\EE(t),\eta_\EE(t))$ for all
$t \in (-\infty,-1) \subset R_\mu$.
\item
$(-1,1) \to^\text{i} (1,1)$ is given by $t \mapsto (t,1)$ for all
$t \in [-1,1]$.
\item
$(1,1) \to^\text{e} p_0$ is in $\EE_\mu$ and given by
$t \mapsto (\chi_\EE(t),\eta_\EE(t))$ for all
$t \in (1,+\infty) \subset R_\mu$.
\end{itemize}
Above, $p_0$ follows from lemma \ref{lemBddEdge0}, $p_0 \to^\text{e} (-1,1)$
and $(1,1) \to^\text{e} p_0$ follow from lemma \ref{lemBddEdge1}, and
$(-1,1) \to^\text{i} (1,1)$ follows from lemma \ref{lemBddEdge2}.

\begin{figure}
\centering
\mbox{\includegraphics{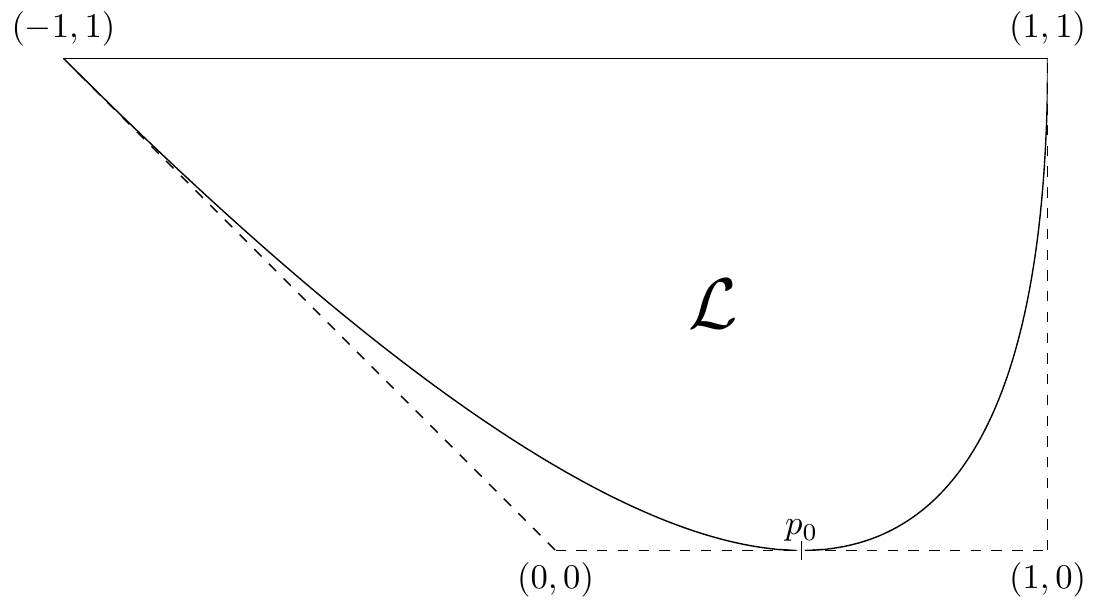}}
\caption{$\LL$ when $\varphi(x) = \frac12$ for all $x \in [-1,1]$.}
\label{figex1}
\end{figure}

Note that part (1) of theorem \ref{thmf'} implies that each
$t \in R = R_\mu = \R \setminus [-1,1]$ corresponds to a root of
multiplicity $2$. Case (1) of lemma \ref{lemGeoObv} then implies
that the edge curve behaves locally like a parabola with negative
curvature in a neighbourhood of any edge point. This can clearly
be seen in figure \ref{figex1}.

\subsubsection{\bf Example 2:}
\label{secEx2}

Take $\varphi(x) = \frac12$ for all $x \in [0,1] \cup [2,3]$. Then,
\begin{align*}
2 C(w) &= \log(w) - \log(w-1) + \log(w-2) - \log(w-3)
\text{ for all } w \in \C \setminus \R, \\
R &= \R \setminus ([0,1] \cup [2,3])
= (-\infty,0) \cup (1,2) \cup (3,+\infty), \\
R_\mu &= R \setminus \{3/2\}
= (-\infty,0) \cup (1,3/2) \cup (3/2,2) \cup (3,+\infty), \\
\; R_{\l-\mu} &= \emptyset,
\; R_0 = \{ 3/2 \},
\; R_1 = \emptyset,
\; R_2 = \emptyset, \\
\chi_\EE(t) &= t - 2 \sqrt{|t|} \; |t-1| \; \sqrt{|t-2|} \; |t-3|
\frac{\sqrt{|t| \; |t-2|} - \sqrt{|t-1| \; |t-3|}}
{|t-2| \; |t-3| + |t| \; |t-1|}
\text{ for all } t \in R, \\
\eta_\EE(t) &= 1 - 2 \sqrt{|t| \; |t-1| \; |t-2| \; |t-3|}
\frac{(\sqrt{|t| \; |t-2|} - \sqrt{|t-1| \; |t-3|})^2}
{|t-2| \; |t-3| + |t| \; |t-1|}
\text{ for all } t \in R.
\end{align*}

As in example 1, lemmas \ref{lemBddEdge0}, \ref{lemBddEdge1}, and
\ref{lemBddEdge2} give a complete description of $\partial \LL$. Plotting
$\partial \LL$ gives figure \ref{figex2}. Following $\partial \LL$ in a
clockwise direction from the bottom:
\begin{itemize}
\item
$p_0 = (2,0)  = (\frac12 + \int_a^b x \varphi(x) dx, 0)$.
\item
$p_0 \to^\text{e} (0,1)$ is in $\EE_\mu$ and given by
$t \mapsto (\chi_\EE(t),\eta_\EE(t))$ for all
$t \in (-\infty,0) \subset R_\mu$.
\item
$(0,1) \to^\text{i} (1,1)$ is given by $t \mapsto (t,1)$
for all $t \in [0,1]$.
\item
$(1,1) \to^\text{e} p_1$ is in $\EE_\mu$ and given by
$t \mapsto (\chi_\EE(t),\eta_\EE(t))$ for all
$t \in (1,\frac32) \subset R_\mu$.
\item
$p_1 = (\frac32,1)$ is in $\EE_0$ and given by
$p_1 = (\chi_\EE(\frac32),\eta_\EE(\frac32))$ where $\frac32 \in R_0$.
\item
$p_1 \to^\text{e} (2,1)$ is in $\EE_\mu$ and given by
$t \mapsto (\chi_\EE(t),\eta_\EE(t))$ for all
$t \in (\frac32,2) \subset R_\mu$.
\item
$(2,1) \to^\text{i} (3,1)$ is given by $t \mapsto (t,1)$ for all
$t \in [2,3]$.
\item
$(3,1) \to^\text{e} p_0$ is in $\EE_\mu$ and given by
$t \mapsto (\chi_\EE(t),\eta_\EE(t))$ for all
$t \in (3,+\infty) \subset R_\mu$.
\end{itemize}
Above, $p_0$ follows from lemma \ref{lemBddEdge0}. Also, $p_0 \to^\text{e} (0,1)$,
$(1,1) \to^\text{e} p_1$, $p_1$, $p_1 \to^\text{e} (2,1)$ and
$(3,1) \to^\text{e} p_0$ follow from lemma \ref{lemBddEdge1}. Finally,
$(0,1) \to^\text{i} (1,1)$ and $(2,1) \to^\text{i} (3,1)$ follow from
lemma \ref{lemBddEdge2}.

\begin{figure}
\centering
\mbox{\includegraphics{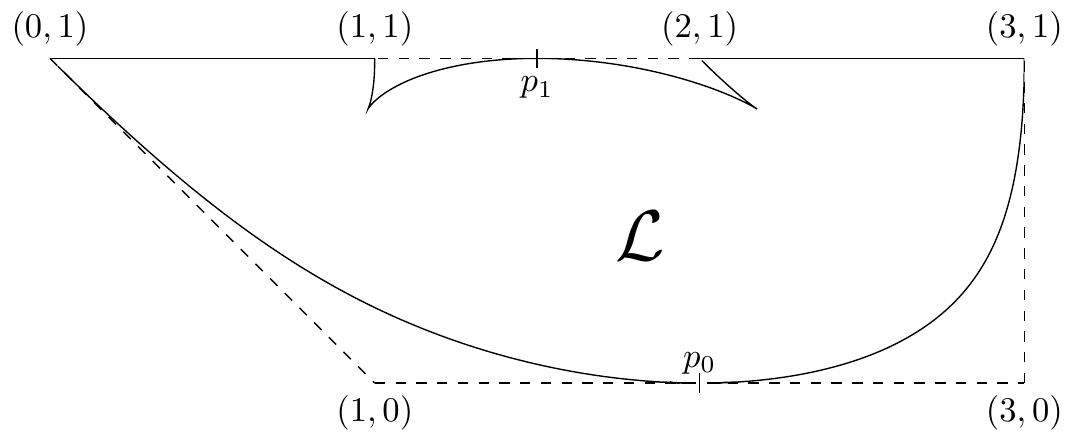}}
\caption{$\LL$ when $\varphi(x) = \frac12$ for all
$x \in [0,1] \cup [2,3]$.}
\label{figex2}
\end{figure}

Above we state that $p_1 = (\frac32,1)$ is in $\EE_0$ and given by
$p_1 = (\chi_\EE(\frac32),\eta_\EE(\frac32))$ where $\frac32 \in R_0$.
Case (5) of lemma \ref{lemGeoObv} thus implies that the edge curve
behaves locally like a parabola at
$(\frac32,1) = (\chi_\EE(\frac32),\eta_\EE(\frac32))$, and is
tangent to the upper boundary at this point. This can clearly be
seen in figure \ref{figex2}. We also state that
$(1,1) \to^\text{e} p_1$ is in $\EE_\mu$. Note from the figure,
though we do not attempt to prove this rigorously, that 
this section contains a cusp. As detailed in cases (1) and (2) of
lemma \ref{lemGeoObv}, the cusp is necessarily of first order and
must correspond to a root of multiplicity $3$ in $R_\mu$, and all other
points of $(1,1) \to^\text{e} p_1$ correspond to roots of multiplicity
$2$. Similarly for $p_1 \to^\text{e} (2,1)$. Similar cusps in the edge
correspond to a root of multiplicity $3$ in $R_{\l-\mu}$, as
detailed in case (4) of lemma \ref{lemGeoObv},
but we do not give an example here.

\subsubsection{\bf Example 3:}
\label{secEx3}

Take $\varphi(x) = 1$ for all $x \in [0,\frac12] \cup [1,\frac32]$. Then,
\begin{align*}
C(w) &= \log(w) - \log(w-1/2) + \log(w-1) - \log(w-3/2)
\text{ for all } w \in \C \setminus \R, \\
R &= \R, \; R_\mu = (-\infty,0) \cup (1/2,3/4) \cup (3/4,1) \cup (3/2,+\infty), \\
R_{\l-\mu} &= (0,1/2) \cup (1,3/2),
\; R_0 = \{3/4\},
\; R_1 = \{1/2,3/2\},
\; R_2 = \{0,1\}, \\
\chi_\EE(t) &= t - 2 (t-1/2) (t-3/2)
\frac{t (t-1) - (t-1/2) (t-3/2)}{(t-1) (t-3/2) + t (t-1/2)}
\text{ for all } t \in R = \R, \\
\eta_\EE(t) &= 1 - 2 \frac{(t (t-1) - (t-1/2) (t-3/2))^2}{(t-1) (t-3/2) + t (t-1/2)}
\text{ for all } t \in R = \R.
\end{align*}

Part (1) of lemma \ref{lemBddEdge3} implies that lemmas \ref{lemBddEdge0}
and \ref{lemBddEdge1} give a complete description of $\partial \LL$.
Plotting $\partial \LL$ gives figure \ref{figex3}. Following $\partial \LL$
in a clockwise direction from the bottom:
\begin{itemize}
\item
$p_0 = (\frac54,0)  = (\frac12 + \int_a^b x \varphi(x) dx, 0)$.
\item
$p_0 \to^\text{e} p_1$ is in $\EE_\mu$ and
given by $t \mapsto (\chi_\EE(t),\eta_\EE(t))$ for all
$t \in (-\infty,0) \subset R_\mu$.
\item
$p_1 = (\frac34,\frac14)$ is in $\EE_2$ and given by
$p_1 = (\chi_\EE(0),\eta_\EE(0))$ where $0 \in R_2$.
\item
$p_1 \to^\text{e} p_2$ is in $\EE_{\l-\mu}$
and given by $t \mapsto (\chi_\EE(t),\eta_\EE(t))$ for all
$t \in (0,\frac12) \subset R_{\l-\mu}$.
\item
$p_2 = (\frac12,\frac34)$ is in $\EE_1$ and given by
$p_2 = (\chi_\EE(\frac12),\eta_\EE(\frac12))$ where $\frac12 \in R_1$.
\item
$p_2 \to^\text{e} p_3$ is in $\EE_\mu$ and
given by $t \mapsto (\chi_\EE(t),\eta_\EE(t))$ for all
$t \in (\frac12,\frac34) \subset R_\mu$.
\item
$p_3 = (\frac34,1)$ is in $\EE_0$ and given by
$p_3 = (\chi_\EE(\frac34),\eta_\EE(\frac34))$ where $\frac34 \in R_0$.
\item
$p_3 \to^\text{e} p_4$ is in $\EE_\mu$ and
given by $t \mapsto (\chi_\EE(t),\eta_\EE(t))$ for all
$t \in (\frac34,1) \subset R_\mu$.
\item
$p_4 = (\frac54,\frac34)$ is in $\EE_2$ and given by
$p_4 = (\chi_\EE(1),\eta_\EE(1))$ where $1 \in R_2$.
\item
$p_4 \to^\text{e} p_5$ is in $\EE_{\l-\mu}$
and given by $t \mapsto (\chi_\EE(t),\eta_\EE(t))$ for all
$t \in (1,\frac32) \subset R_{\l-\mu}$.
\item
$p_5 = (\frac32,\frac14)$ is in $\EE_1$ and given by
$p_5 = (\chi_\EE(\frac32),\eta_\EE(\frac32))$ where $\frac32 \in R_1$.
\item
$p_5 \to^\text{e} p_0$ is in $\EE_\mu$ and
given by $t \mapsto (\chi_\EE(t),\eta_\EE(t))$ for all
$t \in (\frac32,+\infty) \subset R_\mu$.
\end{itemize}
Above, $p_0$ follows from lemma \ref{lemBddEdge0}, and the
remainder follows from lemma \ref{lemBddEdge1}.

\begin{figure}
\centering
\mbox{\includegraphics{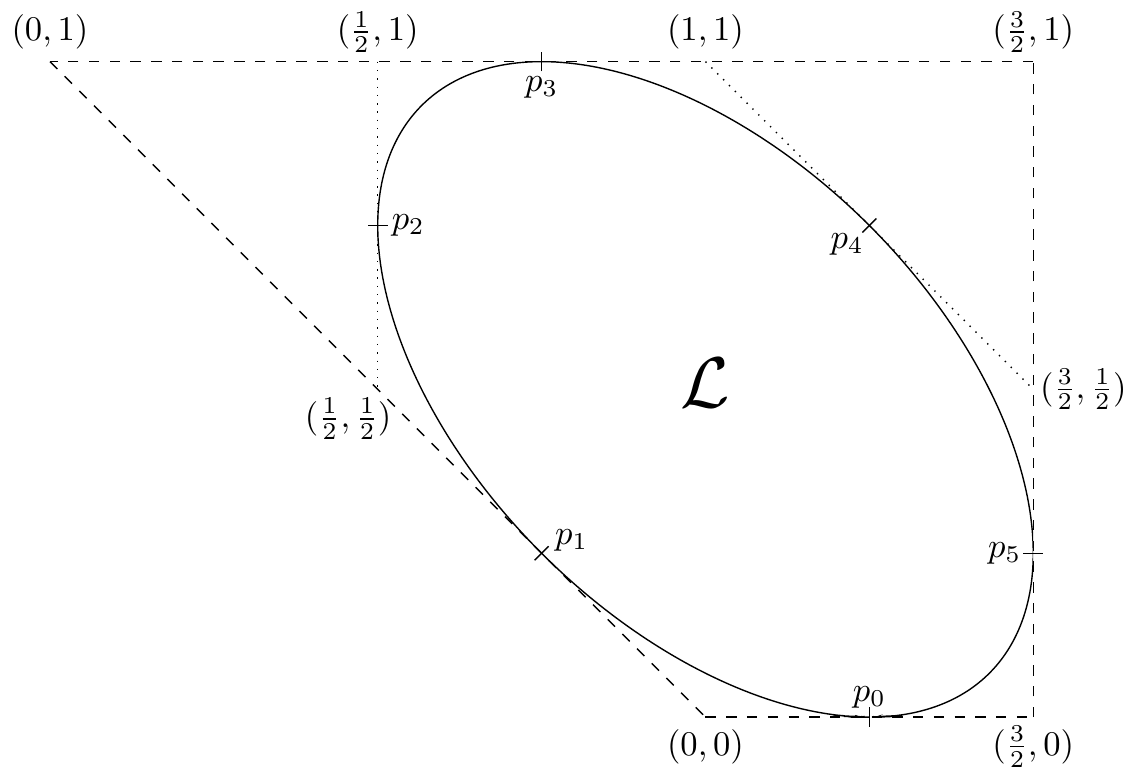}}
\caption{$\LL$ when $\varphi(x) = 1$ for all
$x \in [0,\frac12] \cup [1,\frac32]$.}
\label{figex3}
\end{figure}

Above we state that $p_5 = (\frac32,\frac14)$ is in $\EE_1$ and given by
$p_5 = (\chi_\EE(\frac32),\eta_\EE(\frac32))$ where $\frac32 \in R_1$.
It is not difficult to show that $t = \frac32 \in R_1$ corresponds
to a root of multiplicity $1$. Thus case (6) of lemma \ref{lemGeoObv}
implies that the edge curve behaves locally like a parabola at
$p_5$ with tangent vector $(0,1)$. This can
clearly be seen in the figure. Similarly for
$p_2 = (\frac12,\frac34) = (\chi_\EE(\frac12),\eta_\EE(\frac12)) \in \EE_1$,
where now we have drawn a dotted line to represent the tangent for
clarity. Similarly for
$p_4 = (\frac54,\frac34) = (\chi_\EE(1),\eta_\EE(1))$ and
$p_1 = (\frac34,\frac14) = (\chi_\EE(0),\eta_\EE(0))$ which are in $\EE_2$,
except now the tangent vector is $(1,-1)$ (see case (8) of lemma
\ref{lemGeoObv}).

Finally, recall that the asymptotic shape of the frozen boundary of the
rescaled regular hexagon is the {\em Arctic circle}, as shown on the right
of figure \ref{figArctic}. Recall that we consider tilings of the regular
hexagon as tilings of the half-hexagon by adding deterministic
lozenges/particles, as shown on the right of figure \ref{figEquivIntPart}.
Also, recall that by shifting these tilings, as described at the
beginning of section \ref{sectdsodgtp}, equivalent Gelfand-Tsetlin patterns
are obtained. The shifted asymptotic shape of the frozen boundary of the
regular hexagon is shown on the right of figure \ref{figGTAsyShape}. This
is identical to figure \ref{figex3}. We therefore have recovered the Arctic
circle boundary theorem for the regular hexagon.

\subsubsection{\bf Example 4:}
\label{secEx4}

Take $\varphi (x) := 1$ for all
$x \in [0,\frac13] \cup [1,\frac43] \cup [c,c+\frac13]$,
where $c := \frac1{12} (23 + \sqrt{217}) > \frac43$. Define
$c_1 := (\frac12 + \frac{c}3) +
[(\frac12 + \frac{c}3)^2 - \frac{4}9 (c + \frac13)]^{1/2})$
and $c_2 := (\frac12 + \frac{c}3) -
[(\frac12 + \frac{c}3)^2 - \frac{4}9 (c + \frac13)]^{1/2})$.
Then $c_1 \in (\frac43,c)$, $c_2 \in (\frac13,1)$, and,
\begin{align*}
C(w) &= \log(w) - \log(w-1/3) + \log(w-1) - \log(w-4/3) \\
&\;\;\; + \log(w-c) - \log(w-c-1/3)
\text{ for all } w \in \C \setminus \R, \\
R &= \R, \; R_\mu = (-\infty,0) \cup (1/3,c_2) \cup (c_2,1)
\cup (4/3,c_1) \cup (c_1,c) \cup (c+1/3,+\infty), \\
R_{\l-\mu} &= (0,1/3) \cup (1,4/3) \cup (c,c+1/3), \\
R_0 &= \{c_1,c_2\},
\; R_1 = \{1/3,4/3,c+1/3\},
\; R_2 = \{0,1,c\}, \\
\chi_\EE(t) &= t - 3 (t-1/3) (t-4/3) (t-c-1/3) \frac{A(t)}{B(t)},
\hspace{0.5cm}
\eta_\EE(t) = 1 - 3 \frac{A(t)^2}{B(t)}, \\
A(t) &:= t (t-1) (t-c) - (t-1/3) (t-4/3) (t-c-1/3), \\
B(t) &:= (t-1) (t-4/3) (t-c) (t-c-1/3) + t (t-1/3) (t-c) (t-c-1/3) \\
&\;\;\; + t (t-1/3) (t-1) (t-4/3).
\end{align*}

Part (1) of lemma \ref{lemBddEdge3} implies that lemmas \ref{lemBddEdge0}
and \ref{lemBddEdge1} give a complete description of $\partial \LL$.
Plotting $\partial \LL$ gives figure \ref{figex4}. Following $\partial \LL$
in a clockwise direction from the bottom:
\begin{itemize}
\item
$p_0 = (1 + \frac13 c,0)  = (\frac12 + \int_a^b x \varphi(x) dx, 0)$.
\item
$p_0 \to^\text{e} p_1$ is in $\EE_\mu$ and
given by $t \mapsto (\chi_\EE(t),\eta_\EE(t))$ for all
$t \in (-\infty,0) \subset R_\mu$.
\item
$p_1 = (\frac49 + \frac4{27c}, \frac59 - \frac4{27c})$ is in $\EE_2$ and given by
$p_1 = (\chi_\EE(0),\eta_\EE(0))$ where $0 \in R_2$.
\item
$p_1 \to^\text{e} p_2$ is in $\EE_{\l-\mu}$
and given by $t \mapsto (\chi_\EE(t),\eta_\EE(t))$ for all
$t \in (0,\frac13) \subset R_{\l-\mu}$.
\item
$p_2 = (\frac13, \frac79 + \frac2{27c})$ is in $\EE_1$ and given by
$p_2 = (\chi_\EE(\frac13),\eta_\EE(\frac13))$ where $\frac13 \in R_1$.
\item
$p_2 \to^\text{e} p_3$ is in $\EE_\mu$ and
given by $t \mapsto (\chi_\EE(t),\eta_\EE(t))$ for all
$t \in (\frac13,c_2) \subset R_\mu$.
\item
$p_3 = (c_2,1)$ is in $\EE_0$ and given by
$p_3 = (\chi_\EE(c_2),\eta_\EE(c_2))$ where $c_2 \in R_0$.
\item
$p_3 \to^\text{e} p_4$ is in $\EE_\mu$ and
given by $t \mapsto (\chi_\EE(t),\eta_\EE(t))$ for all
$t \in (c_2,1) \subset R_\mu$.
\item
$p_4 = (\frac{11}9 + \frac2{27(c-1)},\frac79 - \frac2{27(c-1)})$ is in $\EE_2$,
given by $p_4 = (\chi_\EE(1),\eta_\EE(1))$ where $1 \in R_2$.
\item
$p_4 \to^\text{e} p_5$ is in $\EE_{\l-\mu}$
and given by $t \mapsto (\chi_\EE(t),\eta_\EE(t))$ for all
$t \in (1,\frac43) \subset R_{\l-\mu}$.
\item
$p_5 = (\frac43,\frac59 + \frac4{27(c-1)})$ is in $\EE_1$ and given by
$p_5 = (\chi_\EE(\frac43),\eta_\EE(\frac43))$ where $\frac43 \in R_1$.
\item
$p_5 \to^\text{e} p_6$ is in $\EE_\mu$ and
given by $t \mapsto (\chi_\EE(t),\eta_\EE(t))$ for all
$t \in (\frac43,c_1) \subset R_\mu$.
\item
$p_6 = (c_1,1)$ is in $\EE_0$ and given by
$p_6 = (\chi_\EE(c_1),\eta_\EE(c_1))$ where $c_1 \in R_0$.
\item
$p_6 \to^\text{e} p_7$ is in $\EE_\mu$ and
given by $t \mapsto (\chi_\EE(t),\eta_\EE(t))$ for all
$t \in (c_1,c) \subset R_\mu$.
\item
$p_7 = (c + \frac{(c - \frac13) (c - \frac43)}{3 c (c-1)},
1 - \frac{(c - \frac13) (c - \frac43)}{3 c (c-1)}) \in \EE_2$, given by
$p_7 = (\chi_\EE(c),\eta_\EE(c))$ for $c \in R_2$.
\item
$p_7 \to^\text{e} p_8$ is in $\EE_{\l-\mu}$ and
given by $t \mapsto (\chi_\EE(t),\eta_\EE(t))$ for all
$t \in (c,c+\frac13) \subset R_{\l-\mu}$.
\item
$p_8 = (c + \frac13, 1 - \frac{(c + \frac13) (c - \frac23)}{3 c (c-1)}) \in \EE_1$,
given by $p_7 = (\chi_\EE(c+\frac13),\eta_\EE(c+\frac13))$ for $c+\frac13 \in R_1$.
\item
$p_8 \to^\text{e} p_0$ is in $\EE_\mu$ and
given by $t \mapsto (\chi_\EE(t),\eta_\EE(t))$ for all
$t \in (c+\frac13,+\infty) \subset R_\mu$.
\end{itemize}
Above, $p_0$ follows from lemma \ref{lemBddEdge0}, and the
remainder follows from lemma \ref{lemBddEdge1}.

\begin{figure}
\centering
\mbox{\includegraphics{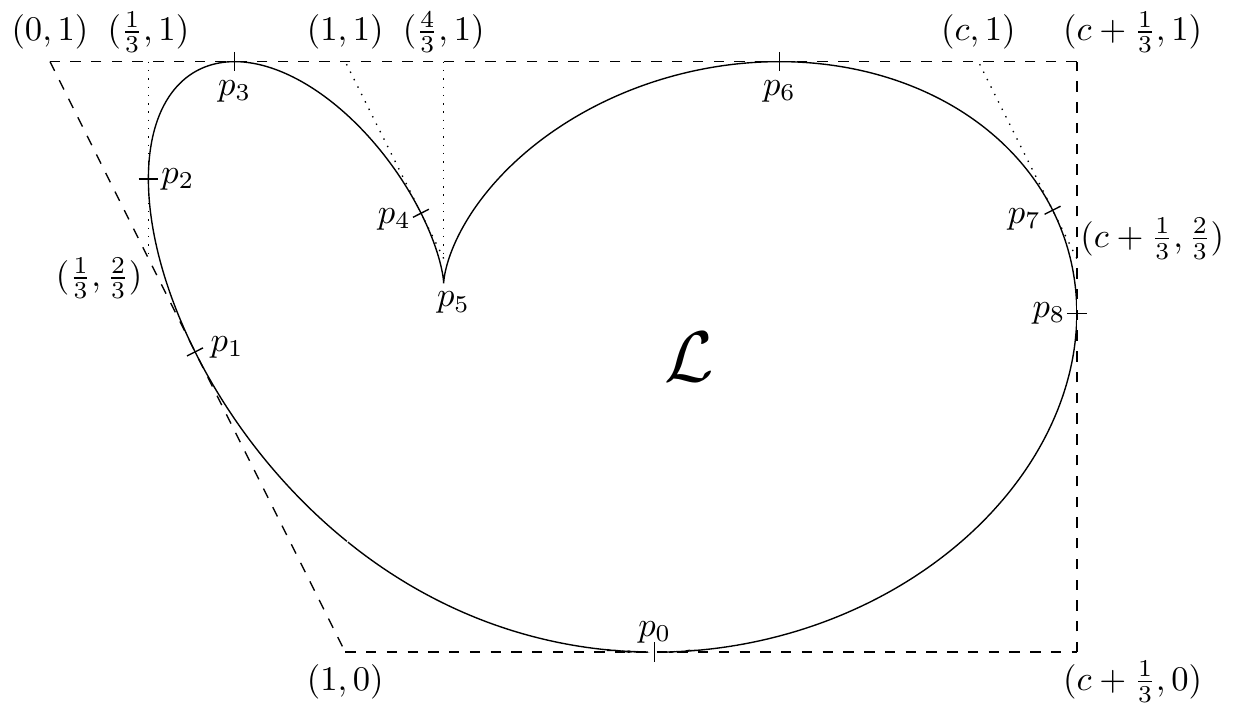}}
\caption{$\LL$ when $\varphi (x) := 1$ for all
$x \in [0,\frac13] \cup [1,\frac43] \cup [c,c+\frac13]$,
where $c := \frac1{12} (23 + \sqrt{217})$. The vertical
direction has been scaled by $2$ for clarity.}
\label{figex4}
\end{figure}

Above we state that $p_5 = (\frac43,\frac59 + \frac4{27(c-1)})$ is in
$\EE_1$ and given by $p_5 = (\chi_\EE(\frac43),\eta_\EE(\frac43))$
where $\frac43 \in R_1$. It is not difficult to show that
$t = \frac43 \in R_1$ corresponds to a root of multiplicity $2$. (Here
we use the fact that $c$, as defined above, is a root of $6 c^2 - 23 c + 13$.)
Thus case (7) of lemma \ref{lemGeoObv} implies that the edge curve
has a cusp of first order at $p_5$ with orientation $\mathbf{x}(t) = (0,1)$.
This can clearly be seen in the figure. Cusps in the edge with
orientation $\mathbf{x}(t) = (1,-1)$ correspond to roots, $t \in R_2$,
of multiplicity $2$, as detailed in case (9) of lemma \ref{lemGeoObv},
but we do not give an example here.

\subsubsection{\bf Example 5:}
\label{secEx5}

Take $\varphi (x) := 1-x$ for all $x \in [0,1]$, and $\varphi (x) := 1+x$ for all
$x \in [-1,0]$. Then,
\begin{align*}
C(w) &= (w+1) \log(w+1) - 2w \log(w) + (w-1) \log(w-1)
\text{ for all } w \in \C \setminus \R, \\
R &= \R \setminus [-1,1],
\; R_\mu = R,
\; R_{\l-\mu} = \emptyset,
\; R_0 = \emptyset,
\; R_1 = \emptyset,
\; R_2 = \emptyset, \\
\chi_\EE(t) &= t + \frac{|t+1|^{t+1} |t|^{-2t} |t-1|^{t-1} - 1}
{|t+1|^{t+1} |t|^{-2t} |t-1|^{t-1} (\log|t+1| - 2\log|t| + \log|t-1|)}
\text{ for all } t \in R, \\
\eta_\EE(t) &= 1 + \frac{(|t+1|^{t+1} |t|^{-2t} |t-1|^{t-1} - 1)^2}
{|t+1|^{t+1} |t|^{-2t} |t-1|^{t-1} (\log|t+1| - 2\log|t| + \log|t-1|)}
\text{ for all } t \in R.
\end{align*}

Note that each $t \in (-1,0) \cup (0,1)$ satisfies case (1) of
lemma \ref{lemBddEdge2}, and so $(t,1) \in \partial \LL$ for all
such $t$. However, the points $\{-1,0,1\}$ do not satisfy any of
the cases of lemma \ref{lemBddEdge2}. Thus part (2) of lemma
\ref{lemBddEdge3} does not apply, and so lemmas \ref{lemBddEdge0},
\ref{lemBddEdge1}, and \ref{lemBddEdge2} only give a partial
description of $\partial \LL$.

In order to get a complete description, following the proof of part
(2) of lemma \ref{lemBddEdge3}, we must consider all possible limits
of sequences of the form $\{(\chi_\LL(w_n),\eta_\LL(w_n))\}_{n\ge1}$ as
$n \to \infty$, where $\{w_n\}_{n\ge1} \subset \mathbb{H}$ satisfies
$w_n \to -1$, $0$ or $1$. Denote, as in lemma \ref{lemBddEdge3},
$u_n := \text{Re}(w_n)$, $v_n := \text{Im}(w_n)$,
$R_n := \text{Re}(C(w_n))$, and $I_n := - \text{Im}(C(w_n))$. Then,
using the above expression for $C$, it is not difficult to see that
$R_n \sim O(1)$. Also, for all $n$ sufficiently large,
$\sin(I_n) > - \frac18 v_n \log((u_n+1)^2 + v_n^2)$ when $w_n \to -1$,
$\sin(I_n) > - \frac18 v_n \log(u_n^2 + v_n^2)$ when $w_n \to 0$,
and $\sin(I_n) > - \frac18 v_n \log((u_n-1)^2 + v_n^2)$ when $w_n \to 1$.
Equations (\ref{eqlemBddEdge31}) and (\ref{eqlemBddEdge32}) finally
give $(\chi_\LL(w_n),\eta_\LL(w_n)) \to (-1,1)$ when $w_n \to -1$,
$(\chi_\LL(w_n),\eta_\LL(w_n)) \to (0,1)$ when $w_n \to 0$,
and $(\chi_\LL(w_n),\eta_\LL(w_n)) \to (1,1)$ when $w_n \to 1$.
These limits, combined with lemmas \ref{lemBddEdge0}, \ref{lemBddEdge1}
and \ref{lemBddEdge2} give a complete description. Plotting this we
get figure \ref{figex5}:
\begin{itemize}
\item
$p_0 = (\frac12,0) = (\frac12 + \int_a^b x \varphi(x) dx, 0)$.
\item
$p_0 \to^\text{e} (-1,1)$ is in $\EE_\mu$ and given by
$t \mapsto (\chi_\EE(t),\eta_\EE(t))$ for all
$t \in (-\infty,-1) \subset R_\mu$.
\item
$(-1,1) \to^\text{i} (1,1)$ is given by $t \mapsto (t,1)$ for all
$t \in [-1,1]$.
\item
$(1,1) \to^\text{e} p_0$ is in $\EE_\mu$ and given by
$t \mapsto (\chi_\EE(t),\eta_\EE(t))$ for all
$t \in (1,+\infty) \subset R_\mu$.
\end{itemize}
Above, $p_0$ follows from lemma \ref{lemBddEdge0}, and $p_0 \to^\text{e} (-1,1)$
and $(1,1) \to^\text{e} p_0$ follow from lemma \ref{lemBddEdge1}. Also,
$(-1,1) \to^\text{e} (0,1)$ and $(0,1) \to^\text{e} (1,1)$ follow from
lemma \ref{lemBddEdge2}, and the points $\{(-1,1), (0,1), (1,1)\}$ follow
from the direct calculations given above.

\begin{figure}
\centering
\mbox{\includegraphics{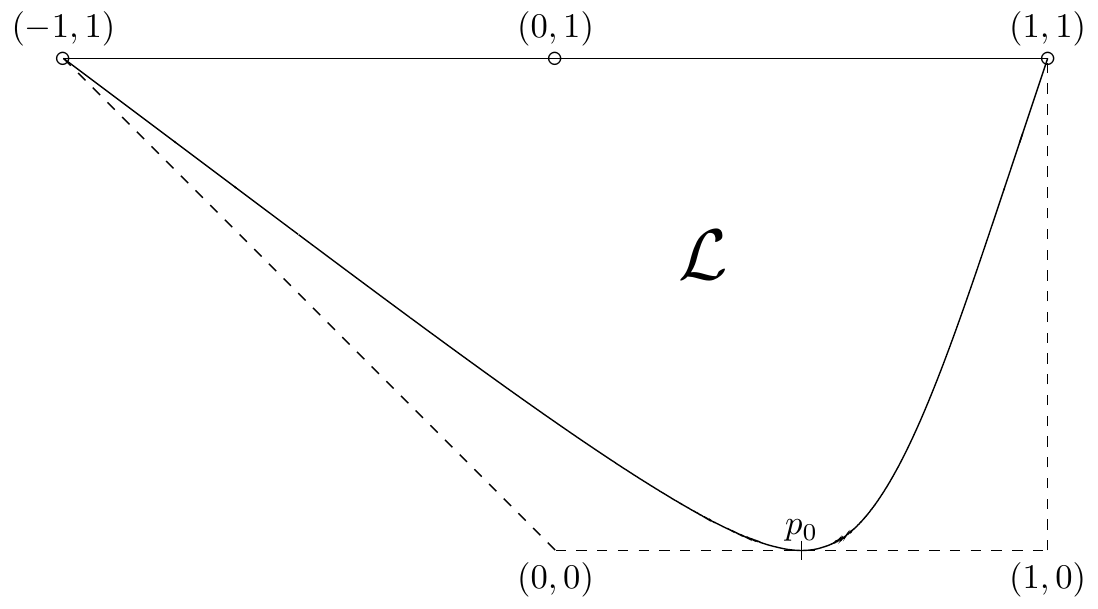}}
\caption{$\LL$ when $\varphi (x) := 1-x$ for all $x \in [0,1]$, and
$\varphi (x) := 1+x$ for all $x \in [-1,0]$.}
\label{figex5}
\end{figure}

\subsubsection{\bf Example 6:}
\label{secEx6}

Take $\varphi (x) = \frac{15}{16} (x-1)^2 (x+1)^2$ for all $x \in [-1,1]$.
Then,
\begin{align*}
C(w) &= \frac{15}{16} \left( \frac{10}3 w - 2 w^3
+ (w^2-1)^2 (\log(w+1) - \log(w-1)) \right)
\text{ for all } w \in \C \setminus [-1,1], \\
C'(w) &= \frac{15}{16} \left( \frac{16}3 - 8 w^2
+ 4 w (w^2-1) (\log(w+1) - \log(w-1)) \right)
\text{ for all } w \in \C \setminus [-1,1], \\
R &= \R \setminus [-1,1],
\; R_\mu = R,
\; R_{\l-\mu} = \emptyset,
\; R_0 = \emptyset,
\; R_1 = \emptyset,
\; R_2 = \emptyset, \\
\chi_\EE(t) &= t + \frac{e^{C(t)}-1}{e^{C(t)} C'(t)}
\;\;\; \text{and} \;\;\;
\eta_\EE(t) = 1 + \frac{(e^{C(t)}-1)^2}{e^{C(t)} C'(t)}
\text{ for all } t \in R = \R \setminus [-1,1].
\end{align*}

Note that each $t \in (-1,1)$ satisfies case (1) of lemma
\ref{lemBddEdge2}, and so $(t,1) \in \partial \LL$ for all
such $t$. However, the points $\{-1,1\}$ do not satisfy any of
the cases of lemma \ref{lemBddEdge2}. Thus part (2) of lemma
\ref{lemBddEdge3} does not apply, and so lemmas \ref{lemBddEdge0},
\ref{lemBddEdge1}, and \ref{lemBddEdge2} only give a partial
description of $\partial \LL$. This can clearly be seen in figure
\ref{figex6}, where we have plotted those parts of the
boundary that we obtain from lemmas \ref{lemBddEdge0},
\ref{lemBddEdge1}, and \ref{lemBddEdge2}:
\begin{itemize}
\item
$p_0 = (\frac12,0) = (\frac12 + \int_a^b x \varphi(x) dx, 0)$.
\item
$p_0 \to^\text{e} p_1$ is in $\EE_\mu$ and given by
$t \mapsto (\chi_\EE(t),\eta_\EE(t))$ for all
$t \in (-\infty,-1) \subset R_\mu$.
\item
$(-1,1) \to^\text{e} (1,1)$ is given by $t \mapsto (t,1)$ for all
$t \in (-1,1)$.
\item
$p_2 \to^\text{e} p_0$ is in $\EE_\mu$ and given by
$t \mapsto (\chi_\EE(t),\eta_\EE(t))$ for all
$t \in (1,+\infty) \subset R_\mu$.
\end{itemize}
Above $p_1 := \lim_{t \uparrow -1} (\chi_\EE(t),\eta_\EE(t)) \sim (-0.004,0.290)$
and $p_2 := \lim_{t \downarrow 1} (\chi_\EE(t),\eta_\EE(t)) \sim (0.714,0.290)$.
Also, $p_0$ follows from lemma \ref{lemBddEdge0}, $p_0 \to^\text{e} p_1$
and $p_2 \to^\text{e} p_0$ follow from lemma \ref{lemBddEdge1}, and
$(-1,1) \to^\text{e} (1,1)$ follows from lemma \ref{lemBddEdge2}.

\begin{figure}
\centering
\mbox{\includegraphics{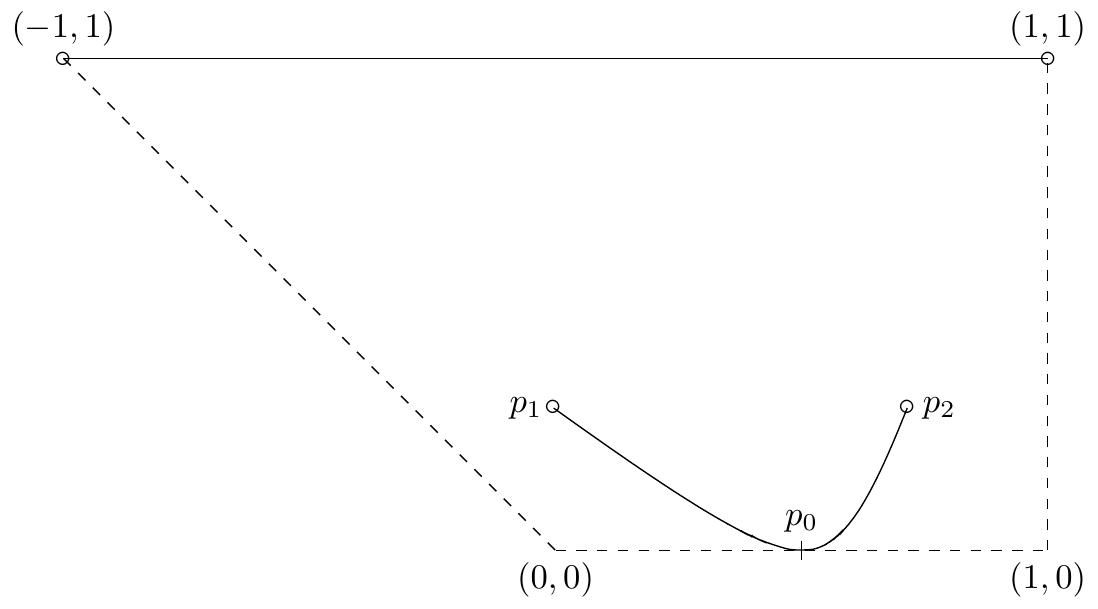}}
\caption{Parts of $\partial \LL$ when $\varphi (x) = \frac{15}{16} (x-1)^2 (x+1)^2$
for all $x \in [-1,1]$.}
\label{figex6}
\end{figure}

In order to get a complete description, following the proof of part
(2) of lemma \ref{lemBddEdge3}, we must consider all possible limits
of sequences of the form $\{(\chi_\LL(w_n),\eta_\LL(w_n))\}_{n\ge1}$ as
$n \to \infty$, where $\{w_n\}_{n\ge1} \subset \mathbb{H}$ satisfies
$w_n \to -1$ or $1$. This calculation is beyond the scope of this
paper, however, since the technicalities involved in extending
lemma \ref{lemBddEdge2} to cover this situation are highly non-trivial.
In the paper \cite{Duse15b}, we make heavy use of the theory of
singular integrals to examine this and other, surprisingly subtle,
situations. We obtain a complete description of $\partial \LL$
for a broad class of measures.

\section{The behaviour of the roots of \texorpdfstring{$f'$}{Lg}}
\label{sectrofn'}

In this section we examine the behaviour of the roots of the function
$f'$ given in equation (\ref{eqf'}). More generally, it is advantageous to
examine the behaviour of the roots of Cauchy transforms of
signed-measures of a particular type: Let $A := A_1 \cup \cdots \cup A_p$ ($p\ge1$)
and $B := B_1 \cup \cdots \cup B_q$ ($q\ge0$) be unions of disjoint closed intervals.
Assume that each $A_i$ and $B_j$ are disjoint, except possibly at their end-points.
Let $\nu^+$ and $\nu^-$ be non-negative measures with $\supp(\nu^+) \subset A$ and
$\supp(\nu^-) \subset B$. Finally, assume that the end-points of each $A_i$ are
contained in $\supp(\nu^+)$, and $0 < \nu^+[A_i] < \infty$. Similarly for $\nu^-$
and each $B_j$. Therefore $\nu^+$ is a positive measure, since $p\ge1$.
Also, $\nu^-$ is positive when $q\ge1$, and is $0$ when $q=0$. The Cauchy
transform of the signed-measure, $\nu := \nu^+ - \nu^-$, is defined by,
\begin{equation}
\label{eqf'Gen}
g(w) := \int_A \frac{\nu^+[dx]}{w-x} - \int_B \frac{\nu^-[dx]}{w-x},
\end{equation}
for all $w \in \C \setminus (A \cup B)$. The main result of this section examines
the behaviour of the roots of this Cauchy transform:
\begin{lem}
\label{lemRootsg}
Let $I \subset \R \setminus (A \cup B)$ be an open interval which contains at least
$k\ge0$ roots of $g$, counting multiplicities. First assume that $\nu^+[A] > \nu^-[B]$.
Then, counting multiplicities,
\begin{enumerate}
\item 
$g$ has at most $p+q-1$ roots in $\C \setminus (A \cup B)$.
\item
$g$ has at most $p+q-2$ roots in $I$ whenever $p+q-1$ is even and either
$\{\inf I,\sup I\} \subset A$ or $\{\inf I,\sup I\} \subset B \cup \{\pm \infty\}$.
\item
$g$ has at most $p+q-2$ roots in $I$ whenever $p+q-1$ is odd and one of
$\{\inf I,\sup I\}$ is in $A$, with the other in $B \cup \{\pm \infty\}$.
\item
$g$ has at most $p+q-k-2$ roots in $\C \setminus (A \cup B \cup I)$ whenever $k$
is even and either $\{\inf I,\sup I\} \subset A$ or
$\{\inf I,\sup I\} \subset B \cup \{\pm \infty\}$.
\item
$g$ has at most $p+q-k-2$ roots in $\C \setminus (A \cup B \cup I)$ whenever $k$
is odd and one of $\{\inf I,\sup I\}$ is in $A$, with the other in
$B \cup \{\pm \infty\}$.
\end{enumerate}
Next assume that $\nu^+[A] = \nu^-[B]$. Then, counting multiplicities,
\begin{enumerate}
\setcounter{enumi}{5}
\item 
$g$ has at most $p+q-2$ roots in $\C \setminus (A \cup B)$.
\item
$g$ has at most $p+q-3$ roots in $I$ whenever $p+q-2$ is even and either
$\{\inf I,\sup I\} \subset A$ or $\{\inf I,\sup I\} \subset B$.
\item
$g$ has at most $p+q-3$ roots in $I$ whenever $p+q-2$ is odd and one of
$\{\inf I,\sup I\}$ is in $A$, with the other in $B$.
\end{enumerate}
\end{lem}

\begin{proof}
We will first show:
\begin{enumerate}
\item[(a)]
The required results hold when $\nu$ is of the form,
\begin{equation*}
\nu^+ := \sum_{i=1}^M \nu_i^+ \; \d_{a_i}
\hspace{0.5cm} \text{and} \hspace{0.5cm}
\nu^- := \sum_{j=1}^N \nu_j^- \; \d_{b_j},
\end{equation*}
where $M \ge 1$, $N \ge 0$, $\{a_1,\ldots,a_M,b_1,\ldots,b_N\} \subset \R$ are
distinct, $\nu_i^+ > 0$ and $\nu_j^- > 0$ for all $i,j$, and each of
$\{ \{a_i \in A_1\},\ldots,\{a_i \in A_p\},\{b_j \in B_1\},\ldots,\{b_j \in B_q\} \}$
are non-empty.
\end{enumerate}

Next we consider the general case. For all $n\ge1$, define the measures,
\begin{equation*}
\nu_n^+ := \sum_{i=1}^{M_n} \nu_{i,n}^+ \; \d_{a_{i,n}}
\hspace{0.5cm} \text{and} \hspace{0.5cm}
\nu_n^- := \sum_{j=1}^{N_n} \nu_{j,n}^- \; \d_{b_{j,n}},
\end{equation*}
where $\{a_{1,n},\ldots,a_{M_n,n}\} \subset A$ and
$\{b_{1,n},\ldots,b_{N_n,n}\} \subset B$,
$\{a_{1,n},\ldots,a_{M_n,n},b_{1,n},\ldots,b_{N_n,n}\}$ are distinct, and
$\nu_{i,n}^+ > 0$ and $\nu_{j,n}^- > 0$ for all $i,j$. These are constructed
such that $\nu_n^+ \to \nu^+$ and $\nu_n^- \to \nu^-$ as $n \to \infty$, in the
sense of weak convergence. Also define,
\begin{equation}
\label{eqfn'Gen}
g_n(w) := \int_A \frac{\nu_n^+[dx]}{w-x} - \int_B \frac{\nu_n^-[dx]}{w-x},
\end{equation}
for all $n \ge 1$ and $w \in \C \setminus (A \cup B)$. We will show that:
\begin{enumerate}
\item[(b)]
Let $w_c \in \C \setminus (A \cup B)$ be a root of
$g$ of multiplicity $k\ge1$, and choose $\e>0$ such that
$B(w_c,2\e) \subset \C \setminus (A \cup B)$ and $w_c$ is the unique root
of $g$ in $B(w_c,2\e)$. (This is always possible since the roots of an
analytic function are isolated.) Then, for all $n$ sufficiently large,
$g_n$ has $k$ roots in $B(w_c,\e)$, counting multiplicities.
\end{enumerate}
Finally note, part (a) implies that the required results hold for each $g_n$.
Part (b) then easily gives the required results for $g$.

Consider (a). Note,
\begin{equation*}
g(w) = \sum_{i=1}^M \frac{\nu_i^+}{w-a_i} - \sum_{j=1}^N \frac{\nu_j^-}{w-b_j}
= \left( \prod_{k=1}^M \frac1{w-a_k} \right)
\left( \prod_{l=1}^N \frac1{w-b_l} \right) P(w),
\end{equation*}
where $P$ is the polynomial,
\begin{equation*}
P(w) = \sum_{i=1}^M \nu_i^+ \left( \prod_{k \neq i} (w - a_k) \right)
\left( \prod_l (w - b_l) \right)
- \sum_{j=1}^N \nu_j^- \left( \prod_k (w - a_k) \right)
\left( \prod_{l \neq j} (w - b_l) \right).
\end{equation*}
Note that $\nu^+[A] = \nu_1^+ + \cdots + \nu_M^+$ and
$\nu^-[B] = \nu_1^- + \cdots + \nu_N^-$. Thus $P$ has degree $M+N-1$ whenever
$\nu^+[A] > \nu^-[B]$, and degree at most $M+N-2$ whenever $\nu^+[A] = \nu^-[B]$.
Also note that, since $\{a_1,\ldots,a_M,b_1,\ldots,b_N\}$ are distinct, $P$
has no roots in $\{a_1,\ldots,a_M,b_1,\ldots,b_N\}$. Therefore the roots of
$P$ and $g$ coincide, and so $g$ has $M+N-1$ roots in
$\C \setminus \{a_1,\ldots,a_M,b_1,\ldots,b_N\}$ whenever $\nu^+[A] > \nu^-[B]$,
counting multiplicities, and at most $M+N-2$ roots whenever $\nu^+[A] = \nu^-[B]$.

Alternatively note that we can write,
\begin{equation*}
g(w) = \sum_{k=1}^p \sum_{a_i \in A_k} \frac{\nu_i^+}{w-a_i}
- \sum_{l=1}^q \sum_{b_j \in B_l} \frac{\nu_j^-}{w-b_j}.
\end{equation*}
Note that each of $\{ \{a_i \in A_1\},\ldots,\{a_i \in A_p\} \}$ contains at
least $2$ points since the end-points of each interval of
$A = A_1 \cup \cdots \cup A_p$ are contained in
$\supp(\nu^+) = \{a_1,\ldots,a_M\}$. Also note that, since the interiors of
each $A_i$ and $B_j$ are disjoint and $\{a_1,\ldots,a_M,b_1,\ldots,b_N\}$ are
distinct, there are, for example, no elements from $\{b_1,\ldots,b_N\}$ between
elements of $\{a_i \in A_1\}$. Also, letting $a_m > a_l$ denote any two
neighbouring elements of $\{a_i \in A_1\}$,
\begin{equation}
\label{eqlemRootsg1}
\lim_{w \in \R, w \uparrow a_m} g(w) = - \infty
\hspace{0.5cm} \text{and} \hspace{0.5cm}
\lim_{w \in \R, w \downarrow a_l} g(w) = + \infty.
\end{equation}
Thus $g$ has at least $1$ root in $(a_l,a_m)$, counting multiplicities, at
least $|\{a_i \in A_1\}|-1$ roots in $A_1 \setminus \{a_1,\ldots,a_M\}$, and
at least $M-p$ roots in $A \setminus \{a_1,\ldots,a_M\}$ (recall
$\{a_1,\ldots,a_M\} = \supp(\nu^+) \subset A = A_1 \cup \cdots \cup A_p$).
Similarly $g$ has at least $N-q$ roots in $B \setminus \{b_1,\ldots,b_N\}$.
Thus $g$ has at most $M+N-1-(M-p+N-q) = p+q-1$ roots in $\C \setminus (A \cup B)$
whenever $\nu^+[A] > \nu^-[B]$, and at most $p+q-2$ roots whenever
$\nu^+[A] = \nu^-[B]$. We have thus shown (1) and (6) for these types of measures.

Now let $I \subset \R \setminus (A \cup B)$ be an open interval which contains
at least $k\ge0$ roots of $g$, counting multiplicities, with
$\{\inf I,\sup I\} \subset A$. First note that $I = (a_l,a_m)$ for some $a_l > a_m$,
since the end-points of each interval of $A = A_1 \cup \cdots \cup A_p$ are
contained in $\supp(\nu^+) = \{a_1,\ldots,a_M\}$. Equation (\ref{eqlemRootsg1})
thus implies that $g$ has an odd number of roots in $I$, counting multiplicities.
It thus follows from (1) and (6) that $g$ has at most $p+q-2$ roots in $I$ whenever
$\nu^+[A] > \nu^-[B]$ and $p+q-1$ is even, and $g$ has at most $p+q-3$ roots in
$I$ whenever $\nu^+[A] = \nu^-[B]$ and $p+q-2$ is even. Also, whenever $k$ is
even and $\nu^+[A] > \nu^-[B]$, (1) implies that $g$ has at most
$p+q-1-(k+1) = p+q-k-2$ roots in $\C \setminus (A \cup B \cup I)$. The other
possibilities of (2,3,4,5,7,8) follow in similar way. This gives part (a).

Consider (b). This follows from Rouch\'{e}'s Theorem if we can show that
$|g(w)| > |g(w) - g_n(w)|$ for all $n$ sufficiently large and
$w \in \partial B(w_c,\e)$, the boundary of $B(w_c,\e)$. We shall show:
\begin{equation*}
\inf_{w \in \partial B(w_c,\e)} |g(w)| > 0
\hspace{0.5cm} \mbox{and} \hspace{0.5cm}
\lim_{n \to \infty} \sup_{w \in \bar{B}(w_c,\e)} | g(w) - g_n(w) | = 0.
\end{equation*}
The first part follows from the extreme value theorem, since $g$ is
analytic in $B(w_c,2\e)$. Suppose the second part does not hold. Then
there exists an $\xi > 0$ for which, for all $n\ge1$, there exists
some $p_n \ge n$ and $z_n \in \bar{B}(w_c,\e)$ with
$\xi < | g(z_n) - g_{p_n}(z_n) |$. Choosing $\{z_n\}_{n\ge1}$ to be
convergent, and denoting the limit by $z_c$,
\begin{equation*}
\xi < | g(z_n) - g(z_c) | + | g(z_c) - g_{p_n}(z_c) | + | g_{p_n}(z_c) - g_{p_n}(z_n) |,
\end{equation*}
for all $n$. Then, since $B(w_c,2\e) \subset \C \setminus (A \cup B)$
and $z_n,z_c \in \bar{B}(w_c,\e)$,
equation (\ref{eqfn'Gen}) gives
\begin{equation*}
\xi < | g(z_n) - g(z_c) | + | g(z_c) - g_{p_n}(z_c) |
+ \frac1{\e^2} (\nu_{p_n}^+ [A] + \nu_{p_n}^- [B]) |z_n - z_c|,
\end{equation*}
for all $n$. Recall that $g$ is analytic in $B(w_c,2\e)$ and
$z_n \to z_c \in \bar{B}(w_c,\e)$ as $n \to \infty$. Thus the first
term on the right hand side converges to $0$ as $n \to \infty$. Also,
recall that $\nu_n^+ \to \nu^+$ and $\nu_n^- \to \nu^-$ as $n \to \infty$,
in the sense of weak convergence. Thus the third term on the right
hand side converges to $0$ as $n \to \infty$. Finally, weak convergence
and equations (\ref{eqf'Gen}) and (\ref{eqfn'Gen}) imply that the second
term on the right hand side converges to $0$ as $n \to \infty$. The
inequality thus gives a contradiction, and so the second part must hold,
as required.
\end{proof}

We now consider the behaviour of the roots of $f'$. Recall (see equation
(\ref{eqf'})) that $(\chi,\eta) \in [a,b] \times [0,1]$,
$b \ge \chi \ge \chi+\eta-1 \ge a$, and
\begin{equation*}
f'(w)
= \int_\chi^b \frac{\mu[dx]}{w-x}
- \int_{\chi+\eta-1}^\chi \frac{(\l-\mu)[dx]}{w-x}
+ \int_a^{\chi+\eta-1} \frac{\mu[dx]}{w-x},
\end{equation*}
for all $w \in \C \setminus S$, where $S := S_1 \cup S_2 \cup S_3$ and
\begin{equation}
\label{eqSupp}
S_1 := \supp(\mu |_{[\chi,b]}),
\hspace{0.5cm}
S_2 := \supp((\l-\mu) |_{[\chi+\eta-1,\chi]}),
\hspace{0.5cm}
S_3 := \supp(\mu |_{[a,\chi+\eta-1]}).
\end{equation}
Note, hypotheses \ref{hypWeakConv} and \ref{hypunrnsnvn} imply that the
measures are non-negative. Also $\{a,b\} \subset \supp(\mu)$ and $b-a>1$. Thus
either $S_1$ or $S_3$ is non-empty, or both. Also $S_2$ is non-empty
whenever $\eta=0$. This gives
5 possibilities:
\begin{itemize}
\item
$\eta > 0$ and $S_1,S_2,S_3$ are non-empty.
\item
$\eta = 0$ and $S_1,S_2,S_3$ are non-empty.
\item
$\eta > 0$, $S_2$ is non-empty and either $S_1$ or $S_3$ is empty.
\item
$\eta = 0$, $S_2$ is non-empty and either $S_1$ or $S_3$ is empty.
\item
$\eta > 0$ and $S_2$ is empty.
\end{itemize}
Next write the domain of $f'$ as the disjoint union,
\begin{equation}
\label{eqf'domain2}
\C \setminus S = (\C \setminus \R) \cup J \cup K,
\end{equation}
where $J := \cup_{j=1}^4 J_j$, $K := \R \setminus (S \cup J)$, and
\begin{itemize}
\item
$J_1 := (\sup S, + \infty)$.
\item
$J_2 := (-\infty, \inf S)$.
\item
$J_3 := (\sup S_2,\inf S_1)$ whenever $S_1,S_2$ are non-empty and $\inf S_1 > \sup S_2$.
Otherwise $J_3 := \emptyset$.
\item
$J_4 := (\sup S_3,\inf S_2)$ whenever $S_2,S_3$ are non-empty and $\inf S_2 > \sup S_3$.
Otherwise $J_4 := \emptyset$.
\end{itemize}
Note that $K \subset \R$ is open, and so it can be partitioned as
$K = \cup_{k=1}^\infty K_k$, where $\{K_1, K_2,\ldots\}$ is a set of
pairwise disjoint open intervals. This partition is unique up to order, and is
either empty, finite, or countable. An example domain of $f'$, with the
above intervals clearly labelled, is given in figure \ref{figSupports}.

The behaviour of the roots of $f'$ for each of the 5 possibilities is then:
\begin{thm}
\label{thmf'}
Assume $S_1,S_2,S_3$ are non-empty. Then $b > \chi > \chi+\eta-1 > a$,
$J_1 = (b,+\infty)$, $J_2 = (-\infty,a)$, $\chi \in J_3$ whenever
$\chi \in \C \setminus S$, and $\chi+\eta-1 \in J_4$ whenever
$\chi+\eta-1 \in \C \setminus S$. Moreover whenever $\eta>0$, $f'$ has,
counting multiplicities,
\begin{enumerate}
\item
at most $2$ roots in $(\C \setminus \R) \cup J$.
\item
a root in at most one of $\{\C \setminus \R, J_1, J_2, J_3, J_4\}$.
\item
at most $3$ roots in any of $\{K_1,K_2,\ldots\}$.
\item
at least $2$ roots in at most one of $\{K_1,K_2,\ldots\}$.
\item
$0$ roots in $(\C \setminus \R) \cup J$ whenever
$f'$ has at least $2$ roots in one of $\{K_1,K_2,\ldots\}$.
\end{enumerate}
Also whenever $\eta = 0$, $f'$ has, counting multiplicities,
\begin{enumerate}
\setcounter{enumi}{5}
\item
at most $1$ root in $(\C \setminus \R) \cup J_1 \cup J_2$.
\item
$0$ roots in $J_3 \cup J_4$.
\item
at most $1$ root in each of $\{K_1,K_2,\ldots\}$.
\end{enumerate}

Now assume that $S_2,S_3$ are non-empty and $S_1$ is empty (similar considerations
apply when $S_1,S_2$ are non-empty and $S_3$ is empty). Then $b = \chi > \chi+\eta-1 > a$,
$J_1 = (\sup S_2,+\infty)$, $J_2 := (-\infty,a)$, $\chi = b \in J_1$, $J_3 = \emptyset$, and
$\chi+\eta-1 \in J_4$ whenever $\chi+\eta-1 \in \C \setminus S$. Moreover whenever
$\eta>0$, $f'$ has, counting multiplicities,
\begin{enumerate}
\setcounter{enumi}{8}
\item
at most $1$ root in $(\C \setminus \R) \cup J$.
\item
at most $1$ root in each of $\{K_1,K_2,\ldots\}$.
\end{enumerate}
Also whenever $\eta = 0$, $f'$ has, counting multiplicities,
\begin{enumerate}
\setcounter{enumi}{10}
\item
$0$ roots in $(\C \setminus \R) \cup J$.
\item
at most $1$ root in each of $\{K_1,K_2,\ldots\}$.
\end{enumerate}

Finally, assume that $\eta>0$ and $S_2$ empty. Then $J_3 = J_4 = \emptyset$. Also, whenever
$\chi \in \C \setminus S$ and $\chi+\eta-1 \in \C \setminus S$, $\chi$ and $\chi+\eta-1$
are both contained in the same element of $\{J_1,J_2,K_1,K_2,\ldots\}$. Also,
(11) and (12) hold. The case $\eta=0$ and $S_2$ empty never occurs.
\end{thm}

\begin{proof}
First suppose that $S_1,S_2,S_3$ are non-empty. Then, since $\{a,b\} \in \supp(\mu)$,
equation (\ref{eqSupp}) gives $b > \chi > \chi+\eta-1 > a$. Also equation
(\ref{eqf'domain2}) gives $J_1 = (b,+\infty)$, $J_2 := (-\infty,a)$,
$\chi \in J_3$ whenever $\chi \in \C \setminus S$, and $\chi+\eta-1 \in J_4$ whenever
$\chi+\eta-1 \in \C \setminus S$.

Consider the situation depicted on the top of figure \ref{figSupports}. Taking
$\nu^+ := \mu |_{[\chi,b]} + \mu |_{[a,\chi+\eta-1]}$ and
$\nu^- := (\l-\mu) |_{[\chi+\eta-1,\chi]}$, $f'$ satisfies the requirements of
equation (\ref{eqf'Gen}) for any choice of $A = \cup_i A_i$ and $B = \cup_j B_j$
shown in figure \ref{figSupports}. Finally note that $\nu^+[A] > \nu^-[B]$
whenever $\eta > 0$ and $\nu^+[A] = \nu^-[B]$ whenever $\eta = 0$.

Take $\eta > 0$ and consider (1). Taking the first choice of $A$ and $B$ in figure
\ref{figSupports}, part (1) of lemma \ref{lemRootsg} implies that $f'$ has at most
$2$ roots in  $\C \setminus (A \cup B) = (\C \setminus \R) \cup J$ (here $p=2$
and $q=1$), as required. Consider (3). Taking the second choice of $A$ and $B$,
part (1) of lemma \ref{lemRootsg} implies that $f'$ has at most $3$ roots in
$\C \setminus (A \cup B) = (\C \setminus \R) \cup J \cup K_1$ (here
$p=3$ and $q=1$). Thus $f'$ has at most $3$ roots in $K_1$, and similarly for $K_2$,
as required.

Consider (2). First note that non-real roots occur in complex conjugate pairs. Thus,
whenever $f'$ has a root in $\C \setminus \R$, part (1) implies that $f'$ has exactly
$2$ roots in $\C \setminus \R$, and $0$ roots in $J$. Alternatively suppose that $f'$
has a root in $J_1 = (b,+\infty)$. Then, taking $I := J_1$ and the first choice of
$A$ and $B$ in figure \ref{figSupports}, part (5) of lemma \ref{lemRootsg} implies
that $f'$ has $0$ roots in
$\C \setminus (A \cup B \cup I) = (\C \setminus \R) \cup J_2 \cup J_3 \cup J_4$
(here $\inf I = b \in A$, $\sup I = + \infty$, $p=2, q=1$ and $k=1$). The other
possibilities of part (2) follow in a similar way.

Consider (4). Suppose that $f'$ has at least $2$ roots in $K_1$. Taking
$I=K_1$ and third choice of $A$ and $B$, part (4) of lemma \ref{lemRootsg} implies
that $f'$ has at most $1$ root in $\C \setminus (A \cup B \cup I)
= (\C \setminus \R) \cup J \cup K_2$ (here
$\{\inf I,\sup I\} \subset A$, $p=4$, $q=1$ and $k=2$). Thus $f'$ has at most $1$
root in $K_2$ whenever $f'$ has at least $2$ roots in $K_1$, and vice-versa, as
required.

Consider (5). Suppose that $f'$ has at least $2$ roots in $K_1$. Taking
$I=K_1$ and the second choice of $A$ and $B$, part (4) of lemma \ref{lemRootsg}
implies that $f'$ has $0$ roots in $\C \setminus (A \cup B \cup I) =
(\C \setminus \R) \cup J$ (here $\{\inf I,\sup I\} \subset A$,
$p=3$, $q=1$ and $k=2$). Similarly $f'$ has $0$ roots in
$(\C \setminus \R) \cup J$ whenever $K_2$ contains at least $2$
roots, as required.

Now take $\eta = 0$, i.e. $\nu^+[A] = \nu^-[B]$, and consider (6). Taking the
first choice of $A$ and $B$, part (6) of lemma \ref{lemRootsg} implies that
$f'$ has at most $1$ root in $\C \setminus (A \cup B) = (\C \setminus \R) \cup J$
(here $p=2$ and $q=1$), as required. Consider (7). Taking $I = J_3$ and the
first choice of $A$ and $B$, part (8) of lemma \ref{lemRootsg} implies that
$f'$ has $0$ roots in $I = J_3$ (here $\inf I \subset B$, $\sup I \subset A$,
$p=2$ and $q=1$). Similarly $f'$ has $0$ roots in $J_4$, as required.
Finally, consider (8). Taking $I = K_1$ and the
second choice of $A$ and $B$, part (7) of lemma \ref{lemRootsg} implies that
$f'$ has at most $1$ root in $I = K_1$ (here $\{\inf I, \sup I\} \subset A$,
$p=3$ and $q=1$). Similarly for $K_2$, as required.

We have thus shown the required results for the example depicted at the top of
figure \ref{figSupports}. The result for the general situation when $S_1,S_2,S_3$
are non-empty follows using similar constructions of $A = \cup_i A_i$ and
$B = \cup_j B_j$ to those used above. The other cases of the lemma follow
from similar considerations.
\end{proof}

\begin{figure}[t]
\centering
\begin{tikzpicture}

\draw [dotted] (0,0) --++(0,-4.5);
\draw [dotted] (1.5,0) --++(0,-4.5);
\draw [dotted] (2.5,0) --++(0,-4.5);
\draw [dotted] (4,0) --++(0,-4.5);
\draw [dotted] (5,0) --++(0,-4.5);
\draw [dotted] (7,0) --++(0,-4.5);
\draw [dotted] (8,0) --++(0,-4.5);
\draw [dotted] (9.5,0) --++(0,-4.5);
\draw [dotted] (10.5,0) --++(0,-4.5);
\draw [dotted] (12,0) --++(0,-4.5);

\draw [dotted] (-1,0) --++(14,0);
\draw (0,0) --++(1.5,0);
\draw (2.5,0) --++(1.5,0);
\draw (5,0) --++(2,0);
\draw (8,0) --++(1.5,0);
\draw (10.5,0) --++(1.5,0);

\draw (-.75,-.3) node {$J_2$};
\draw (0,.2) node {$a$};
\draw (.75,-.3) node {$S_3$};
\draw (2,-.3) node {$K_2$};
\draw (3.25,-.3) node {$S_3$};
\draw (4.5,-.3) node {$J_4$};
\draw (6,-.3) node {$S_2$};
\draw (7.5,-.3) node {$J_3$};
\draw (8.75,-.3) node {$S_1$};
\draw (10,-.3) node {$K_1$};
\draw (11.25,-.3) node {$S_1$};
\draw (12,.2) node {$b$};
\draw (12.75,-.3) node {$J_1$};

\draw [dotted] (-1,-1.5) --++(14,0);
\draw (0,-1.5) --++(4,0);
\draw (5,-1.5) --++(2,0);
\draw (8,-1.5) --++(4,0);

\draw (2,-1.8) node {$A_2$};
\draw (6,-1.8) node {$B_1$};
\draw (10,-1.8) node {$A_1$};

\draw [dotted] (-1,-3) --++(14,0);
\draw (0,-3) --++(4,0);
\draw (5,-3) --++(2,0);
\draw (8,-3) --++(1.5,0);
\draw (10.5,-3) --++(1.5,0);

\draw (2,-3.3) node {$A_3$};
\draw (6,-3.3) node {$B_1$};
\draw (8.75,-3.3) node {$A_2$};
\draw (11.25,-3.3) node {$A_1$};

\draw [dotted] (-1,-4.5) --++(14,0);
\draw (0,-4.5) --++(1.5,0);
\draw (2.5,-4.5) --++(1.5,0);
\draw (5,-4.5) --++(2,0);
\draw (8,-4.5) --++(1.5,0);
\draw (10.5,-4.5) --++(1.5,0);

\draw (.75,-4.8) node {$A_4$};
\draw (3.25,-4.8) node {$A_3$};
\draw (6,-4.8) node {$B_1$};
\draw (8.75,-4.8) node {$A_2$};
\draw (11.25,-4.8) node {$A_1$};

\end{tikzpicture}
\caption{A test case of $S = S_1 \cup S_2 \cup S_3$, with the intervals
given in equation (\ref{eqf'domain2}), and various
definitions of the sets $A := \cup_i A_i$ and $B := \cup_j B_j$.
Solid intervals are closed, and dotted intervals are open.}
\label{figSupports}
\end{figure}
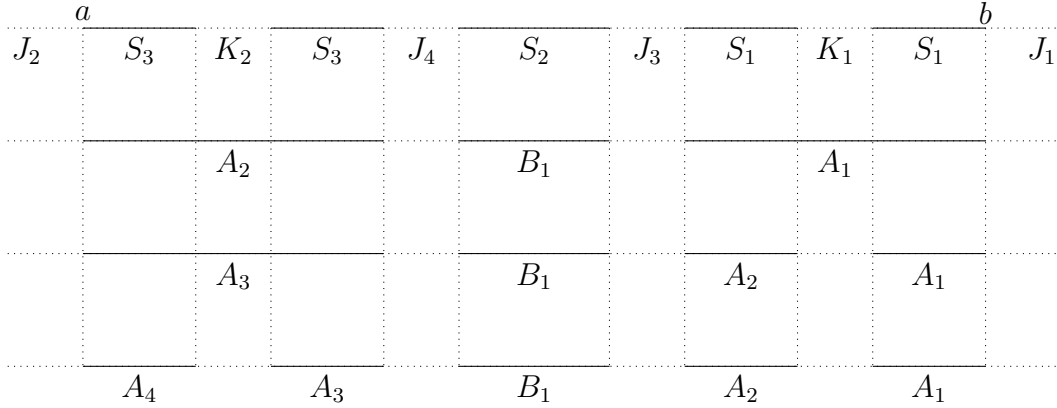

Recall definitions \ref{defLiq} and \ref{defEdge} of $\LL$ and
$\EE = \EE_\mu \cup \EE_{\l-\mu} \cup \EE_0 \cup \EE_1 \cup \EE_2$:
\begin{itemize}
\item
$\LL$ is the set of all $(\chi,\eta) \in [a,b] \times [0,1]$
for which $b \ge \chi \ge \chi+\eta-1 \ge a$ and $f'$ has non-real roots.
\item
$\EE_\mu$ is the set of all $(\chi,\eta)$
for which $f'$ has a repeated root in $\R \setminus [\chi+\eta-1,\chi]$.
\item
$\EE_{\l-\mu}$ is the set of all $(\chi,\eta)$ for which $f'$ has a
repeated root in $(\chi+\eta-1,\chi)$.
\item
$\EE_0$ is the set of all $(\chi,\eta)$ for which $\eta=1$ and $f'$ has a root
at $\chi \; (=\chi+\eta-1)$.
\item
$\EE_1$ is the set of all $(\chi,\eta)$ for which $\eta<1$ and $f'$ has a root at $\chi$.
\item
$\EE_2$ is the set of all $(\chi,\eta)$ for which $\eta<1$ and $f'$ has a root
at $\chi+\eta-1$.
\end{itemize}
We end this section by using theorem \ref{thmf'} to show the following:

\begin{cor}
\label{corf'}
Using the notation of theorem \ref{thmf'},
\begin{enumerate}
\item[(a)]
$S_1$, $S_2$ and $S_3$ are all non-empty whenever
$(\chi,\eta) \in \LL \cup \EE_\mu \cup \EE_{\l-\mu}$.
\item[(b)]
$\eta>0$ whenever $(\chi,\eta) \in \LL \cup \EE$.
\item[(c)]
$\{\LL, \EE_\mu, \EE_{\l-\mu}, \EE_0, \EE_1, \EE_2 \}$
is pairwise disjoint.
\end{enumerate}
\end{cor}

\begin{proof}
Consider (a). First suppose that $(\chi,\eta) \in \LL$. Then $f'$ has
non-real roots, by definition. Therefore $f'$ has at least $2$ roots in
$\C \setminus \R$, counting multiplicities, since non-real roots occur
in complex conjugate pairs. This is only possible in properties (1-5)
of theorem \ref{thmf'}, and so $S_1,S_2,S_3$ are non-empty. Next suppose
that $(\chi,\eta) \in \EE_\mu \cup \EE_{\l-\mu}$. Then $f'$ has a
real-valued repeated root, by definition. Again, this is only possible
in properties (1-5), and so $S_1,S_2,S_3$ are non-empty.

Consider (b). Fix $\eta=0$. Then either properties (6-8) or (11-12) of
theorem \ref{thmf'} are satisfied. First suppose that $(\chi,\eta) \in \LL$.
Recall that $f'$ has at least $2$ roots in $\C \setminus \R$, counting
multiplicities. This leads to a contradiction,
since it violates either (6) or (11). Next suppose that
$(\chi,\eta) \in \EE_\mu \cup \EE_{\l-\mu}$. Then $f'$ has a repeated
root, which violates either (6-8) or (11-12). Next suppose that
$(\chi,\eta) \in \EE_0$. This gives a trivial contradiction, since
$\eta=1$ for all $(\chi,\eta) \in \EE_0$.

Finally, suppose that $(\chi,\eta) \in \EE_1 \cup \EE_2$. Then $f'$ has a
root in $\{\chi, \chi+\eta-1\}$. Recall that properties (6-8) are
satisfied whenever $S_1,S_2,S_3$ are non-empty. In this case
$b > \chi > \chi+\eta-1 > a$, $J_1 = (b,+\infty)$, $J_2 = (-\infty,a)$,
$\chi \in J_3$ whenever $\chi \in \C \setminus S$, and $\chi+\eta-1 \in J_4$
whenever $\chi+\eta-1 \in \C \setminus S$. Thus $f'$ has a root in
$\{\chi, \chi+\eta-1\} \subset J_3 \cup J_4$, which violates (7).
Also recall that properties (11-12) are satisfied when $S_2,S_3$
are non-empty and $S_1$ is empty, or alternatively, when $S_1,S_2$
are non-empty and $S_3$ is empty. In the first case $b = \chi > \chi+\eta-1 > a$,
$J_1 = (\sup S_2,+\infty)$, $J_2 = (-\infty,a)$, $\chi = b \in J_1$,
$J_3 = \emptyset$, and $\chi+\eta-1 \in J_4$ whenever
$\chi+\eta-1 \in \C \setminus S$. In the second case
$b > \chi > \chi+\eta-1 = a$, $J_1 = (b,+\infty)$, $J_2 = (-\infty,\inf S_2)$,
$\chi+\eta-1 = a \in J_2$, $\chi \in J_3$ whenever $\chi \in \C \setminus S$,
and $J_4 = \emptyset$. In either case $f'$ has a root in
$\{\chi,\chi+\eta-1\} \subset J$, which violates (11). (b) then
follows by contradiction.

Consider (c). Suppose first that $(\chi,\eta) \in \LL$. Recall that $f'$
has at least $2$ roots in $\C \setminus \R$, counting multiplicities, and
this is only possible when $\eta>0$ and $S_1,S_2,S_3$ are
non-empty, i.e., properties (1-5) of theorem \ref{thmf'}. Thus
$b > \chi > \chi+\eta-1 > a$, $J_1 = (b,+\infty)$, $J_2 = (-\infty,a)$,
$\chi \in J_3$ whenever $\chi \in \C \setminus S$, and $\chi+\eta-1 \in J_4$ whenever
$\chi+\eta-1 \in \C \setminus S$. Thus $\eta < 1$, and so $(\chi,\eta) \not\in \EE_0$.
Also, (1) and (5) show that $f'$ has exactly $2$ roots in $\C \setminus \R$,
counting multiplicities, $0$ roots in $J$, and at most $1$ root in each of
$\{K_1,K_2,\ldots\}$. Thus $f'$ has no real-valued repeated roots, and no roots in
$\{\chi,\chi+\eta-1\} \subset J$, and so
$(\chi,\eta) \not\in \EE_\mu \cup \EE_{\l-\mu} \cup \EE_1 \cup \EE_2$.

Next suppose that $(\chi,\eta) \in \EE_\mu$, i.e., $f'$ has a repeated root in
$\R \setminus [\chi+\eta-1,\chi]$. There are $2$ possibilities (see equation
(\ref{eqf'domain2})):
\begin{itemize}
\item
The repeated root is contained in $J \setminus [\chi+\eta-1,\chi]$.
\item
The repeated root is contained in $K \setminus [\chi+\eta-1,\chi]$.
\end{itemize}
Either case is possible only when $\eta>0$ and $S_1,S_2,S_3$ are non-empty, i.e.,
properties (1-5) of theorem \ref{thmf'}. Thus $b > \chi > \chi+\eta-1 > a$,
$J_1 = (b,+\infty)$, $J_2 = (-\infty,a)$,
$\chi \in J_3$ whenever $\chi \in \C \setminus S$, and $\chi+\eta-1 \in J_4$ whenever
$\chi+\eta-1 \in \C \setminus S$. Thus $\eta < 1$, and so $(\chi,\eta) \not\in \EE_0$.
Also, (1), (4) and (5) show that $f'$ has at most one real-valued repeated root, and so
$(\chi,\eta) \not\in \EE_{\l-\mu}$. Moreover, (5) shows that $f'$ has $0$ roots in
$J$ whenever the repeated root is contained in $K \setminus [\chi+\eta-1,\chi]$, and
(1) shows that $f'$ has no other roots in $J$ whenever the repeated root
is contained in $J \setminus [\chi+\eta-1,\chi]$. Thus $f'$ has $0$ roots in
$\{\chi,\chi+\eta-1\} \subset J$, and so $(\chi,\eta) \not\in \EE_1 \cup \EE_2$.

Next suppose that $(\chi,\eta) \in \EE_{\l-\mu}$, i.e., $f'$ has a repeated root in
$(\chi+\eta-1,\chi)$. Similar arguments to those used above (in the case $(\chi,\eta) \in \EE_\mu$)
then give $(\chi,\eta) \not\in \EE_0 \cup \EE_1 \cup \EE_2$. Next suppose that
$(\chi,\eta) \in \EE_0$. Then $\eta=1$, and so $(\chi,\eta) \not\in \EE_1 \cup \EE_2$.

Finally suppose that $(\chi,\eta) \in \EE_1 \cap \EE_2$, i.e., $\eta < 1$ and $f'$ has
a root at $\chi$ and at $\chi+\eta-1$. We show that this contradicts theorem \ref{thmf'}.
First recall, whenever $S_1,S_2,S_3$ are non-empty, that $b > \chi > \chi+\eta-1 > a$,
$J_1 = (b,+\infty)$, $J_2 = (-\infty,a)$, $\chi \in J_3$ and $\chi+\eta-1 \in J_4$.
Thus $f'$ has a root in $J_3$ and a root in $J_4$, which violates either
(2) (when $\eta>0$) or (7) (when $\eta=0$). Also recall, whenever $S_2,S_3$
are non-empty and $S_1$ is empty, that $b = \chi > \chi+\eta-1 > a$,
$J_1 = (\sup S_2,+\infty)$, $J_2 = (-\infty,a)$, $\chi = b \in J_1$,
$J_3 = \emptyset$, and $\chi+\eta-1 \in J_4$. Thus $f'$ has a root in $J_1$ and a
root in $J_4$, which violates either (9) (when $\eta > 0$) or
(11) (when $\eta=0$). Similarly whenever $S_1,S_2$ are non-empty and $S_3$ is empty.
Finally recall, whenever $\eta>0$ and $S_2$ is empty, that $J_3 = J_4 = \emptyset$,
and that $\chi$ and $\chi+\eta-1$ are both contained in the same element of
$\{J_1,J_2,K_1,K_2,\ldots\}$. Thus $J = J_1 \cup J_2$, and $f'$ has at least $2$ roots
in an element of $\{J_1,J_2,K_1,K_2,\ldots\}$. This violates (11)
and (12). Thus we have a contradiction, and so $\EE_1 \cap \EE_2 = \emptyset$.
\end{proof}

\section{Appendix}
\label{secApp}

\subsection{Derivation of the correlation kernel}
\label{secdofckfit}

In section \ref{sectdsodgtp} we construct a determinantal random point
process on configurations of particles in $\Z \times \{1,\ldots,n\}$.
We now use the Eynard-Mehta theorem (see, for example, proposition
2.13 of Johansson, \cite{Joh06b}) to show that this process has the
correlation kernel given in equation (\ref{eqKnrusvFixTopLine}).
This is a generalisation of Defosseux, \cite{Def08}, and Metcalfe,
\cite{Met13}, which consider a similar process on configurations in
$\R \times \{1,\ldots,n\}$. The kernel in \cite{Def08} and \cite{Met13}
is recovered from the kernel in equation (\ref{eqKnrusvFixTopLine})
using asymptotic arguments. The kernel in equation (\ref{eqKnrusvFixTopLine})
was also independently obtained by Petrov, \cite{Pet14}. Our proof,
based on the methods used in \cite{Def08} and \cite{Met13}, is more
elementary than that of Petrov. In particular, we note that Petrov also
uses the Eynard-Mehta theorem. However, he first considers the more
complicated $q^{-\text{vol}}$ measure, which is unnecessary. Moreover, we note
that the main technical difficulty of the Eynard-Mehta method is to deal
with a certain matrix inverse (see equation (\ref{eqprCnCnDet1})). We
use the finite difference operator to rewrite the expression as a ratio
of determinants (see equation (\ref{eqprCnCnDet2})), whereas Petrov
constructs a diagonal matrix.

We begin by briefly recalling the model. See section
\ref{sectdsodgtp} for more details. Consider all $n$-tuples,
$(y^{(1)},y^{(2)},\ldots,y^{(n)}) \in  \Z \times \Z^2 \times \cdots \times \Z^n$,
which satisfy
\begin{equation*}
y_1^{(r+1)} \; \ge \; y_1^{(r)} \; > \; y_2^{(r+1)} \; \ge \; y_2^{(r)}
\; > \cdots \ge \; y_r^{(r)} \; > \; y_{r+1}^{(r+1)},
\end{equation*}
for all $r$, denoted $y^{(r+1)} \succ y^{(r)}$.
Fix $n\ge1$ and $x \in \Z^n$ with $x_1 > x_2 > \cdots > x_n$,
and define the following probability measure on the set of all such $n$-tuples:
\begin{equation*}
\nu[(y^{(1)},\ldots,y^{(n)})]
:= \frac1Z \cdot \left\{
\begin{array}{rcl}
1 & ; &
\text{when} \; x = y^{(n)} \succ y^{(n-1)} \succ \cdots \succ y^{(1)}, \\
0 & ; & \text{otherwise},
\end{array}
\right.
\end{equation*}
where $Z > 0$ is a normalisation constant.

We now construct a related probability space, the determinantal structure of which
is more convenient to examine. Consider all tuples,
$(z^{(1)},\ldots,z^{(n-1)}) \in  \Z^n \times \cdots \times \Z^n$ with 
\begin{equation*}
z_1^{(r+1)} \; \ge \; z_1^{(r)} \; > \; z_2^{(r+1)} \; \ge \; z_2^{(r)}
\; > \cdots > \; z_n^{(r+1)} \; \ge \; z_n^{(r)},
\end{equation*}
for all $r$, also denoted $z^{(r+1)} \succ z^{(r)}$. Fix
$z^{(0)} := (x_n+n-1,\ldots,x_n+1,x_n)$ and define the following probability
measure on the set of all such $(n-1)$-tuples:
\begin{equation}
\label{eqnu'}
\nu'[(z^{(1)},\ldots,z^{(n-1)})]
:= \frac1{Z'} \cdot \left\{
\begin{array}{rcl}
1 & ; &
\text{when} \; x \succ z^{(n-1)} \succ \cdots \succ z^{(1)} \succ z^{(0)}, \\
0 & ; & \text{otherwise},
\end{array}
\right.
\end{equation}
where $Z' > 0$ is a normalisation constant.

Consider the relationship between the spaces. First note that,
whenever $x = y^{(n)} \succ y^{(n-1)} \succ \cdots \succ y^{(1)}$ for some
$(y^{(1)},y^{(2)},\ldots,y^{(n)}) \in  \Z \times \Z^2 \times \cdots \times \Z^n$,
then
\begin{equation*}
x_1 \ge y_1^{(r)} > \cdots > y_r^{(r)} > x_n+n-r-1,
\end{equation*}
for all $r \le n$. Whenever $x \succ z^{(n-1)} \succ \cdots \succ z^{(1)} \succ z^{(0)}$
for some $(z^{(1)},\ldots,z^{(n-1)}) \in \Z^n \times \Z^n \times \cdots \times \Z^n$,
\begin{alignat}{8}
\label{eqfreepart}
x_1 \ge z_1^{(r)} > \cdots > z_r^{(r)} > & \; z_{r+1}^{(r)}
\; & \; > \; & \; z_{r+2}^{(r)} \; & \; > \; & \cdots & \; > \; & \; z_n^{(r)} \\
\nonumber
& {}_{x_n+n-r-1}^{\;\;\;\;\;\rotatebox{90}{\,=}}
& & {}_{x_n+n-r-2}^{\;\;\;\;\;\rotatebox{90}{\,=}}
& & \ldots & & {}_{\;\;x_n}^{\;\;\;\rotatebox{90}{\,=}}
\end{alignat}
for all $r \le n-1$. We refer to $z_1^{(r)},\ldots,z_r^{(r)}$ as the
{\em free particles} of $z^{(r)}$, and $z_{r+1}^{(r)},\ldots,z_n^{(r)}$ as the
{\em deterministic particles}. Note the natural bijection between
$\{(y^{(1)},\ldots,y^{(n)}) \in \Z \times \Z^2 \times \cdots \times \Z^n
: x = y^{(n)} \succ y^{(n-1)} \succ \cdots \succ y^{(1)}\}$ and
$\{(z^{(1)},\ldots,z^{(n-1)}) \in \Z^n \times \Z^n \times \cdots \times \Z^n
: x \succ z^{(n-1)} \succ \cdots \succ z^{(1)} \succ z^{(0)} \}$:
Remove $y^{(n)}$ from each $n$-tuple $(y^{(1)},\ldots,y^{(n)})$ and
map the remaining components, $y^{(r)} = (y_1^{(r)},\ldots,y_r^{(r)})$ for each
$r \le n-1$, individually as,
\begin{equation*}
y^{(r)} \mapsto (y_1^{(r)},\ldots,y_r^{(r)},x_n+n-r-1,x_n+n-r-2,\ldots,x_n).
\end{equation*}
The measure $\nu'$ is induced by the measure $\nu$ under this bijective map.
The probabilistic structure of particles in the first space (measure $\nu$)
is therefore identical to the probabilistic structure of the free particles in
the second space (measure $\nu'$). From now on we restrict to the second space.

A more convenient expression for $\nu'$ can be obtained from the work of
Warren, \cite{W06}:
\begin{equation*}
\det \left[ 1_{z_j^{(r+1)} \ge z_i^{(r)}} \right]_{i,j=1}^n
= \left\{ \begin{array}{ll} 
1 & ; \text{ when} \; z^{(r+1)} \succ z^{(r)}, \\
0 & ; \text{ otherwise,}
\end{array} \right.
\end{equation*}
for all $r$.
Equation (\ref{eqnu'}) thus gives,
\begin{equation}
\label{eqnu'2}
\nu'[(z^{(1)},\ldots,z^{(n-1)})]
= \frac1{Z'} \prod_{r=0}^{n-1} \det \left[ \phi_{r,r+1}(z_i^{(r)},z_j^{(r+1)}) \right]_{i,j=1}^n,
\end{equation}
where $z^{(n)} := x$, and
\begin{equation*}
\phi_{r,r+1} (u,v) := 1_{v \ge u},
\end{equation*}
for all $r$ and $u,v \in \Z$.

Note, each $(z^{(1)},\ldots,z^{(n-1)}) \in \Z^n \times \Z^n \times \cdots \times \Z^n$
can be equivalently considered as a configuration of
particles in $\Z \times \{1,\ldots,n-1\}$ by placing a particle at position
$(u,r) \in \Z \times \{1,\ldots,n-1\}$ whenever $u$ is an element of $z^{(r)}$.
The measure $\nu'$ in equation (\ref{eqnu'2}) therefore defines a random point
process on configurations of  particles in $\Z \times \{1,\ldots,n-1\}$.
Proposition 2.13 of Johansson, \cite{Joh06b}, proves
that this process is determinantal with correlation kernel,
\begin{equation}
\label{eqprCnCnDet1}
K_n ((u,r),(v,s)) =  \tilde K_n ((u,r),(v,s)) - \phi_{r,s} (u,v),
\end{equation}
for all $r,s \in \{1,\ldots,n-1\}$ and $u,v \in \Z$, where
\begin{equation*}
\tilde K_n ((u,r),(v,s))
:= \sum_{k,l=1}^n \phi_{r,n} (u,z_k^{(n)}) (A^{-1})_{kl} \phi_{0,s} (z_l^{(0)},v),
\end{equation*}
and
\begin{align*}
\phi_{r,s} (u,v)
&:= 0 \text{ when } s \le r, \\
\phi_{r,s} (u,v)
&:= 1_{v \ge u} \text{ when } s=r+1, \\
\phi_{r,s} (u,v)
&:= \sum_{z_1,\ldots,z_{s-r-1}} \phi_{r,r+1}(u,z_1)
\phi_{r+1,r+2}(z_1,z_2) \cdots \phi_{s-1,s}(z_{s-r-1},v) \text{ when } s>r+1, \\
A \in \C^{n \times n} &\text{ with }
A_{kl} := \phi_{0,n}(z_k^{(0)},z_l^{(n)}) \text{ for all } k,l.
\end{align*}

Note that, for all $r,s \in \{1,\ldots,n\}$ and $u,v \in \Z$,
\begin{align}
\nonumber
\phi_{r,s} (u,v)
&= \left\{
\begin{array}{ll} 
0 & ; s \le r, \\
1_{v \ge u} & ; s = r+1, \\
\sum_{z_1,\ldots,z_{s-r-1}} 1_{v \ge z_{s-r-1} \ge \cdots \ge z_2 \ge z_1 \ge u}
& ; s > r+1,
\end{array} \right. \\
\label{eqphirsuvCHSP}
&= \left\{
\begin{array}{ll} 
0 & ; s \le r, \\
1_{v \ge u} & ; s = r+1, \\
1_{v \ge u} h_{s-r-1}((1)^{v-u+1}) & ; s > r+1,
\end{array} \right.
\end{align}
where $h_r((1)^{v})$ denotes the complete homogeneous symmetric
polynomial of degree $r$ in $v$ variables, evaluated at the $v$-tuple
$(1)^{v} := (1,1,\ldots,1)$. Evaluating gives,
\begin{equation}
\label{eqphirsuv0}
\phi_{r,s}(u,v)
= \left\{
\begin{array}{ll} 
0 & ; s \le r, \\
1_{v-u \ge 0} & ; s = r+1, \\
1_{v-u \ge 0} \frac1{(s-r-1)!} \prod_{j=1}^{s-r-1}(v-u+s-r-j)
& ; s > r+1.
\end{array} \right.
\end{equation}
Thus, noting that $\phi_{r,s}(u,v) = 0$ when $s > r+1$ and
$v-u \in \{-1,-2,\ldots,-(s-r-1)\}$,
\begin{equation}
\label{eqphirsuv1}
\phi_{r,s}(u,v)
= 1_{v-u+s-r-1 \ge 0} \left\{
\begin{array}{ll} 
0 & ; s \le r, \\
1 & ; s = r+1, \\
\frac1{(s-r-1)!} \prod_{j=1}^{s-r-1}(v-u+s-r-j)
& ; s > r+1.
\end{array} \right.
\end{equation}
Finally, letting $\Delta_v^{n-s}$ denote $n-s$ iterations of the finite difference
operator in the variable $v$ (i.e., $\Delta_v f(v) := f(v+1) - f(v)$
for any function $f : \Z \to \R$), induction gives
\begin{equation}
\label{eqphirsuvFDO}
\phi_{r,s}(u,v) =
1_{v-u+s-r-1 \ge 0} \; \Delta_v^{n-s} \frac1{(n-r-1)!} \prod_{j=1}^{n-r-1}(v-u+s-r-j),
\end{equation}
for all $r,s \in \{1,\ldots,n\}$ and $u,v \in \Z$.

Recall that a particle is at position $(v,s)$ whenever $v$ is an element of $z^{(s)}$.
Thus, since we are only interested in the determinantal structure of the free particles
(see equation (\ref{eqfreepart}) and the following comments), it is natural to
restrict to $v \ge x_n+n-s$. For all $r,s \in \{1,\ldots,n-1\}$, $u \in \Z$ and
$v \ge x_n+n-s$, recalling that $z_l^{(0)} = x_n+n-l$, 
\begin{equation*}
\phi_{0,s}(z_l^{(0)},v)
= \Delta_v^{n-s} \frac1{(n-1)!} \prod_{j=1}^{n-1} (v-z_l^{(0)}+s-j)
= \Delta_v^{n-s} \phi_{0,n} (z_l^{(0)}, v+s-n),
\end{equation*}
where the first part follows from equation (\ref{eqphirsuvFDO}),
and the second part from equation (\ref{eqphirsuv1}). Equation
(\ref{eqprCnCnDet1}) thus gives,
\begin{equation*} 
\tilde K_n ((u,r),(v,s)) = \Delta_v^{n-s} \sum_{k,l=1}^n
\phi_{r,n} (u,z_k^{(n)}) (A^{-1})_{kl} \phi_{0,n} (z_l^{(0)},v+s-n).
\end{equation*}
Recall that $z_k^{(n)} = x_k$ and apply Cramer's rule to get,
\begin{equation}
\label{eqprCnCnDet2}
\tilde K_n ((u,r),(v,s)) = \Delta_v^{n-s} \sum_{k=1}^n
\phi_{r,n} (u,x_k) \frac{\det[A(k,v)]}{\det[A]},
\end{equation}
where $A(k,v) \in  \C^{n\times n}$ is the matrix $A$ with column $k$ replaced by
\begin{equation*} 
(\phi_{0,n}(z_1^{(0)},v+s-n), \phi_{0,n}(z_2^{(0)},v+s-n), \ldots, \phi_{0,n}(z_n^{(0)},v+s-n))^T.
\end{equation*}

Note, equation (\ref{eqphirsuvCHSP}) can alternatively be written as,
\begin{equation*}
\phi_{r,s} (u,v) = 1_{s>r} h_{v-u} ((1)^{s-r}),
\end{equation*}
for all $r,s \in \{1,\ldots,n\}$ and $u,v \in \Z$ (by convention $h_k := 0$
whenever $k < 0$). Thus, recalling that $z_i^{(0)} = x_n+n-i$ and $z_j^{(n)} = x_j$,
equation (\ref{eqprCnCnDet1}) gives
\begin{equation*}
\det [A] = \det \left[ \phi_{0,n} (z_i^{(0)}, z_j^{(n)}) \right]_{i,j=1}^n
= \det \left[ h_{x_j-x_n-n+i} ((1)^{n}) \right]_{i,j=1}^n,
\end{equation*}
Similarly, for all $k$ and $v$,
\begin{equation*}
\det [A(k,v)]  = \det \left[ h_{x(k,v)_j -x_n-n+i} ((1)^{n}) \right]_{i,j=1}^n,
\end{equation*}
where $x(k,v)_j := x_j$ for all $j \ne k$ and $x(k,v)_k := v+s-n$. Therefore,
\begin{equation*}
\frac{\det [A(k,v)]}{\det[A]}
= \lim_{q \uparrow 1} \frac{\det [ h_{x(k,v)_j-x_n-n+i} (q^{n-1},\ldots,q^1,q^0) ]_{i,j} }
{\det[h_{x_j-x_n-n+i} (q^{n-1},\ldots,q^1,q^0) ]_{i,j} }.
\end{equation*}
The above is written as a limit for computational convenience.
Recall that $x_1 > x_2 > \cdots > x_n$ and that $v \ge x_n+n-s$.
Therefore $x_m - x_n \ge 0$ and $x(k,v)_m - x_n \ge 0$ for all $m$.
Expression (3.7) of MacDonald, \cite{Mac95}, thus gives,
\begin{equation*}
\frac{\det [A(k,v)]}{\det[A]}
= \lim_{q \uparrow 1} \frac{\det [ q^{(x(k,v)_j-x_n)(n-i)}]_{i,j}}
{\det[q^{(x_j-x_n)(n-i)}]_{i,j}}.
\end{equation*}
The numerator and denominator are both Vandermonde determinants, and so
\begin{align*}
\frac{\det [A(k,v)]}{\det[A]}
&= \lim_{q \uparrow 1} \prod_{1 \le i < j \le n}
\left( \frac{q^{x(k,v)_j-x_n} - q^{x(k,v)_i-x_n}}{q^{x_j-x_n} - q^{x_i-x_n}} \right) \\
&= \lim_{q \uparrow 1} \prod_{i=1, i \ne k}^n
\left( \frac{q^{v+s-n-x_n} - q^{x_i-x_n}}{q^{x_k-x_n} - q^{x_i-x_n}} \right)
= \prod_{i=1, i \ne k}^n \left( \frac{v+s-n-x_i}{x_k-x_i} \right).
\end{align*}
Combining with equations (\ref{eqphirsuv0}) and
(\ref{eqprCnCnDet2}) gives,
\begin{align*} 
\lefteqn{\tilde K_n ((u,r),(v,s))} \\ 
&= \Delta_v^{n-s} \sum_{k=1}^n \frac{1_{x_k-u \ge 0}}{(n-r-1)!}
\prod_{j=1}^{n-r-1}(x_k-u+n-r-j) \prod_{i \ne k} \left( \frac{v+s-n-x_i}{x_k-x_i} \right) \\
&= \Delta_v^{n-s} \sum_{k=1}^n 1_{x_k \ge u} \frac1{(n-r-1)!}
\prod_{j=u+r-n+1}^{u-1} (x_k-j) \prod_{i \ne k} \left(\frac{v+s-n-x_i}{x_k-x_i} \right),
\end{align*}
for all $r,s \in \{1,\ldots,n-1\}$, $u \in \Z$ and $v \ge x_n+n-s$.
Finally, for any function $f : \Z \to \R$,
\begin{equation}
\label{eqfindiff}
(\Delta_v^{n-s} f)(v+s-n) = (n-s)! \sum_{l=v-n+s}^v
\frac{f(l)}{\prod_{j=v-n+s,j\ne l}^v (l-j)},
\end{equation}
and so
\begin{equation*} 
\tilde K_n ((u,r),(v,s)) = \frac{(n-s)!}{(n-r-1)!} \sum_{k=1}^n \sum_{l=v-n+s}^v
1_{x_k \ge u} \frac{\prod_{j=u+r-n+1}^{u-1}(x_k-j)}{\prod_{j=v-n+s,j\ne l}^v (l-j)}
\prod_{i \ne k} \left( \frac{l-x_i}{x_k-x_i} \right).
\end{equation*}
This, combined with equations (\ref{eqprCnCnDet1}) and (\ref{eqphirsuv0}) gives
the correlation kernel in equation (\ref{eqKnrusvFixTopLine}), as required.
Residue theory gives a natural contour integral representation of this kernel,
which is amenable to steepest descent analysis, as shown in equations
(\ref{eqKnrnunsnvn1}) and (\ref{eqJnrnunsnvn1}).

We finish this section by finding an alternate useful expression for $\phi_{r,s} (u,v)$,
for all $r,s \in \{1,\ldots,n-1\}$ and $u,v \in \Z$. First note that we can replace
$1_{v-u+s-r-1 \ge 0}$ in equation (\ref{eqphirsuvFDO}) by $1_{v \ge u}$. Next note that
the product in this equation can be interpreted as a polynomial of degree $n-r-1$ in the
variable $v+s-n$. Applying Lagrange interpolation to the polynomial using the $n$ points
$x_1 > x_2 > \cdots > x_n$ gives,
\begin{align*} 
\phi_{r,s}(u,v)
&= 1_{v \ge u} \; \Delta_v^{n-s} \frac1{(n-r-1)!}
\sum_{k=1}^n \prod_{j=1}^{n-r-1} (x_k-u+n-r-j) \prod_{i\neq k}
\left( \frac{v+s-n-x_i}{x_k-x_i} \right) \\
&= 1_{v \ge u} \; \Delta_v^{n-s} \frac1{(n-r-1)!}
\sum_{k=1}^n \prod_{j=u+r-n+1}^{u-1} (x_k-j) \prod_{i\neq k}
\left( \frac{v+s-n-x_i}{x_k-x_i} \right).
\end{align*}
Equation (\ref{eqfindiff}) finally gives,
\begin{equation*}
\phi_{r,s}(u,v) = 1_{v \ge u} \frac{(n-s)!}{(n-r-1)!} \sum_{k=1}^n \sum_{l=v-n+s}^v
\frac{\prod_{j=u+r-n+1}^{u-1}(x_k-j)}{\prod_{j=v-n+s,j\ne l}^v (l-j)}
\prod_{i \ne k} \left( \frac{l-x_i}{x_k-x_i} \right).
\end{equation*}

\subsection{Proof that \texorpdfstring{$W_\LL : \LL \to \mathbb{H}$}{Lg} is
a diffeomorphism}
\label{secDiff}

In theorem \ref{thmwc} we defined the function $W_\LL : \LL \to \mathbb{H}$
and proved that it is a homeomorphism with inverse
$(\chi_\LL(\cdot),\eta_\LL(\cdot)) : \mathbb{H} \to \LL$. In this section
we show the stronger result that $W_\LL$ is a diffeomorphism. More exactly,
we show that $W_\LL \in C^\infty(\LL,\mathbb{H})$ and
$(\chi_\LL(\cdot),\eta_\LL(\cdot)) \in C^\infty(\mathbb{H},\LL)$.
Though we do not need it for this paper, we include the proof out of interest.

We begin by briefly recalling the relevant definitions: $\LL$
is the set of all $(\chi,\eta) \in [a,b] \times [0,1]$ for which
$b \ge \chi \ge \chi + \eta - 1 \ge a$, and for which
\begin{equation}
\label{eqdiff1}
f_{(\chi,\eta)}'(w) = \int_a^b \frac{\mu[dx]}{w-x} + \log(w - \chi) - \log(w-\chi-\eta+1),
\end{equation}
has non-real roots (here $\log$ is principal value). Theorem \ref{thmf'} implies
that there is a unique root in $\mathbb{H}$ whenever $(\chi,\eta) \in \LL$,
$W_\LL : \LL \to \mathbb{H}$ maps to this root, and
$(\chi_\LL(\cdot),\eta_\LL(\cdot)) : \mathbb{H} \to \LL$ is the inverse of $W_\LL$.

Define $U := \{ (\chi,\eta,u,v) \in \R^4 : (\chi,\eta) \in \LL, u \in \R, v>0\}$,
and $\Phi : U \to \R^2$ by,
\begin{equation}
\label{eqdiff2}
\Phi(\chi,\eta,u,v)
= (\phi_1(\chi,\eta,u,v), \phi_2(\chi,\eta,u,v))
:= \left( \text{Re} \; f_{(\chi,\eta)}'(u+iv), \text{Im} \; f_{(\chi,\eta)}'(u+iv) \right).
\end{equation}
Note that $\Phi \in C^\infty (U,\R^2)$. Note also that,
\begin{align}
\label{eqdiff3}
\Phi(\chi,\eta,u,v) = 0
&\Leftrightarrow f_{(\chi,\eta)}'(u + i v) = 0 \\
\nonumber
&\Leftrightarrow W_\LL(\chi,\eta) = u + i v \\
\nonumber
&\Leftrightarrow (\chi_\LL(u + i v),\eta_\LL(u + i v)) = (\chi,\eta).
\end{align}

The Jacobian of $\Phi$ with respect to the variables $u$ and $v$ is given by,
\begin{equation*}
\frac{\partial (\phi_1,\phi_2)}{\partial (u,v)}
= \frac{\partial \phi_1}{\partial u} \frac{\partial \phi_2}{\partial v}
- \frac{\partial \phi_1}{\partial v} \frac{\partial \phi_2}{\partial u}
= \left( \frac{\partial \phi_1}{\partial u} \right)^2
+ \left( \frac{\partial \phi_2}{\partial u} \right)^2
= |f_{(\chi,\eta)}''(u+iv)|^2,
\end{equation*}
for all $(\chi,\eta,u,v) \in U$ (the second step above follows from the
Cauchy-Riemann's equations, since $f_{(\chi,\eta)}' : \C \setminus \R \to \C$
is analytic). Therefore, whenever $\Phi(\chi_0,\eta_0,u_0,v_0) = 0$ for some
fixed $(\chi_0,\eta_0,u_0,v_0) \in U$, equation (\ref{eqdiff3}) gives
$f_{(\chi_0,\eta_0)}'(u_0 + i v_0) = 0$. Finally, theorem \ref{thmf'} implies that
$u_0 + i v_0$ is a root of $f_{(\chi_0,\eta_0)}'$ of multiplicity $1$, and so
\begin{equation*}
\frac{\partial (\phi_1,\phi_2)}{\partial (u,v)} (\chi_0,\eta_0,u_0,v_0) \neq 0.
\end{equation*}
The implicit function theorem (theorem 3.3.1 of \cite{Krantz02}) thus implies
that there exists functions, $f_1$ and $f_2$, defined on a neighbourhood of
$(\chi_0,\eta_0)$, which are contained in $C^\infty$ (recall, $\Phi \in C^\infty (U,\R^2)$)
and which satisfy $\Phi(\chi,\eta,f_1(\chi,\eta),f_2(\chi,\eta)) = 0$ for all
$(\chi,\eta)$ in this neighbourhood. Equation (\ref{eqdiff3}) thus gives
$W_\LL(\chi,\eta) = f_1(\chi,\eta) + i f_2(\chi,\eta)$, for all such $(\chi,\eta)$,
and so $W_\LL \in C^\infty(\LL,\mathbb{H})$.

Also, the Jacobian of $\Phi$ with respect to the variables $\chi$ and $\eta$
is given by (see equations (\ref{eqdiff1}) and (\ref{eqdiff2})),
\begin{equation*}
\frac{\partial (\phi_1,\phi_2)}{\partial (\chi,\eta)}
= \frac{\partial \phi_1}{\partial \chi} \frac{\partial \phi_2}{\partial \eta}
- \frac{\partial \phi_1}{\partial \eta} \frac{\partial \phi_2}{\partial \chi}
= - \frac{(1-\eta) v}{((u-\chi)^2 + v^2) ((u-\chi-\eta+1)^2 + v^2)} < 0.
\end{equation*}
Thus, fixing $(\chi_0,\eta_0,u_0,v_0) \in U$, the implicit function theorem
implies that there exists functions, $g_1$ and $g_2$, defined on a neighbourhood
of $(u_0,v_0)$, which are contained in $C^\infty$ (recall, $\Phi \in C^\infty (U,\R^2)$)
and which satisfy $\Phi(g_1(u,v),g_2(u,v),u,v) = 0$ for all $(u,v)$ in
this neighbourhood. Equation (\ref{eqdiff3}) thus gives
$(\chi_\LL(u + i v),\eta_\LL(u + i v)) = (g_1(u,v),g_2(u,v))$, for all such
$(u,v)$, and so $(\chi_\LL(\cdot),\eta_\LL(\cdot)) \in C^\infty(\mathbb{H},\LL)$,
as required.

\vspace{0.5cm}

\textbf{Acknowledgements:} This research was carried out at the Royal Institute
of Technology (KTH), Stockholm, and was partially supported by grant
KAW 2010.0063 from the Knut and Alice Wallenberg Foundation. Special thanks to
Kurt Johansson for helpful comments and suggestions.

\end{document}